%% file: gdiag.tex
\documentclass[english,a4paper,11pt,twoside]{article}

\usepackage{amsmath,amsthm,amssymb,amsfonts}
\usepackage{enumerate,mathrsfs}
\usepackage{hyperref}
\usepackage[cmtip,all]{xy}

\input{def}
\title{Homotopy theory of \texorpdfstring{$G$}{G}-diagrams and equivariant excision}
\author{Emanuele Dotto and Kristian Moi} { \footnotetext{The first author was partially supported by the ERC Adv.Grant No.228082. The second author was supported by the Danish National Research Foundation through the Centre for Symmetry and Deformation (DNRF92)}}

\begin{document}
\maketitle
\thispagestyle{empty} 

\abstract{Let $G$ be a finite group acting on a small category $I$. We study functors $X \colon I \to \mathscr{C}$ equipped with families of compatible natural transformations that give a kind of generalized $G$-action on $X$. Such objects are called $G$-diagrams. When $\mathscr{C}$ is a sufficiently nice model category we define a model structure on the category of $G$-diagrams in $\mathscr{C}$. There are natural $G$-actions on Bousfield-Kan style homotopy limits and colimits of $G$-diagrams. We prove that weak equivalences between point-wise (co)fibrant $G$-diagrams induce weak $G$-equivalences on homotopy (co)limits. A case of particular interest is when the indexing category is a cube. We use homotopy limits and colimits over such diagrams to produce loop and suspension spaces with respect to permutation representations of $G$. We go on to develop a theory of enriched equivariant homotopy functors and give an equivariant ``linearity'' condition in terms of cubical $G$-diagrams. In the case of $G$-topological spaces we prove that this condition is equivalent to Blumberg's notion of $G$-linearity. In particular we show that the Wirthm\"{u}ller isomorphism theorem is a direct consequence of the equivariant linearity of the identity functor on $G$-spectra.}

\tableofcontents
\newpage

\section*{Introduction}
  
The concept of $G$-diagram was introduced, under different names, in Villarroel-Flores's thesis \cite{vilthesis} and independently in the paper \cite{js} of Jackowski and S{\l}omi{\'n}ska, and they were further studied in \cite{vilf}. In the current literature the theory of $G$-diagrams has been only partially developed. It is limited, due to the fact that it is used for very specific applications, to properties of homotopy colimits of $G$-diagrams in the category of spaces or of simplicial sets (see e.g. \cite{js} or \cite{thevenaz-webb}). 
The contribution of the present paper is a systematic treatment of $G$-diagrams in a nice (simplicial, cofibrantly generated, etc.) model category. An immediate advantage of this general theory is that it allows us to work in the category of genuine $G$-spectra. Additionally, it is the first treatment of homotopy limits of $G$-diagrams. As an application of this abstract framework, we set up a theory of equivariant enriched homotopy functors and formulate an ``equivariant excision'' condition in terms of cubical $G$-diagrams. This condition agrees with Goodwillie's notion of excision \cite{calcII} when $G$ is the trivial group, and with Blumberg's definition from \cite{Blumberg} for the category of $G$-spaces.

Given a finite group $G$ acting on a category $I$ by functors $a(g) \colon I \to I$, a $G$-diagram in a category $\cat{C}$ is a functor $X \colon I \to \cat{C}$ together with natural transformations $g_X \colon X \to X \circ a(g)$ for every $g$ in $G$, which are compatible with the group structure. A map of $G$-diagrams is a natural transformation between the underlying diagrams that commutes with the structure maps (see Definitions \ref{def:gd} and \ref{def:gmaps}).
We write $\cat{C}^I_a$ for the resulting category of $G$-diagrams. The category $\cat{C}^I_a$ is isomorphic to the category of diagrams in $\cat{C}$ indexed on the Grothendieck construction of the action functor $a\colon G\rightarrow Cat$ (see Lemma \ref{lemma:eqdef} and \cite[2]{js}).
If the category of $G$-objects $\cat{C}^G$ is a sufficiently nice model category, such as $G$-spaces with the fixed points model structure, or orthogonal $G$-spectra with the genuine $G$-stable model structure, we prove the following \ref{modelstruct}.
\begin{theorem*}
   Let $\cat{C}$ be a $G$-model category (see \ref{defGmodelstr}). There is a cofibrantly generated $sSet^G$-enriched model structure on the category of $G$-diagrams $\cat{C}^{I}_a$ with weak equivalences (resp. fibrations) the maps of $G$-diagrams $f\colon X\rightarrow Y$ such that the value $f_i$ at the object $i \in ob I$ is a weak equivalence (resp. fibration) in the model category $\cat{C}^{G_i}$ of objects with an action of the stabilizer group $G_i$.
\end{theorem*}

The authors first became interested in $G$-diagrams while working on equivariant delooping results for so-called Real algebraic $K$-theory and Real topological Hochschild homology. A recurring example of a $G$-diagram in this work is the following:
\begin{example*}
Let $X$ be a pointed space with an action of $C_2$, the cyclic group of order two, with $\sigma \colon X \to X$ representing the action of the non-trivial group element. A diagram of pointed spaces 
\begin{equation}\label{eq:11diag}\xymatrix{Y \ar[r]^-p & X & Z \ar[l]_-q}\end{equation}
together with mutually inverse homeomorphisms $r \colon Y \to Z$ and $l \colon Z \to Y$ which cover $\sigma$, in the sense that $p \circ l = \sigma \circ q$ and $q \circ r = \sigma \circ p$, defines a $C_2$-diagram of pointed spaces. The pullback $Y \times_X Z$ inherits a natural $C_2$-action given by $(y,z) \mapsto (l(z),r(y))$, and similarly the homotopy pullback 
\[Y \times^h_X Z = \{(y,\gamma,z) \in Y \times X^I \times Z \, | \, p(y) = \gamma(0) {\text{ and }} \gamma(1) = q(z) \}\]
inherits the action $(y,\gamma,z) \mapsto (l(z), \sigma \circ \bar{\gamma}, r(y))$, where $\bar{\gamma}(t) = \gamma(1-t)$. The usual inclusion $Y \times_X Z \inj Y \times^h_X Z$ is equivariant with respect to these actions. Let $\R^{1,1}$ denote the sign representation of $C_2$ on $\R$ and let $\Omega^{1,1}X$ be the space of pointed maps from the one point compactification $S^{\R^{1,1}}$ to $X$ with $C_2$ acting by conjugation. If $Y$ (and hence $Z$) is contractible, then a contracting homotopy induces a $C_2$-homotopy equivalence
\[Y \times^h_X Z \simeq \Omega^{1,1}X.\]
On underlying spaces this just an instance of the well-known homotopy equivalence
\[\Omega X \simeq \holim (\ast \to X \leftarrow \ast).\]
\end{example*}

This example illustrates how limits and homotopy limits of punctured $C_2$-squares of spaces carry a $C_2$-action, and how these can be used to construct the loop space by the sign representation of $C_2$. More generally, when it makes sense to talk about the limit, colimit, homotopy limit or homotopy colimit of a $G$-diagram $X$ in any ambient category $\cat{C}$, these constructions have natural $G$-actions induced by the structure maps $g_X$ (see Corollary \ref{cor:lim-action} and \S\ref{sec:enr}). Moreover, the usual comparison maps $\lim X \to \holim X$ and $\hocolim X \to \colim X$ are equivariant as we already observed for the $C_2$-diagram (\ref{eq:11diag}). In general most constructions involving (co)limits and (co)ends enrichments applied to $G$-diagrams produce $G$-objects and equivariant maps between them. The homotopy limits and colimits of $G$-diagrams are homotopy invariant in the following sense (see also Proposition \ref{htpyinvlimcolim}):
\begin{proposition*} The functors $\holim\colon \cat{C}^{I}_a\rightarrow \cat{C}^{G}$ and $\hocolim\colon \cat{C}^{I}_a\rightarrow \cat{C}^{G}$ preserve equivalences between fibrant diagrams and point-wise cofibrant diagrams respectively. 
\end{proposition*}

We prove other fundamental properties of these equivariant homotopy limits and colimits functors, analogous to classical theorems from homotopy theory of diagrams:
\textit{\begin{itemize}
\item \ref{Gcof} Homotopy cofinality theorem for homotopy limits and colimits of $G$-diagrams, generalizing the results \cite[1]{thevenaz-webb} and \cite[§6]{vilf},
\item \ref{Fubini} A twisted Fubini theorem, showing that homotopy colimits of $G$-diagrams over a Grothendieck construction can be calculated ``point-wise'' (an equivariant analogue of \cite[26.5]{CS}). As an immediate corollary we obtain  an equivariant analogue of Thomason's homotopy colimit theorem from \cite{thomason},
\item \ref{elmendorf} An Elmendorf theorem, showing that for suitable ambient categories one can equivalently define the homotopy theory of $G$-diagrams by replacing $G$ with the opposite of its orbit category (an equivariant analogue of the classical result of \cite{Elmendorf}).
\end{itemize}}

\vspace{.5cm}

As an application of this model categorical theory of $G$-diagrams, we define and study equivariant excision. Classically, a homotopy invariant functor between model categories is excisive if it sends homotopy cocartesian squares to homotopy cartesian squares (see \cite{calcII}). Blumberg shows in \cite{Blumberg} that this notion is not well behaved when the categories involved are categories of $G$-objects; enriched homotopy functors on the category of pointed $G$-spaces $Top_{\ast}^G\rightarrow Top_{\ast}^G$ that are classically linear (excisive and sending the point to a $G$-contractible space) are a model only for the category of na\"{i}ve $G$-spectra. In order to model genuine $G$-spectra,  one needs a property stronger than classical  linearity. Blumberg achieves this by adding an extra condition to linearity; a compatibility condition with equivariant Spanier-Whitehead duality.

In the present paper we take a different approach to equivariant excision, following the idea that the relation between equivariant excision and excision should resemble the relation between genuine $G$-spectra and na\"{i}ve $G$-spectra.
Instead of adding an extra condition to classical excision, we replace squares by ``equivariant cubes'', similarly to the way one replaces integers with $G$-representations in defining $G$-spectra.
For a finite $G$-set $J$ we consider the poset category $\mathcal{P}(J)$ of subsets of $J$ ordered by inclusion. This category inherits a $G$-action from the $G$-action on $J$.

\begin{definition*}[$G$-excision] A $J$-cube $X$ in $\cat{C}$ is a $G$-diagram in $\cat{C}$ shaped over $\mathcal{P}(J)$, i.e. it is an object of $\cat{C}^{\mathcal{P}(J)}_a$. We say that $X$ is homotopy cartesian if the canonical map
\[X_{\emptyset}\longrightarrow \holim_{\mathcal{P}(J)\backslash\emptyset}X\]
is a weak equivalence in the model category of $G$-objects $\cat{C}^G$. Dually, it is homotopy cocartesian if the canonical map $\displaystyle\hocolim_{\mathcal{P}(J)\backslash J}X\rightarrow X_J$ is an equivalence in $\cat{C}^G$.
A suitably homotopy invariant functor $\Phi\colon \cat{C}^G\rightarrow \cat{D}^G$ is called $G$-excisive if it sends homotopy cocartesian $G_+$-cubes to homotopy cartesian $G_+$-cubes.
\end{definition*}

 Here $G_+$ is the set $G$ with an added disjoint base point, and $G$ acts on it by left multiplication. It plays the role of a ``regular'' $G$-set, analogous to the regular representation of $G$ in stable equivariant homotopy theory. The added basepoint has an important role, discussed in details in \ref{nobasepoint}.
 We prove in \ref{Glintop} that this notion of $G$-excision is equivalent to Blumberg's definition from \cite{Blumberg} when $\cat{C}$ is the category of pointed spaces.
The paper contains a series of fundamental properties of $G$-excision, that appropriately reflect the fundamental properties of excision to a genuine equivariant context. They can be summarized as follows: 
\textit{\begin{itemize}
\item \ref{squaresinG} A $G$-excisive functor $\cat{C}^G\rightarrow \cat{D}^G$ is classically excisive, that is, it sends homotopy cocartesian squares in  $\cat{C}^G$ to  homotopy cartesian squares in $\cat{D}^G$,
\item\ref{GlinHlin} A $G$-linear functor is also $H$-linear for every subgroup $H$ of $G$,
\item\ref{classlintospec} Every enriched $G$-linear homotopy functor $\Phi$ from finite $G$-CW-complexes to $G$-spectra is equivalent to one of the form $E_{\Phi}\wedge{(-)}$ for some $G$-spectrum $E_\Phi$, 
\item\ref{Gcartcocart} The identity functor on $G$-spectra is $G$-excisive: For any finite $G$-set $J$, a $J$-cube of spectra is homotopy cartesian if and only if it is homotopy cocartesian,
\item\ref{linandwedges} Any $G$-excisive reduced homotopy functor $\Phi \colon \cat{C}^G\rightarrow \cat{D}^G$ satisfies the Wirthm\"{u}ller isomorphism theorem, that is, the canonical map $\Phi(G\otimes_Hc)\rightarrow \hom_H(G,\Phi(c))$ is an equivalence in $\cat{D}^G$ for every subgroup $H$ of $G$ and $H$-object $c$ of $\cat{C}^H$.
\item\ref{corlinandloops},\ref{Gderivative} If $\cat{D}^G$ is suitably presentable, a construction similar to Goodwillie's derivative of \cite{calcII} defines a universal $G$-excisive approximation to any homotopy functor $\cat{C}^G\rightarrow \cat{D}^G$.
\end{itemize}}

These properties have interesting consequences for the identity functor on $G$-spectra. The fact that it is $G$-excisive shows that the theory of equivariant cubes provides a good context in which the category of $G$-spectra is ``$G$-stable''. Moreover, Theorem \ref{linandwedges} applied to the identity functor on $G$-spectra gives a new proof of the classical Wirthm\"{u}ller isomorphism theorem. An analysis of the structure of the proofs of \ref{linandwedges} and \ref{Gcartcocart} gives the following argument: The identity on $G$-spectra is $G$-excisive as a direct consequence of the  equivariant Freudenthal suspension theorem, by formally manipulating homotopy limits and colimits. Given an $H$-equivariant spectrum $E$, there is an explicit homotopy cocartesian $(G/H)_+$-cube of spectra $WE$ with initial vertex $(WE)_\emptyset=G_+\wedge_HE$, and with $\displaystyle\holim_{\mathcal{P}(G/H_+)\backslash\emptyset}WE=F_H(G_+,E)$. By $G$-excision for the identity functor $WE$ is homotopy cartesian, that is, the canonical map $G_+\wedge_HE\rightarrow F_H(G_+,E)$ is a stable equivalence of $G$-spectra.

\subsection*{Acknowledgments}
We wish to thank Irakli Patchkoria for help with proofreading and for pointing out a mistake in an earlier version of this article. We would also like to thank Ib Madsen for his enduring support and for encouraging us to write this paper.

\section{Definitions and setup}\label{sec:def}

\subsection{Categories of \texorpdfstring{$G$}{G}-diagrams}
We first introduce some notation and conventions. If $\cat{C}$ is a (possibly large) category and $I$ is a small category we write $\cat{C}^I$ for the usual category of functors from $I$ to $\cat{C}$. By topological space we will mean compactly generated weak Hausdorff space and $Top$ is the category of such spaces with continuous maps between them. We write $Map(X,Y)$ for the space of maps from $X$ to $Y$ endowed with the compact-open topology. The based variants of the above are $Top_\ast$ and $Map_\ast(X,Y)$.

In the following $\cat{C}$ will be a category, $G$ a finite group, and $I$ a small category. By a slight abuse of notations we will also write $G$ for the category with one object $\ast$ and one morphism $g \colon \ast \to \ast$ for each element $g \in G$, and with composition given by $g \circ h = gh$. The group $G$ will act on $I$ from the left and we will encode the action as a functor $a \colon G \to Cat$ sending $\ast$ to $I$. Most of the content of this section can be found in the work of Jackowski-S{\l}omi{\'n}ska\cite{js} or in Villarroel-Flores's paper \cite{vilf}.  

\begin{definition}(cf. \cite[2.2]{js}, \cite[3.1]{vilf})\label{def:gd}Let $X \colon I \to \cat{C}$ be an $I$-shaped diagram in $\cat{C}$. A $G$-structure on $X$ with respect to the action $a$ is a collection of natural transformations $\{g_X \colon X \to X \circ a(g) \}$ such that
  \begin{enumerate}
  \item \label{def:gd-1}$e_X = id_X$
  \item \label{def:gd-2}$(g_X)_{a(h)} \circ h_X = (gh)_X$ for all $g,h \in G$,
  \end{enumerate}
where $(g_X)_{a(h)}$ is the natural transformation obtained by restricting $g_X$ along the functor $a(h) \colon I \to I$. An $I$-shaped diagram $X$ with a $G$-structure will be called an $I$-shaped $G$-diagram in $\cat{C}$ with respect to the action $a$, or simply a $G$-diagram in $\cat{C}$ if $I$ and $a$ are understood.
\end{definition}

In order to simplify the notation we will mostly write $g$ in stead of $a(g)$ when this does not cause confusion. Accordingly, when $X$ and $Y$ are $I$-indexed $G$-diagrams we will write $f_g$ for the restriction of a map $f \colon X \to Y$ along the functor $g = a(g) \colon I \to I$. In the later sections we will sometimes write $g$ instead of $g_X$.

\begin{definition}\label{def:gmaps}
  A map of $G$-diagrams $f \colon X \to Y$ is a natural transformation $f \colon X \to Y$ of underlying diagrams such that for each $g \in G$ the diagram
\[\xymatrix{X \ar[r]^f \ar[d]_-{g_X} &  Y \ar[d]^-{g_Y}\\
X \circ g \ar[r]_{f_g} & Y \circ g}
\]
commutes in $\cat{C}^I$. 
\end{definition}

The composite of two maps of $G$-diagrams is again a map of $G$-diagrams. For a fixed action $a$ of the group $G$ on $I$ we write $\cat{C}^I_a$ for the category whose objects are the $G$-diagrams in $\cat{C}$ with respect to $a$ and with morphisms the maps of $G$-diagrams. 

\begin{example}\label{Nover}
Let $[n]$ be the usual category with objects $0,1, \ldots, n$ and a morphism $i \to j$ if and only if $i \leq j$. For a small category $I$ the nerve $NI$ is the usual simplicial set with $NI_n = Fun([n],I)$. Taking over-categories gives a functor $N(I/-) \colon I \to sSet$. The $G$-action on $I$ gives maps $N_{/i,g} \colon N(I/i) \to N(I/gi)$ for $g \in G$ and $i$ an object of $I$, by mapping
\[(i_0\rightarrow\dots\rightarrow i_n\rightarrow i)\stackrel{g}{\longmapsto}(gi_0\rightarrow\dots\rightarrow gi_n\rightarrow gi)\]
 These maps combine to give a $G$-diagram structure on $N(I/-)$. Similarly the functor $N(-/I)^{op} \colon I^{op} \to sSet$ with the maps $N_{i,g/} \colon N(i/I)^{op} \to N(gi/I)^{op}$ defines a $G$-diagram in $sSet$.
\end{example}

Let $I$ and $J$ be small categories with $G$-actions $a$ and $b$ respectively and let $F \colon I \to J$ be a functor. We say that $F$ is $G$-equivariant if it commutes strictly with the $G$-actions, that is, if $F(gi) = gF(i)$ and $F(g\alpha) = gF(\alpha)$ for all objects $i$ in $I$ and morphisms $\alpha$ in $I$. If $Y$ is a $J$-shaped $G$-diagram then the restriction $F^\ast Y = Y \circ F$ has a naturally induced $G$-structure with maps $g_{(F^\ast Y)} =  F^\ast(g_Y)$.

Now assume that $\cat{C}$ is complete and cocomplete. Then the functor $F^\ast \colon \cat{C}^J \to \cat{C}^I$ has a left adjoint $F_\ast$ and a right adjoint $F_!$ given by left and right Kan extension, respectively. We will now see that if $X$ is an $I$-shaped $G$-diagram, then there are natural $G$-structures on $F_\ast X$ and $F_! X$. We treat the left Kan extension first.

The value of the functor $ F_\ast X$ on an object $j$ of $J$ is given by the coequalizer
\begin{equation*} \xymatrix{\displaystyle \coprod_{(i_0 \stackrel{\alpha}{\to} i_1, f \colon F(i_1) \to j)} X_{i_0} \ar @< 4pt> [r]^-s
\ar @<-4pt> [r]_-t & \displaystyle \coprod_{(i_0 , f \colon F(i_0) \to j)} X_{i_0}   \ar@{>>}[r] & F_\ast X _j,}
\end{equation*}
where $s$ projects onto the source of the indexing map $\alpha$ and $t$ maps into the target of $\alpha$ by the map $X(\alpha)$. For an element $g \in G$ the natural transformation $g_X$ induces a map of diagrams

\begin{equation*} \xymatrix{\displaystyle \coprod_{(i_0 \stackrel{\alpha}{\to} i_1, f \colon F(i_1) \to j)} X_{i_0} \ar @< 4pt> [r]^-s
\ar @<-4pt> [r]_-t \ar[d]^-{\coprod g_{X_{i_0}}} & \displaystyle \coprod_{(i_0 , f \colon F(i_0) \to j)} X_{i_0}   \ar@{>>}[r]\ \ar[d]^-{\coprod g_{X_{i_0}}}  & F_\ast X _j \ar @{-->} [d]^{g_{F_\ast X_j}} \\
\displaystyle \coprod_{(i_0' \stackrel{\alpha'}{\to} i_1', f' \colon F(i_1') \to gj)} X_{i_0'} \ar @< 4pt> [r]^-s
\ar @<-4pt> [r]_-t & \displaystyle \coprod_{(i_0' , f' \colon F(i_0') \to gj)} X_{i_0'}   \ar@{>>}[r] & F_\ast X _{gj}}
\end{equation*}
and the dotted arrow is the $j$-component of the natural transformation $g_{F_\ast X} \colon F_\ast X \to (F_\ast X)\circ g$. It is not hard to see that the set $\{ g_{F_\ast X} \}_{g \in G}$ constitutes a $G$-structure on $F_\ast X$ and that the underlying functor $F_\ast$ takes maps of $I$-indexed $G$-diagrams to maps of $J$-indexed $G$-diagrams. Similarly, for the right Kan extension $F_!$ a dual construction with equalizers gives a $G$-structure $\{ g_{F_! X} \}_{g \in G}$ on $F_! X$. We write simply $F_\ast X$ and $F_!X$ for the $G$-diagrams obtained in this way.

\begin{proposition}The constructions $F_\ast X$ and $F_!X$ define functors $F_\ast \colon \cat{C}^I_a \to \cat{C}^J_b$ and $ F_! \colon \cat{C}^I_a \to \cat{C}^J_b$.
\end{proposition}

A particularly interesting case of the above is when $J = \ast$ the category with one object and one morphism and trivial $G$-action. In this case the functors $F_\ast$ and $F_!$ are more commonly known as $\colim_I$ and $\lim_I$, respectively.

\begin{corollary}\label{cor:lim-action}
  Let $X$ be an $I$-indexed $G$-diagram. Then the above constructions induce natural left $G$-actions on $\colim_I X$ and $\lim_I X$.
\end{corollary}

\begin{example}\label{ex:prod}(Products and coproducts)
  Let $I$ be a discrete category with $G$-action, i.e., a $G$-set and consider a $G$-diagram $X$ in the category $Set$ of sets. The coproduct $\coprod_I X$ is the set of pairs $(i,x)$ where $x \in X_i$ and the action of $g \in G$ is given by 
\[g(x,i)=(g_{X_i}(x),gi).\]

The product $\prod_I X$ is the set of functions $\mathbf{x} \colon I \to \bigcup_{i \in I} X_i$ such that $\mathbf{x}(i) \in X_i$ for all $i \in I$. The action of $g \in G$ on $\mathbf{x} \in \prod_I X$ is determined by the equation
\[(g \mathbf{x})(gi) = g_{X_i}(\mathbf{x}(i)).\]
This example generalizes to arbitrary categories with products and coproducts but the notation becomes more cumbersome when one can no longer speak about elements of objects.
\end{example}

We now give an alternative description of $G$-diagrams which is sometimes easier to work with.

\begin{definition}
  Let $G \rtimes_a I$ be the following category:
  \begin{itemize}
  \item $ob G \rtimes_a I = ob I$
  \item A morphism $i \to j$ in $G \rtimes_a I$ is a pair $(g, \alpha \colon gi \to j)$ where $g \in G$.
  \item Composition is given by $(h, \beta \colon hj \to k) \circ (g, \alpha \colon gi \to j) = (gh, \beta \circ h\alpha \colon ghi \to k)$.

  \end{itemize}
\end{definition}

\begin{remark}
The category $G \rtimes_a I$ is the Grothendieck construction of the functor $a\colon G\rightarrow Cat$, sometimes denoted $G \int a$ (see e.g. \cite{thomason}). 
\end{remark}

A $G$-diagram $X$ gives rise to a functor $X^{\rtimes_a} \colon G \rtimes_a I \to \cat{C}$ by setting 
\[X^{\rtimes_a}_i = X_i\]
on objects, and defining
\[X^{\rtimes_a} (g, \alpha \colon gi \to j) = X(\alpha) \circ g_{X_i}\]
on morphisms. We leave it to the reader to check that this respects composition of maps.

\begin{lemma}\label{lemma:eqdef}
  The assignment $X \mapsto X^{\rtimes_a}$ is functorial and defines an isomorphism of categories 
\[\Phi \colon \cat{C}^I_a \iso \cat{C}^{G \rtimes_a I}.\]
\end{lemma}
\begin{proof}
  The functoriality is clear. We define a functor $\Phi' \colon \cat{C}^{G \rtimes_a I} \to \cat{C}^I_a$ which is inverse to $\Phi$.
For a diagram $Y \colon G \rtimes_a I \to \cat{C}$ define the underlying diagram of $\Phi'(Y)$ to be $(Y|_I)$, i.e., the restriction of $Y$ along the canonical inclusion $\iota \colon I \inj G \rtimes_a I$ given by $\iota(i) =i$ and $\iota(\alpha \colon i \to j) = (e,\alpha \colon i \to j)$. For an element $g \in G$ the natural transformation $g_{\Phi'(Y)}$ is defined at an object $i$ by $Y(g,id \colon {gi} \to gi)$. Both naturality of the $g_{\Phi'(Y)}$'s and conditions \ref{def:gd-1}) and \ref{def:gd-2}) of Definition \ref{def:gd} follow from the functoriality of $Y$ with respect to morphisms in $G \rtimes_a I$. For a natural transformation $f\colon Y \to Z$ in $\cat{C}^{G \rtimes_a I}$ we define $\Phi'(f) = f|_I$. It is now easy to check that the functors $\Phi$ and $\Phi'$ are mutually inverse. 
\end{proof}

\begin{corollary}
  Let $\cat{C}$ be a bicomplete category. Then $\cat{C}_a^I$ is also bicomplete.
\end{corollary}
\begin{proof}
  The diagram category $\cat{C}^{G \rtimes_a I}$ is bicomplete since $\cat{C}$ is. It follows from \ref{lemma:eqdef} that $\cat{C}_a^I$ is bicomplete.
\end{proof}

\subsection{Enrichments and homotopy (co)limits}\label{sec:enr}

If $\cat{C}$ is any category, then the category $\cat{C}^G$ is naturally enriched in left $G$-sets in the following way. For objects $c,d$ of $\cat{C}^G$ let $\cat{C}(c,d)$ be the set of maps between the underlying objects in $\cat{C}$. Then $G$ acts on $\cat{C}(c,d)$ by conjugation 
\[g \cdot f = g_d \circ f \circ (g^{-1})_c\]
where $(g^{-1})_c$ and $g_d$ represent the actions of $g^{-1}$ and $g$ on $c$ and $d$ respectively. The fixed points set $\cat{C}(c,d)^G$ is precisely the set of $G$-equivariant maps from $c$ to $d$. 

If $I$ is small a category with an action $a$ of $G$, then the category $\cat{C}^I_a$ of $G$-diagrams becomes enriched in left $G$-sets by taking $\underline{\cat{C}}_a^I(X, Y)$ to be the set $\cat{C}^I(X,Y)$ of maps of underlying diagrams $f \colon X \to Y$ with action given by
\[ g \cdot f = (g_Y)_{g^{-1}} \circ f_{g^{-1}} \circ (g^{-1})_X.\]
If $f$ is fixed under the action of $G$, then
\[f = g^{-1} f=  ((g^{-1})_Y)_g \circ f_{g} \circ g_X  = (g_Y)^{-1} \circ f_{g} \circ g_X.\]
In other words, $f$ is fixed if and only if the square \[\xymatrix{X \ar[r]^f \ar[d]_-{g_X} &  Y \ar[d]^-{g_Y}\\
X \circ g \ar[r]_{f_g} & Y \circ g}
\]
commutes for all $g \in G$. It follows that the fixed points $\underline{\cat{C}}^I_a(X, Y)^G$ are precisely the maps of $G$-diagrams $\cat{C}^I_a(X, Y)$. If $I = \ast$ then this statement reduces to the one above about maps in $\cat{C}^G$.

\begin{proposition}\label{prop:enrichment}
  Let $I$ and $J$ be small categories with $G$-actions $a$ and $b$, respectively and let $F \colon I \to J$ be an equivariant functor. Then, for $X$ an $I$-indexed $G$-diagram and $Y$ a $J$-indexed $G$-diagram the bijections
\[\phi_{X,Y} \colon \underline{\cat{C}}_a^I(X,F^\ast Y) \iso \underline{\cat{C}}^J_b(F_\ast X, Y )\]
and 
\[\psi_{X,Y} \colon \underline{\cat{C}}^I_a(F^\ast Y, X  ) \iso \underline{\cat{C}}_b^J(Y, F_!X) \]
induced by the adjunctions on underlying diagrams are $G$-equivariant.
\end{proposition}

\begin{proof}
We show that $\phi = \phi_{X,Y}$ is equivariant, the argument for $\psi_{X,Y}$ is similar.

Let $f \colon X \to F^\ast Y$ be a map of diagrams and $g \in G$. Then $\phi(g \cdot f)$ is the unique map $F_\ast X  \to Y$ such that the diagram
\begin{equation}\label{eq:alpha}\xymatrix{ & X \ar[d]_-{\eta_X} \ar[r]^-{g \cdot f} & F^\ast Y\\
&F^\ast F_\ast X \ar@{-->}[ur]_-{F^\ast(\phi(g \cdot f))} & }
\end{equation}
commutes, where $\eta_X$ is the unit of the $(F_\ast,F^\ast)$-adjunction at the object $X$. Consider the following diagram:
\[\xymatrix{X \ar[rr]^{(g^{-1})_X} \ar[d]_-{\eta_X} && X \circ g^{-1} \ar[rr]^{f_{g^{-1}}} \ar[d]_-{\eta_{X,g^{-1}}} && (F^\ast Y) \circ g^{-1} \ar[rr]^-{(F^\ast g_Y)_{g^{ -1}}} \ar[d]^= && F^\ast Y \\
F^\ast F_\ast X \ar[rr]_-{F^\ast((g^{-1})_{F_\ast X})} && (F^\ast F_\ast X) \circ g^{-1} \ar[rr]_-{F^\ast \phi(f)_{g^{-1}}} &&  (F^\ast Y) \circ g^{-1} \ar[urr]_-{(F^\ast g_Y)_{g^{ -1}}}. && }
\]
The commutativity of the left hand square follows immediately from the definition of $g_{F_\ast X} $ and middle square commutes by the definition of $\phi(f)$. Composing the maps in the top row gives $(F^\ast g_Y)_{g^{-1}} \circ f_{g^{-1}} \circ (g^{-1})_X = g \cdot f$ and composing along the bottom row from $F^\ast F_\ast X$ to $F^\ast Y$ gives
\[F^\ast ((g_Y)_{g^{-1}} \circ \phi(f)_{g^{-1}} \circ (g^{-1})_{F_\ast X}) = F^\ast(g \cdot \phi(f)). \]
It follows that $F^\ast(g \cdot \phi(f))$ defines a lift in the diagram (\ref{eq:alpha}) so, by uniqueness of the lift, we conclude that $\phi(g \cdot f) = g \cdot \phi(f)$.
\end{proof}

Taking fixed points in Proposition \ref{prop:enrichment} we immediately get the following:

\begin{corollary}\label{cor:adj}
  The functors $F_\ast$ and $F_!$ are left and right adjoint, respectively, to the restriction functor $F^\ast \colon \cat{C}_b^J \to \cat{C}^I_a$. In particular the diagonal $\Delta_I=p^\ast \colon \cat{C}^G \to \cat{C}^I_a$ induced by the projection $p\colon I\rightarrow \ast$ has left and right adjoints $p_\ast=\colim_I$ and $p_!=\lim_I$, respectively. 
\end{corollary}

Let $I$ be a category with $G$-action $a$ and let $G$ act diagonally on the product $I^{op}\times I$. Given a $G$-diagram $Z \colon I^{op}\times I \to \cat{C}$ recall that the \emph{end} $\int_i Z_{i,i}$ of $Z$ is the equalizer
\[\xymatrix{\displaystyle{\int_i} Z_{i,i} \ar@{>->}[r] & \displaystyle{\prod_i} Z_{i,i}   \ar @< 4pt> [r]^-s \ar @<-4pt> [r]_-t  & \displaystyle{\prod_{ \alpha \colon i \to j}} Z_{j,i}} \]
where $s$ and $t$ act on the left and right, respectively by the map $\alpha$. The end $\int_i Z_{i,i}$ inherits a left $G$-action by the maps
\[\label{eq:endg}\tag{$\ast$}\xymatrix{ \displaystyle{\int_i} Z_{i,i} \ar[r] \ar@{-->}[d]^{g_{(\int Z)}} & \displaystyle{\prod_i} Z_{i,i}   \ar @< 4pt> [r]^-s \ar @<-4pt> [r]_-t  \ar[d]^{\prod_i g_{Z_{(i,i)}}} & \displaystyle{\prod_{ \alpha \colon i \to j}} Z_{j,i} \ar[d]^{\prod_\alpha g_{Z_{(j,i)}}}\\
\displaystyle{\int_i} Z_{i,i} \ar[r]  &\displaystyle{\prod_i} Z_{i,i}   \ar @< 4pt> [r]^-s \ar @<-4pt> [r]_-t   & \displaystyle{\prod_{ \alpha \colon i \to j}} Z_{j,i}}
\]
The \emph{coend} $\int^i Z_{i,i}$ is the coequalizer
\[\xymatrix{\displaystyle{\coprod_{ \alpha \colon i \to j}} Z_{j,i} \ar @< 4pt> [r]^-s \ar @<-4pt> [r]_-t & \displaystyle{\coprod_i} Z_{i,i} \ar@{->>}[r]& \displaystyle{\int_i} Z_{i,i}.}\]
which inherits a $G$-action in a similar way. 

\begin{example}
  If $X,Y \colon I \to \cat{C}$ are diagrams in $\cat{C}$ then we can describe the set of maps (natural transformations) between them as the end
\[\cat{C}^I(X,Y) = \int_i \cat{C}(X_i,Y_i).\]
Similarly, for $G$-diagrams $X,Y$ in $\cat{C}^I_a$ there is a natural isomorphism of $G$-sets
\[\underline{\cat{C}}^I_a(X,Y) \cong \int_i\cat{C}(X_i,Y_i)\]
with the $G$-action on the left hand as described above.
\end{example}

By a \emph{simplicial category} we will mean a category $\cat{C}$ that is enriched, tensored and cotensored in simplicial sets, in the sense of e.g. \cite[2.2]{enmodcat} or \cite[II,2.1]{GJ}. This means that for any two objects $c$ and $d$ in $\cat{C}$ there is a simplicial set $Map_\cat{C}(c,d)$, and a natural bijection $\cat{C}(c,d) \cong Map_\cat{C}(c,d)_0$. Moreover, given a simplicial set $K$ there are objects $K \otimes c$ and $map_\cat{C}(K,c)$ of $\cat{C}$. These satisfy some associativity constraints and naturality conditions making $Map_\cat{C}(-,-)$ and $map_\cat{C}(-,-)$ contravariant functors in the first variable and covariant in the second variable and $-\otimes-$ covariant in both variables. Finally for all $c,d$ in $\cat{C}$ and $K$ in $sSet$ there are natural isomorphisms in $sSet$
\[Map_{\cat{C}}(K \otimes c,d) \cong Map(K,Map_{\cat{C}}(c,d)) \cong Map_{\cat{C}}(c,map_{\cat{C}}(K,d)),\]
where $Map$ with no subscript denotes the usual internal hom-object in $sSet$.

Using this structure we will now describe additional structure on the category $\cat{C}^I_a$ of $I$-indexed $G$-diagrams in a simplicial category $\cat{C}$. We begin with enrichment. We noted above that for a pair $X,Y$ of $G$-diagrams in $\cat{C}$ the set $\cat{C}^I(X,Y)$ has a $G$-action induced by the $G$-structures on $X$ and $Y$. This gives $\cat{C}^I_a$ the structure of a category enriched in left $G$-sets. The functor $i,j \mapsto Map_\cat{C}(X_i,Y_j)$ going from $I^{op}\times I$ to $sSet$ becomes a $G$-diagram by letting $g \in G$ act at $i,j$ by 
\[Map_\cat{C}(g_{X_i}^{-1}, g_{Y_j}) \colon   Map_\cat{C}(X_i,Y_j) \to  Map_\cat{C}(X_{gi},Y_{gj}).\]
\begin{definition}With $X,Y$ as above, set
\[Map_{\cat{C}^I_a}(X,Y) = \int_i Map_\cat{C}(X_i,Y_i)\]
with the $G$-action as described in the diagram (\ref{eq:endg}).
\end{definition}

In other words the mapping space $Map_{\cat{C}_a^I}(X,Y)$ is the equalizer
\[\xymatrix{Map_{\cat{C}_a^I}(X,Y) \ar@{>->}[r] & \displaystyle{\prod_i} Map_\cat{C}(X_i,Y_i)   \ar @< 4pt> [r]^-s \ar @<-4pt> [r]_-t  & \displaystyle{\prod_{ \alpha \colon i \to j}} Map_\cat{C}(X_j,Y_i) } \]

 It is not hard to see that this defines an enrichment of $\cat{C}^I_a$ in $sSet^G$ and that for each $n \geq 0$ there is an isomorphism of $G$-sets
\[Map_{\cat{C}^I_a}(X,Y)_n \cong \underline{\cat{C}}^I_a(\Delta^n \otimes X,Y).\]

\begin{definition}Let $K \colon I \to sSet$, $L \colon I^{op} \to sSet$, and $X \colon I \to \cat{C}$ be $G$-diagrams. We set
  \begin{align}
    map_I^a(K,X) &= \int_i map_\cat{C}(K_i ,X_i)\\
L \otimes^a_I X  &= \int^i L_i \otimes X_i
  \end{align}
and give both the $G$-actions from (\ref{eq:endg}).
\end{definition}

When $K$ and $L$ are respectively the $G$-diagrams of simplicial sets $N(I/-)$ and $N(-/I)^{op}$ from \ref{Nover}, these constructions specify to the following.

\begin{definition}For a $G$-diagram $X$ in $\cat{C}$ the homotopy limit and homotopy colimit of $X$ are respectively 
\[\holim_I X = map^a_I(N(I/-),X)\ \ \ \ \ \ \ \ \ \ \ \ \hocolim_I X = N(-/I)^{op} \otimes_I^a X\]
This constructions define functors $\holim,\hocolim\colon \cat{C}^{I}_a\rightarrow\cat{C}^G$.
\end{definition}
In the presence of a model structure the words homotopy limit and colimit will always refer to these particular construction and not, a priori, the derived functors of the limit and colimit respectively. 

Note that there are maps of diagrams $N(-/I)^{op} \to \ast$ and $N(I/-) \to \ast$, where $\ast$ denotes a chosen one-point simplicial set in both cases. From the formulas above it is easy to see that there are natural isomorphisms $map^a_I(\ast,X) \cong \lim X$ and $X \otimes^a_I \ast \cong \colim X$. The maps to the terminal diagrams induce equivariant maps
\[\lim X \to \holim X \ \ \ \ \ \ \ \ \ \hocolim X \to \colim X\]
This paper is in part motivated by the question ``when are these maps weak equivalences in $\cat{C}^G$?''
 
\subsection{Examples of \texorpdfstring{$G$}{G}-diagrams}
In this section we will provide many of the motivating examples for the theory of $G$-diagrams. The diagrams will usually have values in the category $Top_\ast$ of pointed spaces.

For the first two examples we need to fix some notation. Let $Z$ be a pointed space with an action by the finite group $G$. If $T$ is a finite left $G$-set, we write $\R[T]$ for the permutation representation with basis $\{e_t\}_{t \in T}$. The subspace of $\R[T]$ generated by the element $N_T = \sum_{t \in T} e_t$ is a one-dimensional trivial subrepresentation of $\R[T]$. We define $S^{\tilde{T}}$ to be the one-point compactification of the orthogonal complement of $\R \cdot N_T$ under the usual inner product. We write $\Omega^{\tilde{T}}Z$ for the $G$-space of continuous pointed maps $Map_\ast(S^{\tilde{T}}, Z)$ with the conjugation action of $G$ and $\Sigma^{\tilde{T}}Z$ for the smash product $S^{\tilde{T}} \wedge Z$ with the diagonal $G$-action.

\begin{example}The power set $\mathcal{P}(T)$ inherits a left $G$-action from the action on $T$. We think of the poset $\mathcal{P}(T) \setminus \emptyset$ as a category with $G$-action. Let $\omega^{\tilde{T}}Z$ be the $\mathcal{P}(T) \setminus \emptyset$-indexed $G$-diagram whose value on a subset $U \subseteq T$ is $\ast$ if $U \neq T$ and $Z$ if $U=T$. The $G$-structure on $\omega^{\tilde{T}}Z$ is given by the action of $G$ on $Z$ at the fixed object $T$ and by the unique maps $\ast \to \ast$ elsewhere in the diagram. We claim that there is a $G$-homeomorphism
\[\holim_{\mathcal{P}(T)\backslash\emptyset} \omega^{\tilde{T}}Z \cong \Omega^{\tilde{T}} Z\]
which is natural in $Z$. To see this we begin by noticing that the realization of the category $|N(\mathcal{P}(T) \setminus \emptyset)$ is $G$-homeomorphic to the (barycentric subdivision of the) standard simplex $\Delta^{|T|-1}$ in the complement of  $\R \cdot N_T$ in $\R[T]$. Since $\omega^{\tilde{T}}Z$ has all entries trivial except at the last vertex $T$ we see that $\holim \omega^{\tilde{T}}Z$ is homeomorphic to the subspace in $Map(\Delta^{|T|-1}, Z)$ of maps whose restriction to the boundary is the constant map to the base-point of $Z$, that is $\Omega^{\tilde{T}} Z$. The naturality is clear, so this proves the claim.

\end{example}

\begin{example}Similarly, we think of the poset $\mathcal{P}(T) \setminus T$ as a category with $G$-action and define the $G$-diagram $\sigma^{\tilde{T}}Z$ to have the value $Z$ at the vertex $\emptyset$ and $\ast$ elsewhere. The $G$-diagram structure is induced by the $G$-action on $Z$ and the unique maps $\ast \to \ast$. A similar argument to the one for $\omega^{\tilde{T}}Z$ shows that there is a natural $G$-homeomorphism
\[\hocolim_{\mathcal{P}(T)\backslash T} \sigma^{\tilde{T}}Z \cong \Sigma^{\tilde{T}}Z.\]
\end{example}

 \begin{example}\label{defloopsusp}
More generally, for any pointed category $\cat{C}$ and $G$-object $c\in\cat{C}^G$ define the $\tilde{T}$-loop space and $\tilde{T}$-suspension of $c$ respectively as the pullback and pushout in $\cat{C}^G$
\[\xymatrix{\Omega^{\tilde{T}}c\ar[r]\ar[d]&map_\cat{C}(
N\mathcal{P}(T)\backslash \emptyset,c)\ar[d]
\\
\ast\ar[r]&map_\cat{C}(\partial N\mathcal{P}(T)\backslash \emptyset,c),}
\ \ \ \ \ \ \
\xymatrix{(\partial N\mathcal{P}(T)^{op}\backslash T)\otimes c\ar[r]\ar[d]&\ast\ar[d]\\
(N\mathcal{P}(T)^{op}\backslash T)\otimes c\ar[r]&\Sigma^{\tilde{T}}c.
}\]
In the case of a pointed $G$-space or $G$-spectra we recover the usual equivariant loop and suspension spaces. These constructions define an adjoint pair of functors $(\Sigma^{\tilde{T}},\Omega^{\tilde{T}})$ on $\cat{C}^G$, by the sequence of natural bijections
\[\begin{array}{l}\cat{C}^G(\Sigma^{\tilde{T}}c,d)\cong
\cat{C}^{\mathcal{P}(2)\backslash 2}\big((N\mathcal{P}(T)^{op}\backslash T)\otimes c\leftarrow (\partial N\mathcal{P}(T)^{op}\backslash T)\otimes c\rightarrow\ast\otimes c,\Delta d\big)\cong\\
\vspace{-.3cm}
\\
\cat{C}^{\mathcal{P}(2)\backslash \emptyset}\big(\Delta c,
map_\cat{C}(
N\mathcal{P}(T)\backslash \emptyset,d)\rightarrow map_\cat{C}(\partial N\mathcal{P}(T)\backslash \emptyset,d)\leftarrow map_\cat{C}(\ast,d)\big)\cong\cat{C}^G(c,\Omega^{\tilde{T}}d).
\end{array}\]
Here we used that $\ast\otimes c=\ast$ and $map_\cat{C}(\ast,d)=\ast$, as $\cat{C}$ is pointed.
Similarly to the previous examples there are natural isomorphisms in $\cat{C}^G$
\[\holim_{\mathcal{P}(T)\backslash\emptyset} \omega^{\tilde{T}}c\cong \Omega^{\tilde{T}}c \ \ \ \textrm{ and } \ \ \ \hocolim_{\mathcal{P}(T)\backslash T} \sigma^{\tilde{T}}c\cong \Sigma^{\tilde{T}}c .\]
\end{example}

\begin{example}\label{ex:nf}
  We already saw that for a category $I$ with $G$-action the functor $N(I/-)\colon I\rightarrow sSet$ has an obvious $G$-structure. For a functor $F \colon I \to J$ and an object $j$ of $J$ one can form the over-category $F/j$ and the assignment $j \mapsto N(F/j)$ defines a functor $N(F/-) \colon J \to sSet$. If $F$ is an equivariant functor between categories with $G$-action there are functors $F/j \to F/(gj)$ induced by the $G$-actions, and after applying the nerve these give a $G$-structure on the diagram $N(F/-)$. In fact, $N(F/-)$ with this $G$-structure is the left Kan extension $F_\ast N(I/-)$ of $N(I/-)$ along $F$. This will be important later when we discuss homotopy cofinality and cofibrancy of $G$-diagrams. 
\end{example}

\begin{example}\label{ex:qx}
  Let $X \colon I \to \cat{C}$ be a diagram in a simplicial category $\cat{C}$. Define the diagram $qX$ by $qX_i = \hocolim_{I/i} u_i^\ast X$ where $u_i \colon I/i \to I$ is the functor that forgets the map to $i$. A map $\alpha \colon i \to j$ in $I$ induces a functor $I/i \to I/j$ and hence a map $qX_i \to qX_j$. The natural map from the homotopy colimit to the colimit induces maps 
\[qX_i = \hocolim_{I/i} u_i^\ast X \to \colim_{I/i} u_i^\ast X \iso X_i,\]
which combine to a map of diagrams $\rho_X \colon qX \to X$. If $X$ is a $G$-diagram then the functor $I/i \to I/gi$ induced by multiplication by $g \in G$ induces a map $qX_i \to qX_{gi}$ and together these maps constitute a $G$-structure on $qX$. It is a classical fact that the objects $\colim_I qX$ and $\hocolim_I X$ are isomorphic and in \ref{prop:hocolim2} of Proposition \ref{prop:hocolim} we prove that this isomorphism is $G$-equivariant when $X$ is a $G$-diagram.
\end{example}

\section{\texorpdfstring{$G$}{G}-diagrams and model structures}\label{sec:modstr}

This section provides a framework in which the equivariant constructions of homotopy limits and colimits defined earlier in the paper have homotopical sense, and are well behaved. The first step in developing this framework is to give the ambient category $\cat{C}$ enough structure to be able to define a model structure on the category of $G$-diagrams in $\cat{C}$. It turns out that having a model structure on the category $\cat{C}^G$ of $G$-objects in $\cat{C}$ is not enough, but one needs to have homotopical information for all the subgroups of $G$. The good context for a genuine equivariant homotopy theory seems to be that of an  ``equivariant model category''.

\subsection{Equivariant model categories}\label{secGmod}

Let $\cat{C}$ be a complete and cocomplete category, $G$ a finite group and $H,H'\leq G$ a pair of subgroups. A finite set $K$ with commuting left $H'$-action and right $H$-action induces a pair of adjoint functors
\[K\otimes_H(-)\colon\cat{C}^H\rightleftarrows \cat{C}^{H'}\colon \hom_{H'}(K,-)\]
The left adjoint is defined as
\[K\otimes_Hc=\colim\left(H\stackrel{\coprod_{K}c}{\longrightarrow}\cat{C}\right)\]
where $\coprod_{K}c$ is the $H$-equivariant colimit of the constant $H$-diagram $\Delta c$ on the discrete $H$-category $K^\delta$, (see Example \ref{ex:prod}), and the $H'$-action is induced by the $H'$-action on 
$K$. Dually, define
\[\hom_{H'}(K,d)=\lim\left(H' \stackrel{\prod_{K}c}{\longrightarrow}\cat{C}\right)\]
with left $H$-action defined by right action on $K$. These functors are adjoint via the sequence of natural isomorphisms
\[\begin{array}{lll}\cat{C}^{H'}(K\otimes_H c,d) \cong \cat{C}(K\otimes_H c,d)^{H'} \cong \cat{C}^H(\coprod_K c,d)^{H'}\cong\\
\cat{C}^{K}_a(\Delta_K c,\Delta_Kd)^{H'} \cong \cat{C}^{H}(c,\prod_Kd)^{H'} \cong \cat{C}^{H}(c,\lim_{H'}(\prod_Kd))=\\ \cat{C}^{H}(c,\hom_{H'}(K,d))
\end{array}\]

In the following we will always use the fixed point model structure on $sSet^G$
(see e.g. \cite[1.2]{Shipley03aconvenient}) unless otherwise is stated.

\begin{definition}\label{defGmodelstr}
A $G$-model category is a cofibrantly generated simplicial model category $\cat{C}$, together with the data of a cofibrantly generated model structure on $\cat{C}^H$ for every subgroup $H\leq G$, satisfying
\begin{enumerate}
\item The model structure on $\cat{C}^H$ together with the $sSet^H$-enrichment, tensored and cotensored structures induced from $\cat{C}$ forms a cofibrantly generated $sSet^H$-enriched model structure on $\cat{C}^H$,

\item For every pair of subgroups $H,H'\leq G$, and finite set $K$ with commuting free left $H'$-action and free right $H$-action the adjunction
\[K\otimes_H(-)\colon\cat{C}^H\rightleftarrows \cat{C}^{H'}\colon \hom_{H'}(K,-)\]
is a Quillen adjunction.

\end{enumerate}
\end{definition}

\begin{remark}\label{adjforget}
For $H'\leq H$ and $K=H$ with actions given by left $H'$ and right $H$ multiplications, the functor
\[H\otimes_H(-)\colon \cat{C}^{H}\longrightarrow \cat{C}^{H'}\]
is isomorphic to the functor $\res^{H}_{H'}$ that restricts the action. Similarly for $K=H$ with left $H$ multiplication and right $H'$ multiplication the functor  
\[\hom_H(H,-)\colon \cat{C}^{H}\longrightarrow \cat{C}^{H'}\]
is also isomorphic to the functor $\res^{H}_{H'}$. It follows from the second condition that $\res^{H}_{H'}$ is both a left and a right Quillen functor, and therefore it preserves cofibrations, acyclic cofibrations, fibrations, acyclic fibrations and equivalences between cofibrant or fibrant objects. 
\end{remark}

\begin{example}
Let $\cat{C}$ be a cofibrantly generated $sSet$-enriched model category. The collection of projective model structures (na\"{i}ve) on $\cat{C}^H$ for $H\leq G$ defines a $G$-model structure on $\cat{C}^G$. To see this, just notice that if $H'$-acts freely on $K$, a choice of section for the quotient map $K\rightarrow H'\backslash K$ induces a natural isomorphism
\[\res^{H}_{e}\hom_{H'}(K,c)\cong\prod_{H'\backslash K}c\]
where $\res^{H}_{e}\colon \cat{C}^H\rightarrow \cat{C}$ is the forgetful functor. Therefore $\hom_{H'}(K,-)$ preserves fibrations and acyclic fibrations.
\end{example}

\begin{example}\label{exfixedptsmodel}
Let $\cat{C}$ be a cofibrantly generated $sSet$-enriched model category, and fix a pair of finite groups $H\leq G$. For all subgroup $L\leq H$, the $L$-fixed points functor  $(-)^L\colon \cat{C}^H\rightarrow \cat{C}$ is defined as the composite
\[\cat{C}^H\stackrel{\res^{H}_L}{\longrightarrow}\cat{C}^L\stackrel{\lim}{\longrightarrow}\cat{C}\]
If these functors are cellular in the sense of \cite{GuiMay}, the category $\cat{C}^H$ inherits a $sSet^H$-enriched model structure where weak equivalences and fibrations are the maps that are sent by $(-)^L$ respectively to weak equivalences and fibrations in $\cat{C}$, for every subgroup $L\leq H$ (cf. \cite[2.8]{ManMay},\cite{GuiMay},\cite{Marc}). This construction specifies to the standard fixed points model structure on (pointed) spaces with $H$-action.

The collection of model categories $\cat{C}^H$, for $H$ running over the subgroups of $G$, assemble into a  $G$-model category. Let us see that the left adjoint $K\otimes_H(-)$ is a left Quillen functor. The generating cofibrations of $\cat{C}^H$ are by definition the images of the generating cofibrations of $\cat{C}$ by the functors
\[J\otimes(-)\colon \cat{C}\longrightarrow \cat{C}^H\]
where $J$ ranges over finite sets with left $H$-action. Similarly for generating acyclic cofibrations. There is a natural isomorphism
\[K\otimes_H(J\otimes(-))\cong (K\times_HJ)\otimes(-)\]
and the right hand functor preserves cofibrations and acyclic cofibrations by assumption.
Thus $K\otimes_H(-)$ preserves generating (acyclic) cofibrations. Since it is a left adjoint it preserves colimits, and therefore all (acyclic) cofibrations (see e.g. \cite[11.2]{hirsch}).
\end{example}

\begin{example}
Let $\cat{C}=\Sp^{O}$ be the category of orthogonal spectra and $G$ a finite group. The category $(\Sp^{O})^G$ of $G$-objects in $\Sp^O$ is naturally equivalent to the category of orthogonal $G$-spectra  $\mathscr{J}^{\mathscr{V}}_G\mathscr{S}$ of \cite{ManMay} indexed on a universe $\mathscr{V}$  for finite dimensional $G$-representations (cf.\cite[V.1]{ManMay}, \cite[2.7]{Schwede}). Given any subgroup $H\leq G$, we endow $(\Sp^{O})^H$ with the model structure induced by the stable model structure on $\mathscr{J}^{i^\ast\mathscr{V}}_H\mathscr{S}$ of \cite{ManMay} under the equivalence of categories $(\Sp^{O})^H\simeq \mathscr{J}^{i^\ast\mathscr{V}}_H\mathscr{S}$. Here $i\colon H\rightarrow G$ denotes the inclusion, and $i^\ast\mathscr{V}$ is the universe of representations of $H$ that are restrictions of representations of $G$ in $\mathscr{V}$. The adjunctions 
\[K\otimes_H(-)\colon(\Sp^O)^H\rightleftarrows (\Sp^O)^{H'}\colon \hom_{H'}(K,-)\]
are the standard induction-coinduction adjunctions, and they are Quillen adjunctions by \cite[V-2.3]{ManMay}. The collection of model categories $\{(\Sp^O)^H\}_{H\leq G}$ then forms a $G$-model category.
\end{example}


\subsection{The ``\texorpdfstring{$G$}{G}-projective'' model structure on \texorpdfstring{$G$}{G}-diagrams}

Let $G$ be a finite group, $\cat{C}$ a category, and $I$ a small category with $G$-action $a$. Given a $G$-diagram $X$ in $\cat{C}^{I}_a$ and an object $i\in I$, the vertex $X_i\in \cat{C}$ inherits from the $G$-structure on $X$ an action by the stabilizer group $G_i\leq G$ of the object $i$. This gives an evaluation functor $\ev_i\colon \cat{C}^{I}_a \rightarrow \cat{C}^{G_i}$ for every object $i$.

\begin{theorem}\label{modelstruct} Let $\cat{C}$ be a $G$-model category (see \ref{defGmodelstr}).
There is a cofibrantly generated $sSet^G$-enriched model structure on the category of $G$-diagrams $\cat{C}^{I}_a$ with
\begin{enumerate}
\item weak equivalences the maps of $G$-diagrams $f\colon X\rightarrow Y$ whose evaluations $\ev_i f$ are weak equivalences in $\cat{C}^{G_i}$ for every $i\in I$,
\item fibrations the maps of $G$-diagrams $f\colon X\rightarrow Y$ whose evaluations $\ev_i f$ are fibrations in $\cat{C}^{G_i}$ for every $i\in I$,
\item generating cofibrations and acyclic cofibrations
\[F\mathcal{I}=\bigcup_{i\in I}F_i\mathcal{I}_i \ \ \ \ \ \mbox{and}\ \ \ \ \ F\mathcal{J}=\bigcup_{i\in I}F_i\mathcal{J}_i\]
where $\mathcal{I}_i$ and $\mathcal{J}_i$ are respectively generating cofibrations and generating acyclic cofibrations of $\cat{C}^{G_i}$, and $F_i\colon \cat{C}^{G_i}\rightarrow \cat{C}^{I}_a$ is the left adjoint to the evaluation functor $\ev_i$.
\end{enumerate} 
 
\end{theorem}

\begin{remark}
Under the isomorphism $\cat{C}^{I}_a\cong \cat{C}^{G\rtimes_a I}$ of Lemma \ref{lemma:eqdef} the evaluation functor $\ev_i$ corresponds to restriction along the functor $\iota_i\colon G_i\rightarrow G\rtimes_a I$ that sends the unique object to $i$ and a morphism $g$ to $(g,\id_{i}\colon gi=i\rightarrow i)$. Since $\cat{C}$ has all colimits a left adjoint for $\ev_i$ exists. Also notice that the model structure on $\cat{C}^{I}_a$ above does not correspond to the projective model structure on $\cat{C}^{G\rtimes_a I}$.
\end{remark}

Before proving the theorem we need to identify the left adjoints of the evaluation functors.
For fixed objects $i,j\in I$ let $K_{ji}$ be the morphisms set
\[K_{ji}=\hom_{G\rtimes_a I}(i,j)=\{(g\in G,\alpha\colon gi\rightarrow j)
\}\]
The stabilizer group $G_j$ acts freely on the left on $K_{ji}$ by left multiplication on $G$ and by the category action on the morphism component. The group $G_i$ acts freely on the right on $K_{ji}$ by right multiplication on the $G$-component.
For every $c\in \cat{C}^{G_i}$ define a diagram $F_ic\colon I\rightarrow \cat{C}$ by sending an object $j\in I$ to
\[(F_ic)_j=K_{ji}\otimes_{G_i} c\]
A morphism $\beta\colon j\rightarrow j'$ in $I$ induces a map $(F_ic)_j\rightarrow (F_ic)_{j'}$ via the $G_i$-equivariant map $\beta_\ast\colon K_{ji}\rightarrow K_{j'i}$
\[\beta_{\ast}(g,\alpha\colon gi\rightarrow j)=(g,\beta\circ\alpha)\]
The $G_i$-equivariant maps $g\colon K_{ji}\rightarrow K_{(gj)i}$ 
\[g(g',\alpha\colon g'i\rightarrow j)=(gg',g\alpha\colon gg'i\rightarrow gj)\]
define a $G$-structure on $F_ic$. The construction is clearly functorial in $c$, defining a functor $F_i\colon \cat{C}^{G_i}\rightarrow \cat{C}^{I}_a$.

\begin{lemma}\label{adjevi}
The functor $F_i\colon \cat{C}^{G_i}\rightarrow \cat{C}^{I}_a$ is left adjoint to the evaluation functor $\ev_i\colon \cat{C}^{I}_a\rightarrow \cat{C}^{G_i}$.
\end{lemma}

\begin{proof} We prove that under the isomorphism  $\cat{C}^{I}_a\cong \cat{C}^{G\rtimes_a I}$ of Lemma \ref{lemma:eqdef} the functor $F_i$ corresponds to the left Kan extension along the inclusion $\iota_i\colon G_i\rightarrow G\rtimes_a I$. For an object $j\in I$, the category $\iota_i/j$ is the disjoint union of categories
\[\iota_i/j=\coprod\limits_{\substack{z\in G/G_i\\ zi\rightarrow j}}Ez\]
where $Ez$ is the translation category of the right $G_i$-set $z$, with one object for every element of the orbit $z$, and a unique morphism $h\colon g\rightarrow g'$ whenever $g'=gh^{-1}$ for some $h\in G_i$.  An object $c\in \cat{C}^{G_i}$ induces a diagram $Ec\colon Ez\rightarrow G_i\stackrel{c}{\rightarrow} \cat{C}$, where the first functor collapses all the objects to the unique object of $G_i$, and sends the unique morphism $g\rightarrow gh^{-1}$ to $h$. The left Kan extension along $\iota_i$ at $c$ is by definition the diagram $L_ic$ with $j$-vertex
\[(L_ic)_j=\coprod\limits_{\substack{z\in G/G_i\\ zi\rightarrow j}}\colim_{Ez}Ec\]
Notice that the indexing set of the coproduct is precisely the orbit set $K_{ji}/G_i$.
There is a canonical map of diagrams $F_ic\rightarrow L_jc$, which at a vertex $j$ is induced by
\[\coprod_{K_{ji}}c\longrightarrow\coprod\limits_{K_{ji}/G_i}
\colim_{Ez}Ec=(L_ic)_j\]
which on the $(g,\alpha)$-component is the canonical map $c=(Ec)_g\rightarrow \colim_{E[g]}Ec$ to the $[g,\alpha]$-coproduct component. This map respects the $G_j$-structure, which on $L_ic$ acts on the indexing sets $K_{ji}/G_i$. To show that it is an isomorphism, choose a section $s\colon G/G_i\rightarrow G$ for the projection map. This gives a map 
\[(L_ic)_j=\coprod\limits_{K_{ji}/G_i}
\colim_{Ez}Ec\longrightarrow \coprod_{K_{ji}}c\longrightarrow K_{ji}\otimes_{G_i}c=(F_ic)_j\]
that on the $(z,\alpha)$-component is the map induced by $s(z)^{-1}g\colon(Ez)_g=c\rightarrow c$ to the $(s(z),\alpha)$-component.

\end{proof}

\begin{proof}[Proof of \ref{modelstruct}]
Weak equivalences and fibrations in $\cat{C}^{I}_a$ are by definition the morphisms that are sent to weak equivalences and fibrations, respectively, by the functor
\[\prod\limits_{i\in I}\ev_i\colon \cat{C}^{I}_a\longrightarrow \prod_{i\in I}\cat{C}^{G_i}\]
It follows from Lemma \ref{adjevi} that the coproduct of the functors $F_i$ defines a left adjoint
\[F\colon \prod_{i\in I}\cat{C}^{G_i}\stackrel{\prod F_i}{\longrightarrow} \prod_{i\in I}\cat{C}^{I}_a\stackrel{\coprod}{\longrightarrow}\cat{C}^{I}_a\]
for the product of the evaluation functors. The collections
\[\mathcal{I}=\bigcup_{i\in I}(\mathcal{I}_i\times\prod_{j\neq i}\id_{\emptyset_j})\ \ \ \ \mbox{and}\ \ \ \ \mathcal{J}=\bigcup_{i\in I}(\mathcal{J}_i\times\prod_{j\neq i}\id_{\emptyset_j})\]
generate respectively the cofibrations and the acyclic cofibrations of $\prod \cat{C}^{G_i}$ (see e.g. \cite[11.1.10]{hirsch}), where $\emptyset_j$ is the initial object of $\cat{C}^{G_i}$. Moreover their images by $F$ are precisely the families $F\mathcal{I}$ and $F\mathcal{I}$ from the statement.
Following \cite[11.3.1]{hirsch} and \cite[D.21]{Marc}, we prove that
\begin{enumerate}[i)]
\item $\prod \ev_j$ takes relative $F\mathcal{I}$-cell complexes to cofibrations: Let $\lambda$ be a non-zero ordinal and $X\colon \lambda\rightarrow \cat{C}^{I}_a$ a functor such that for all morphism $\beta\rightarrow\beta'$ in $\lambda$ the map $X_\beta\rightarrow X_{\beta'}$ is a pushout of a map in $F\mathcal{I}$. We need to show that for every $j\in I$ the map
\[\ev_j X_0\longrightarrow \ev_j\colim_{\lambda} X=\colim_\lambda\ev_j\circ X\]
is a cofibration in $\cat{C}^{G_i}$. Since $\ev_j$ commutes with colimits, each map $\ev_j X_\beta\rightarrow \ev_j X_{\beta'}$ is the pushout of a map in $\ev_jF\mathcal{I}$. Thus we need to show that every map in  $\ev_jF\mathcal{I}$ is a cofibration of $\cat{C}^{G_j}$. By definition of $\mathcal{I}$, this is the same as showing that for all $i,j\in I$ every generating cofibration of $\mathcal{I}_i$ is sent by $\ev_jF_i$ to a cofibration of $\cat{C}^{G_j}$. The composite functor $\ev_jF_i$ is by definition
\[\ev_jF_i=K_{ji}\otimes_{G_{i}}(-)\colon \cat{C}^{G_i}\longrightarrow \cat{C}^{G_j}\]
which sends generating cofibrations to cofibrations as part of the axioms of a $G$-model category (see \ref{defGmodelstr}).
\item $\prod \ev_j$ takes relative $F\mathcal{J}$-cell complexes to acyclic cofibrations: the argument is similar to the one above.
\end{enumerate}
Moreover $\prod \ev_j$ preserves colimits. By \cite[11.3.1]{hirsch} and \cite[D.21]{Marc}, the families $F\mathcal{I}$ and $F\mathcal{J}$ are respectively a class of generating cofibrations and acyclic cofibrations for the $sSet^G$-enriched model structure on $\cat{C}^{I}_a$ with the fibrations and weak equivalences of the statement.
\end{proof}

\begin{remark}\label{cofibgdiags}
Recall the isomorphism $\cat{C}^{I}_a\cong \cat{C}^{G\rtimes_a I}$ of Lemma \ref{lemma:eqdef}. The model structure on $\cat{C}^{I}_a$ does not correspond to the projective model structure on $\cat{C}^{G\rtimes_a I}$. However, every fibration (resp. weak equivalence) in $\cat{C}^{I}_a$ is in particular a fibration (resp. weak equivalence) in $\cat{C}^{G\rtimes_a I}$. This means that the cofibrations of $\cat{C}^{G\rtimes_a I}$ are also cofibrations in $\cat{C}^{I}_a$. In particular, a sufficient condition for an object of $\cat{C}^{I}_a$ to be cofibrant is to be cofibrant in the projective model structure of $\cat{C}^{G\rtimes_a I}$.
\end{remark}

\begin{proposition}\label{ptwisecof}
If $X\in \cat{C}^{I}_a$ is cofibrant, each vertex $X_i$ is cofibrant in $\cat{C}^{G_i}$.
\end{proposition}
\begin{proof}
An argument dual to the proof of Lemma \ref{adjevi} shows that the right adjoint $R_i$ to the evaluation functor $\ev_i\colon \cat{C}^{I}_a\rightarrow \cat{C}^{G_i}$ has $j$-vertex
\[\ev_jR_i=\hom_{G_i}(K_{ji}^\ast,-)\]
where $K_{ji}^\ast$ is the set $K_{ji}$ with left $G_i$-action $g\cdot  k:=k\cdot g^{-1}$ and right $G_j$-action $k\cdot g:=g^{-1}\cdot k$. Hence $\ev_jR_i$ is a right Quillen functor by the axioms of a $G$-model category.  Since the fibrations and the equivalences on $\cat{C}^{I}_a$ are point-wise, $R_i\colon \cat{C}^{G_i}\rightarrow \cat{C}^{I}_a$ is also a right Quillen functor. It follows that $\ev_i$ is a left Quillen functor, and in particular it preserves cofibrant objects.
\end{proof}

\begin{definition}\label{GQuillen}
Let $\cat{C}$ and $\cat{D}$ be $G$-model categories. A $G$-Quillen adjunction (resp. equivalence) is an enriched adjunction $\cat{C}\rightleftarrows \cat{D}$ such that the induced adjunction $\cat{C}^H\rightleftarrows \cat{D}^{H}$ is a Quillen adjunction (resp. equivalence) for every subgroup $H\leq G$.
\end{definition}

\begin{example}
The Quillen equivalence $|-|\colon sSet\rightleftarrows Top \colon Sing$ (see \cite[I]{GJ}) is a $G$-Quillen equivalence for any finite group $G$. 
\end{example}

\begin{corollary}
A $G$-Quillen equivalence $L\colon \cat{C}\rightleftarrows \cat{D}\colon R$ induces a Quillen equivalence 
\[L\colon \cat{C}^{I}_a\rightleftarrows \cat{D}^{I}_a\colon R.\] 
\end{corollary}
\begin{proof}
The adjunction $L\colon \cat{C}^{I}_a\rightleftarrows \cat{D}^{I}_a\colon R$ is a Quillen adjunction since the right adjoint preserves fibrations and acyclic fibrations, as they are defined point-wise. Let $X\in \cat{C}^{I}_a$ be cofibrant and $Y\in \cat{D}^{I}_a$ fibrant. A map $X\rightarrow R(Y)$ is an equivalence if and only if its adjoint $L(X)\rightarrow Y$ is, since by Proposition \ref{ptwisecof} $X$ is point-wise cofibrant.
\end{proof}

\subsection{Cofibrant replacement of \texorpdfstring{$G$}{G}-diagrams}\label{sec:cofrepl}

When $\cat{C}$ is a cofibrantly generated simplicial model category and $I$ is a small category a standard way to replace a diagram $X \colon I \to \cat{C}$ by a cofibrant diagram is by the construction of Example \ref{ex:qx}. Namely, one defines $qX$ by $qX_i = \hocolim_{I/i} (u_i^\ast X)$ where $u_i \colon I/i \to I$ is the functor that forgets the map to $i$. Then $qX$ is cofibrant in the projective model structure on $\cat{C}^I$ and the natural map $\rho_X \colon qX \to X$ is a weak equivalence if $X$ has cofibrant values in $\cat{C}$. In this section we will generalize this to $G$-diagrams as follows:

\begin{theorem}\label{thm:qxcofrepl}
  If $X$ is a $G$-diagram such that for all $i$ in $I$ the value $X_i$ is cofibrant in $\cat{C}^{G_i}$, then the map $\rho_X \colon qX \to X$ is a cofibrant replacement of $G$-diagrams in the sense that $qX$ is cofibrant and $\rho_X$ is a weak equivalence.
\end{theorem}

The proof is technical and will occupy the rest of this section. We begin by fixing some notation. Let $I$ be a small category with an action $a$ of $G$. Write $I^\delta$ for the discrete category with the same objects as $I$ but no non-identity morphisms. The inclusion $I^\delta \inj I$ is equivariant and  induces a restriction functor $r \colon \cat{C}^I_a \to\cat{C}^{I^\delta}_a$ with left adjoint $r_\ast$. We abbreviate $r(X)$ as $X^\delta$. Note that the functor $r$ preserves fibrations and weak equivalences and hence is a right Quillen functor. It follows that the left adjoint $r_\ast$ is a left Quillen functor. We say that an $I$-indexed $G$-diagram $X$ is \emph{point-wise cofibrant} if for each object $i$ in $I$ the value $X_i$ is cofibrant in $\cat{C}^{G_i}$. 
\begin{lemma}\label{lemma:eqcof}
  \begin{enumerate}[i)]
  \item \label{lemma:eqcof1} If $Y$ is an $I^\delta$-indexed $G$-diagram which is point-wise cofibrant, then $Y$ is cofibrant in $\cat{C}^{I^\delta}_a$.
 \item \label{lemma:eqcof2} In particular, if $X$  is a point-wise cofibrant $I$-indexed $G$-diagram then $r_\ast X^\delta$ is cofibrant in $\cat{C}^I_a$.
  \end{enumerate}
\end{lemma}

\begin{proof}To see that part \ref{lemma:eqcof1}) holds, consider a square 
\begin{equation}\label{eq:cofsq} \xymatrix{&\emptyset \ar[d] \ar[r] & Z \ar@{>>}[d]^f_\sim \\
&Y \ar[r] & W}
\end{equation}
in $\cat{C}^{I^\delta}_a$, where the right hand vertical map is a trivial fibration and $\emptyset$ denotes the initial object. The map $f$ being a trivial fibration means exactly that each component $f_i \colon Z_i \to W_i$ is a trivial fibration in $\cat{C}^{G_i}$. Choose a representative $i$ of each $G$-orbit in $obI$. Each resulting square
\[ \xymatrix{&\emptyset \ar@{>->}[d] \ar[r] & Z_i \ar@{>>}[d]_\sim^{f_i} \\
&Y_i \ar[r] \ar@{-->}[ur]^{\lambda_i} & W_i}
\]
has a lift $\lambda_i$ since $Y_i$ is cofibrant and $f_i$ is a trivial fibration in $\cat{C}^{G_i}$. For $g \in G$ define $\lambda_{gi} = g_{Z_i} \circ  \lambda_i \circ g_{Y_i}^{-1}$. Then, if $gi = i$ the $G_i$-equivariance of the map $\lambda_i$ says precisely that $\lambda_{i} = g_{Z_i} \circ \lambda_i \circ g_{Y_i}^{-1} =  \lambda_{gi}$, so for all $i$ and all $g \in G$ the map $\lambda_{gi}$ is well-defined. It is now easy to see that the $\lambda_{gi}$'s assemble to a map of $G$-diagrams giving a lift in the square (\ref{eq:cofsq}).

Part \ref{lemma:eqcof2}) follows immediately from part \ref{lemma:eqcof1}) and the fact that $r_\ast$ is a left Quillen functor and hence preserves cofibrancy of objects.
\end{proof}

The adjunction $(r_\ast,r)$ induces a comonad $r_\ast r$ on $\cat{C}^I_a$ in the usual way. For a $G$-diagram $X$ the value $(r_\ast r) X$ on $i$ is
\[(r_\ast r) X_i = \coprod_{\alpha \colon j \to i} X_j.\]
The counit $\varepsilon \colon (r_\ast r) X \to X$ maps the $X_j$-component in the coproduct indexed by $\alpha \colon j \to i$ to $X_i$ by the map $X(\alpha)$. The comultiplication $c \colon (r_\ast r) X \to (r_\ast r r_\ast r) X$ has as $i$-component the map
\[\coprod_{\alpha \colon j \to i} X_j \to \coprod_{\alpha \colon j \to i} \left(\coprod_{\alpha' \colon k \to j} X_k \right) \]
that maps the $X_j$-summand indexed by $\alpha \colon j \to i$ by the identity to the $X_j$-summand indexed by $id_j$ in the $\alpha$-summand of the target.

Let $X$ be a $G$-diagram indexed on $I$. The bar construction on the comonad $r_\ast r$ gives a simplicial $G$-diagram $B(r_\ast r)X$ with $B_n(r_\ast r)X = (r_\ast r)^{n+1}X$ so that 
\[B_n(r_\ast r)X_i  = \coprod_{\alpha_0 \colon i_0 \to i}\coprod_{\alpha_1 \colon i_1 \to i_0} \cdots \coprod_{\alpha_n \colon i_n \to i_{n-1}} X_{i_n} \cong \coprod_{i_n \to \cdots \to i_0 \to i} X_{i_n}. \]  
Note that for varying $n$ the indexing $G_i$-simplicial set can be identified with $N(I/i)^{op}$. For 
\[ \sigma = i_n \stackrel{\alpha_n}{\too} \cdots \stackrel{\alpha_1}{\too} i_0 \stackrel{\alpha_0}{\too} i\]
in  $N_n(I/i)^{op}$ the face map $d_{n-k}$ for $k < 0$ composes the maps $\alpha_k$ and $\alpha_{k-1}$ and $d_0$ maps $X_{i_n}$ to the $X_{i_{n-1}}$ indexed by $ d_0(\sigma) \in N_{n-1}(I/i)^{op}$ by the map $X(\alpha_n)$. The degeneracy map $s_n$ inserts an identity in the $(n-l)$-spot. Note that 
\[ \colim_I r_\ast r X = \colim_{I^\delta}rX = \coprod_i X_i ,\]
so that $\colim_I B_n(r_\ast r)X \cong \coprod_{\sigma \in N_n(I^{op})} X_{\sigma(n)}$ and $\colim_I B(r_\ast r)X$ is isomorphic to the usual simplicial replacement $\coprod_\ast X$ of Bousfield and Kan \cite{BK} with $G$-action induced by the $G$-structure on $X$.
\begin{proposition}\label{prop:hocolim}
  Let $X$ be an $I$-indexed $G$-diagram. Then there are natural isomorphisms in $\cat{C}^G$
\begin{enumerate}[i)]
\item \label{prop:hocolim1} $N(-/I)^{op} \otimes_I^a X \cong \textstyle{|\coprod_ \ast}  X|$
\item \label{prop:hocolim2} $\textstyle{|\coprod_ \ast}  X| \cong \colim_I qX.$
\end{enumerate}

\end{proposition}

\begin{proof} To see \ref{prop:hocolim1} we first decompose the tensor product as an iterated coend (cf. \cite[\S 6.6]{riehl})
\[N(-/I)^{op} \otimes_I^a X = \int^i N(i/I)^{op} \otimes X_i \cong \int^i \left(\int^{[n]} \Delta^n \times N_n\left(i/I \right)^{op}\right) \otimes X_i.\]
Here and in the rest of the proof we leave it to the reader to check that this is compatible with the $G$-structures on the diagrams. Rearranging the parentheses and switching the order of the coends gives the isomorphic object
\[\int^{[n]} \int^i \Delta^n \otimes \left( N_n \left( i/I \right) ^{op} \otimes X_i \right) \cong \int^{[n]} \Delta^n  \otimes \left(\int^i \coprod_{i \to i_n \to \cdots \to i_0} X_i\right).\]
Now we analyze the latter $\int^i$-factor. It is a coend of the $G$-diagram $I^{op} \times I \to \cat{C}$ given by 
\[(i,j) \mapsto \coprod_{i \to i_n \to \cdots \to i_0} X_j.\]
This is isomorphic to the diagram
\[ (i,j) \mapsto \coprod_{i_n \to \cdots \to i_0} I(i,i_n) \otimes X_j \]
and we note that since coends commute with colimits there is an isomorphism
\[ \int^i  \coprod_{i_n \to \cdots \to i_0} I(i,i_n) \otimes X_i \cong \coprod_{i_n \to \cdots \to i_0}  \int^i I(i,i_n) \otimes X_i.\]
Here we must be careful since the representable functor $I(-,i_n)$ is not itself a $G$-diagram, but the coproduct  $\coprod_{\sigma \in N_n (I^{op})} I(-,\sigma(n))$ of representable functors is. Finally, we observe that $\int^i I(i,i_n) \otimes X _i \cong X_{i_n}$ so that 
\[ \int^{[n]} \Delta^n  \otimes \left(\int^i \coprod_{i \to i_n \to \cdots \to i_0} X_i \right) \cong \int^{[n]} \Delta^n \otimes \left(\coprod_{i_n \to \cdots \to i_0} X_{i_n} \right) = | \textstyle{\coprod_\ast} X |\]

To get the isomorphism in \ref{prop:hocolim2}) we recall the isomorphism $ \colim_I B(r_\ast r)X \cong \textstyle{\coprod_\ast} X$. Since realization commutes with colimits, there are natural isomorphisms
\[|{\textstyle{\coprod_\ast}} X | \cong |\colim_I B(r_\ast r)X| \cong  \colim_I | B(r_\ast r)X|.\]
Evaluating at $i$ gives
\[| B(r_\ast r)X|_i = \left| [n] \mapsto \coprod_{ i_n \to \cdots \to i_0 \to i} X_{i_n} \right|  \cong \hocolim_{I/i} (u_i^\ast X)\]
where the last isomorphism is an instance of \ref{prop:hocolim1}) for the $G_i$-diagram $u_i^\ast X \colon I/i \to \cat{C}$. This gives an isomorphism 
\[\colim_I | B(r_\ast r)X| \cong \colim_I qX. \] 

\end{proof}

\begin{lemma}\label{lemma:brreedy}
  If $X$ is a point-wise cofibrant $G$-diagram, then the simplicial object $B(r_\ast r)X$ is Reedy cofibrant in $(\cat{C}^I_a)^{\Delta^{op}}$.
\end{lemma}

\begin{proof}
Let $L=L_nB(r_\ast r X)$ be the $n$-th latching object of $B(r_\ast r X)$. The natural map 
\[L_nB(r_\ast r X) \to B_n(r_\ast r X) = B\]
 is at each $i$ in $I$ the inclusion of the summands indexed by the degenerate $n$-simplices in $N_n(I/i)^{op}$ into the coproduct over all $n$-simplices. Thus $B$ decomposes as a coproduct $B= L \amalg N$ where the value of $N$ at $i$ is the coproduct indexed over all the \emph{non}-degenerate simplices of the nerve. The decomposition is clearly compatible with the $G$-diagram structure on each factor. The diagram $N$ is obtained by applying $r_\ast$ to a point-wise cofibrant $I^\delta$-indexed $G$-diagram and is therefore cofibrant. It follows that the map $L \to B$ is a cofibration.
\end{proof}

\begin{corollary}\label{cor:qxcof}
  If $X$ is a point-wise cofibrant $G$-diagram, then $qX$ is cofibrant.
\end{corollary}

\begin{proof}
  We know from the proof of Proposition \ref{prop:hocolim} that $qX$ is the realization of the simplicial object $B(r_\ast r)X$ which is Reedy cofibrant by Lemma \ref{lemma:brreedy}. Since realization takes Reedy cofibrant objects to cofibrant objects \cite[VII,3.6]{GJ} it follows that $qX$ is cofibrant.
\end{proof}

\begin{example}\label{ex:ncof}
  Let $\ast_{I}$ be the $I$-indexed $G$-diagram with value the terminal object $\ast$ of $sSet$. Then $q(\ast_{I})_i = \hocolim_{I/i} (\ast_{I/i}) \cong N(I/i)^{op}$, so that $q(\ast_{I}) \cong N(I/-)^{op}$ and similarly $q(\ast_{I^{op}}) \cong N(-/I)$. By Corollary \ref{cor:qxcof} it follows that the diagrams $N(I/-)$ and $N(-/I)^{op}$ are cofibrant as $G$-diagrams since $\ast$ is cofibrant in $sSet^{G_i}$ for all $i$ in $I$ and taking opposite simplicial sets preserves cofibrations. 
Further, let $I$ and $J$ be categories with respective $G$-actions $a$ and $b$, and $F \colon I \to J$ an equivariant functor. Since the left Kan extension $F_\ast$ preserves cofibrancy the diagrams $N(F/-) \cong F_\ast N(I/-)$ and $N(-/F)^{op} \cong F_\ast N(-/I)^{op}$ are also cofibrant in $sSet^J_b$. 
\end{example}

\begin{proof}[Proof of Theorem \ref{thm:qxcofrepl}]
  It only remains to see that the map $\rho_X$ is a weak equivalence. For this we must show that for each $i$ the map $\rho_{X_i} \colon \hocolim_{I/i} u_i^\ast X \to X_i$ is a weak equivalence in $\cat{C}^{G_i}$. The functor $\iota \colon \ast \to I/i$ sending the unique object to the terminal object is homotopy cofinal in the sense of Definition \ref{def:gcof}, so by Theorem \ref{Gcof} the map $X_i = \hocolim_\ast \iota^\ast u_i^\ast X \to \hocolim_{I/i} u_i^\ast X$ is a weak equivalence. Since it is also section to the map $\rho_{X_i}$ it follows by the two out of three property that $\rho_{X_i}$ is a weak equivalence as well.
\end{proof}


\subsection{Homotopy invariance of map, tensor and of homotopy (co)limits}\label{htpyinvsec}
In this section $\cat{C}$ is a $G$-model category in the sense of definition \ref{defGmodelstr}, and $a$ is a $G$-action on a small category $I$.

\begin{proposition}\label{htpyinvhom}
Let $X\in \cat{C}^{I}_a$ be a $G$-diagram in $\cat{C}$.
If $X$ is fibrant, the functor 
\[map^a_I(-,X)\colon (sSet^{I}_a)^{op}\longrightarrow \cat{C}^G\]
preserves equivalences of cofibrant objects (in $sSet^{I}_a$).
Dually, if $X$ is point-wise cofibrant, the functor
\[(-)\otimes_I^a X\colon sSet^{I^{op}}_a\longrightarrow \cat{C}^G\]
preserves equivalences of cofibrant objects.
\end{proposition}

\begin{proof} We prove the statement for $map^a_I$, the proof for $\otimes_I^a$ is similar. Let $K\to L$ be an equivalence of cofibrant diagrams in $sSet^{I}_a$. By Ken Brown's Lemma we can assume that $K\rightarrow L$  is  a cofibration (cf. \cite[7.7.1]{hirsch}). To show that the induced map is an equivalence, we need to solve the lifting problem
\[\xymatrix{A\ar@{>->}[d]\ar[r]&map^a_I(L,X)\ar[d]\\
B\ar@{-->}[ur]\ar[r]&map^a_I(K,X)
}\]
for every cofibration $A\rightarrow B$ in $\cat{C}^{G}$. Let $Map_\cat{C}(B,X)$ be the $G$-diagram in $sSet$ given by $i \mapsto Map_\cat{C}(B,X_i)$ and where the $G$-structure is given by the maps $Map_\cat{C}(g^{-1},g_{X_i}) \colon Map_\cat{C}(B,X_i) \to Map_\cat{C}(B,X_{gi})$. The adjunction isomorphism
\[\underline{\cat{C}}^G(B,map^a_I(L,X))\cong \underline{sSet}_{a}^{I}(L,Map_\cat{C}(B,X)).\] 
is equivariant, and therefore the lifting problem above is equivalent to the lifting problem in $sSet^{I}_a$
\[\xymatrix{K\ar@{>->}[d]_{\simeq}\ar[r]&Map_\cat{C}(B,X)\ar[d]\\
L\ar@{-->}[ur]\ar[r]&Map_\cat{C}(A,X)
}\]
This can be solved if $Map_\cat{C}(B,X) \rightarrow Map_\cat{C}(A,X)$ is a fibration in $sSet^{I}_a$, i.e., if for every object $i\in I$ the map $Map_\cat{C}(B,X_i)\rightarrow Map_\cat{C}(A,X_i)$ is a fibration of simplicial $G_i$-sets. By assumption $X_i$ is fibrant in $\cat{C}^{G_i}$ and $A \to B$ restricts to a cofibration in $\cat{C}^{G_i}$, so by axiom $SM7$ for the $sSet^{G_i}$-enriched model category  $\cat{C}^{G_i}$ the map is a fibration.
\end{proof}

\begin{proposition}\label{htpyinvtarget}
If $K$ is a cofibrant diagram in $sSet^{I}_a$, the functor
\[map^a_I(K,-)\colon \cat{C}^{I}_a\longrightarrow \cat{C}^G\]
preserves equivalences of fibrant objects. Dually if $K$ is cofibrant in $sSet^{I^{op}}_a$, the functor
\[K\otimes_I^a(-)\colon \cat{C}^{I}_a\longrightarrow \cat{C}^G\]
preserves equivalences of point-wise cofibrant objects.
\end{proposition}

\begin{proof}
The proof is the same as for the non-equivariant case of \cite[18.4]{hirsch}, using the equivariant adjunctions as in the proof of \ref{htpyinvhom}.
\end{proof}
The following result generalizes Villarroel's result \cite[6.1]{vilf}:

\begin{corollary}\label{htpyinvlimcolim} The functors
$\holim\colon \cat{C}^{I}_a\rightarrow \cat{C}^{G}$ and $\hocolim\colon \cat{C}^{I}_a\rightarrow \cat{C}^{G}$ preserve equivalences between fibrant $G$-diagrams and point-wise cofibrant $G$-diagrams respectively. 
\end{corollary}

\begin{proof}
Recall that homotopy limits and homotopy colimits are defined by cotensoring with $N(I/-)$ and tensoring with $N(-/I)^{op}$, respectively. By Proposition \ref{htpyinvtarget} it is enough to show that $N(I/-)$ is cofibrant in $sSet^{I}_a$ and $N(-/I)^{op}$ is cofibrant in $sSet^{I^{op}}_a$. This was shown in Example \ref{ex:ncof}.
\end{proof}

For an equivariant functor $F \colon I \to J$ between categories with $G$-actions $a$ and $b$ respectively define the \emph{homotopy left Kan extension} of a $G$-diagram $X$ in $\cat{C}^I_a$ by 
\[(\ho F_\ast X)_j = \hocolim (F/j \to I \stackrel{X}{\to} \cat{C})\]
with the induced $G$-structure. The usual homotopy colimit $\hocolim_I$ is the homotopy left Kan extension along the functor $I \to \ast$. Using the simplicial resolution $B(r^\ast r)X$ of Section \ref{sec:cofrepl} it is not hard to see that there is a natural isomorphism $\ho F_\ast X   \cong F_\ast(qX)$. 

\begin{lemma}(Transitivity of homotopy left Kan extensions)\label{lemma:trans}
  Let $F \colon I \to J$ and $F' \colon J \to K$ be equivariant functors between small categories with $G$-actions $a$, $b$ and $c$, respectively. If $X$ is a pointwise cofibrant object in $\cat{C}^I_a$ then the natural map
\[ \ho F'_\ast (\ho F_\ast X) \to \ho(F' \circ F)_\ast X\]
is a weak equivalence in $\cat{C}^K_c$. In particular, if $K=\ast$ then there is a weak equivalence
\[ \hocolim_J (\ho F_\ast X) \stackrel{\sim}{\to} \hocolim_I X.\]
\end{lemma}

\begin{proof}Since $X$ is pointwise cofibrant the diagram  $qX$ is cofibrant and so $\ho F_\ast X \cong F_\ast qX$ is cofibrant as well, since $F_\ast$ preserves cofibrancy. The functor $F'_\ast$ preserves weak equivalences between cofibrant objects, so the natural map $F'_\ast (q\ho F_\ast X) \to F'_\ast(\ho F_\ast X) $ is a weak equivalence.
The map in the lemma is the composite of the natural maps 
\[\ho F'_\ast(\ho F_\ast X) \iso F'_\ast(q\ho F_\ast X) \stackrel{\sim}{\too} F'_\ast(\ho F_\ast X) \iso  F'_\ast(F_\ast qX) \iso \ho (F' \circ F)_\ast X,\] 
where the second map is a weak equivalence by the discussion above.
\end{proof}


\subsection{Equivariant cofinality}\label{cofinalitysec}

Let $I$ and $J$ be categories with respective $G$-actions $a$ and $b$, $F\colon I\rightarrow J$ an equivariant functor, and $X\colon J\rightarrow \cat{C}$ a $G$-diagram.

We want to know when the canonical maps
\[\hocolim_IF^{\ast}X\longrightarrow \hocolim_JX \ \ \ \ \ \mbox{             and               }  \ \ \ \ \ \holim_{J}X\longrightarrow \holim_IF^{\ast}X\]
are equivalences in $\cat{C}^G$.
As in the non-equivariant setting, the categories $F/j$ and $j/F$ play a role in answering this question. For every object $j\in J$ these categories inherit a canonical action by the stabilizers group $G_j\leq G$ of $j$.
\begin{definition}\label{def:gcof}
The functor $F\colon I\rightarrow J$ is left (resp. right) cofinal if for every $j\in J$ the nerve of the category $F/j$ (resp. $j/F$) is weakly $G_j$-contractible.
\end{definition}

Notice that for $H\leq G_i$, the $H$-fixed points of the nerve of $F/j$ are isomorphic to the nerve of $(F/j)^H$. Therefore $F$ is left cofinal if and only if the fixed categories $(F/j)^{H}$ are contractible for all $H\leq G_i$, and similarly for right cofinality.

The following cofinality theorem is a generalization of \cite[1]{thevenaz-webb} and \cite[6.3]{vilf}.
\begin{theorem}\label{Gcof}
Let $\cat{C}$ be a $G$-model category, $F\colon I\rightarrow J$ be an equivariant functor, and  $X\in\cat{C}^{J}_b$ a $G$-diagram in $\cat{C}$. If $F$ is left cofinal and $X$ is fibrant, the canonical map
\[\holim_{J}X\longrightarrow \holim_IF^{\ast}X\]
is an equivalence in $\cat{C}^G$. Dually, if $F$ is right cofinal and $X$ is point-wise cofibrant, the map
\[\hocolim_IF^{\ast}X\longrightarrow \hocolim_JX\]
is an equivalence in $\cat{C}^G$.
\end{theorem}

\begin{proof}
We prove the part of the statement about left cofinality. The map $\holim_{J}X\rightarrow \holim_IF^{\ast}X$ factors as
\[map^b_J(NJ/(-),X) \iso map^b_J(NF/(-), X) \to map^a_I(NI/(-),F^\ast X).\]
The first map is a cotensor version of the $(F_\ast,F^\ast)$-adjunction isomorphism. It is equivariant and it is showed to be an isomorphism in \cite[19.6.6]{hirsch}. The second map is induced by the projection map $NF/(-)\rightarrow NJ/(-)$ which is an equivalence in  $sSet^{J}_b$, since for all $H\leq G$ and all object $j\in J^H$ both categories $F/j^H$ and $J/j^H$ are contractible ($J/j^H$ has a final object). Moreover, the $G$-diagrams $NJ/(-)$ and $NF/(-)$ are cofibrant in  $sSet^{J}_a$, by Example \ref{ex:ncof}. Therefore the induced map on mapping objects is an equivalence by the homotopy invariance of $map^b_J$ of Proposition \ref{htpyinvhom}.

\end{proof}

As an application of cofinality we prove a ``twisted Fubini theorem'' for homotopy colimits, describing the homotopy colimit of a $G$-diagram indexed over a Grothendieck construction. The classical version can be found in \cite[26.5]{CS}. Let $I$ be a category with $G$-action and $\Psi\in Cat^{I}_a$ a $G$-diagram of small categories. The Grothendieck construction $I\wr \Psi$ of the underlying diagram of categories inherits a $G$-action, defined on objects by
\[g\cdot \big(i,c\in Ob \Psi(i)\big)=\big(gi,g_\ast c\in \Psi(gi)\big)\]
and sending a morphism $(\alpha\colon i\rightarrow j,\gamma\colon \Psi(\alpha)(c)\rightarrow d)$ from $(i,c)$ to $(j,d)$ to the morphism
\[g\cdot\big(\alpha,\gamma\big)=\big(g\alpha\colon gi\rightarrow gj,\Psi(g\alpha)(gc)=g\Psi(\alpha)\stackrel{g\gamma}{\rightarrow} gd\big)\]
Now let $X\in \cat{C}^{I\wr\Psi}_a$ be a $G$-diagram in a $G$-model category $\cat{C}$. This induces a $G$-diagram $I\rightarrow \cat{C}$ defined at an object $i$ of $I$ by $\hocolim_{\Psi(i)}X|_{\Psi(i)}$, where $X$ is restricted along the canonical inclusion $\iota_i \colon \Psi(i)\rightarrow I\wr\Psi$. The $G$-structure is given by the maps
\[\hocolim_{\Psi(i)}X|_{\Psi(i)}\stackrel{g}{\rightarrow} \hocolim_{\Psi(i)}X|_{\Psi(gi)}\circ g\stackrel{g_\ast}{\rightarrow}\hocolim_{\Psi(gi)}X|_{\Psi(gi)}\]
where the first map is induced by the natural transformation of $\Psi(i)$-diagrams $X|_{\Psi(i)}\rightarrow X|_{\Psi(gi)}\circ g$ provided by the $G$-structure on $X$, and the second map is the canonical map induced by the functor on indexing categories $g\colon \Psi(i)\rightarrow \Psi(gi)$.

\begin{corollary}\label{Fubini}
For every point-wise cofibrant $G$-diagram $X\in \cat{C}^{I\wr\Psi}_a$ there is a natural equivariant weak equivalence 
\[\eta \colon \hocolim_{I}\hocolim_{\Psi(-)}  X|_{\Psi(-)} \stackrel{\simeq}{\to}  \hocolim_{I\wr\Psi}X.\]
\end{corollary} 

\begin{remark}
When $\cat{C}$ is the $G$-model category of spaces with the fixed point model structures and $X\colon I\wr\Psi\rightarrow Top$ is the constant one point diagram the corollary gives a $G$-equivalence
\[|N(I\wr\Psi)|\stackrel{\simeq}{\longrightarrow} \hocolim_{i\in I} |N\Psi(i)|\]
analogous to Thomason's theorem \cite{thomason}. Our proof is modeled on Thomason's proof.
\end{remark}

\begin{proof}[Proof of \ref{Fubini}]
Let $p \colon I \wr \Psi \to I$ be the canonical projection. We start by defining a zig-zag of equivalences
\[\hocolim_I\hocolim_{\Psi(-)} X|_{\Psi(-)} \stackrel{\lambda_1}{\leftarrow} \hocolim_I \ho p_\ast X \stackrel{\lambda_2}{\to} \hocolim_{I \wr \Psi} X,\]
where $\ho p_\ast$ denotes homotopy left Kan extension, and $\lambda_2$ is the equivalence of transitivity of homotopy left Kan extensions \ref{lemma:trans}.

For an object $i$ of $I$ define the functor $F_i \colon p/i \to \Psi(i)$ by $F_i(j,c, f \colon j \to i ) = \Psi(f)(c)$ on objects and on morphisms from $(j,c,f_0 \colon j \to i )$ to $(k,d,f_1 \colon k \to i )$ by 
\[F_i\big(h \colon j \to k, \alpha \colon \Psi(h)(c) \to d\big) = \Psi(f_1)(\alpha) \colon \Psi(f_0)(c) \to \Psi(f_1)(d)\]
The canonical functor $p/i \to I\wr\Psi$ used to define the homotopy left Kan extension $(\ho p_\ast X)_i$ factors as $p/i \stackrel{F_i}{\too} \Psi(i) \stackrel{\iota_i}{\too} I \wr \Psi$. This factorization induces a map $\gamma_i \colon (\ho p_\ast X)_i \to \hocolim_{\Psi(i)} X|_{\Psi(i)}$ which is natural in $i$ and is compatible with the $G$-structures and hence defines a map of $I$-indexed $G$-diagrams $\gamma \colon \ho p_\ast X \to \hocolim_{\Psi(-)} X|_{\Psi(-)}$. This induces the map
\[\lambda_1\colon \hocolim_I \ho p_\ast X\longrightarrow\hocolim_I\hocolim_{\Psi(-)} X|_{\Psi(-)}\]
in the zig-zag. Let us see that this is an equivalence.
For an object $c$ of $\Psi(i)$ the right fiber $c/F_i$ has a $(G_i)_c$-invariant initial object and is therefore contractible. It follows by cofinality \ref{Gcof} that the maps $\gamma_i$ are weak $G_i$-equivalences. By homotopy invariance of homotopy colimits the induced map $\lambda_1$ is a $G$-equivalence.

It remains to introduce the map $\eta \colon \hocolim_{I\wr\Psi}X \to \hocolim_{I}\hocolim_{\Psi(-)} X|_{\Psi(-)}$ from the statement, and compare it with the zig-zag. It is defined using the simplicial replacements from \S\ref{sec:cofrepl}. The iterated homotopy colimit $\hocolim_I\hocolim_{\Psi(-)} X_{\Psi(-)}$ is isomorphic to the realization of the simplicial $\cat{C}^G$-object
\[[p] \mapsto \coprod_{k_p \to \cdots \to k_0, i_p \stackrel{f_p}{\to} \cdots \stackrel{f_1}{\to} i_0} X_{(i_p,k_p)},\]
where the indexing strings of maps are in $N_p\Psi(i_p)^{op}$ and $N_pI^{op}$, respectively. The map $\eta$ in level $p$ maps a summand $X_{(i_p,k_p)}$ by the identity map to the summand of 
\[ \coprod_{\sigma \in N_p (I\wr \Psi)^{op}} X_{\sigma(p)},\]
indexed by the $p$-simplex $(i_p,k_p) \to (i_{p-1}, \Psi(f_p)(k_{p-1})) \to \cdots \to (i_0, \Psi(f_p \cdots f_1)(k_{0}))$ of $N (I\wr \Psi)^{op}$. Just as in Thomason's original proof there is a simplicial homotopy from $\eta \circ \lambda_2$ to $\lambda_1$ and it follows that $\eta$ is a weak equivalence (see in particular \cite[Lemma 1.2.5]{thomason}).

\end{proof}

\subsection{The Elmendorf theorem for \texorpdfstring{$G$}{G}-diagrams}

Let $\cat{C}$ be a cofibrantly generated model category with cellular fixed points, in the sense of \cite{GuiMay}. Then the category $\cat{C}^G$ of $G$-object admits the fixed point model structure, where weak equivalences and fibrations are the equivariant maps whose $H$-fixed points are weak equivalences and fibrations in $\cat{C}$, respectively, for every subgroup $H\leq G$. Let $\mathcal{O}_G$ be the orbit category of $G$, with quotient sets $G/H$ as objects and equivariant maps as morphisms. Elmendorf's theorem (see \cite{Marc}, \cite{Elmendorf}) describes a Quillen equivalence
\[L\colon\cat{C}^{\mathcal{O}^{op}_G}\rightleftarrows\cat{C}^G\colon R\]
where the diagram category $\cat{C}^{\mathcal{O}^{op}_G}$ has the projective model structure. In this section we prove an analogous result, giving a Quillen equivalence between the category of $G$-diagrams in $\cat{C}$ and a category of diagrams with the projective model structure.

Let $I$ be a small category with an action $a$ of $G$. For convenience we will consider the category of $G$-diagrams in $\cat{C}$ as the category $\cat{C}^{G\rtimes_a I}$ of diagrams indexed over the Grothendieck construction of the action (see \ref{lemma:eqdef}). The functor $a\colon G\rightarrow Cat$ induces a functor $\overline{a}\colon \mathcal{O}^{op}_G\rightarrow Cat$ that sends $G/H$ to the category $I^H$ of objects and morphisms of $I$ fixed by the $H$-action. We denote its Grothendieck construction by $\mathcal{O}^{op}_G\rtimes_{\overline{a}}I$. The inclusion functor $G\rightarrow \mathcal{O}^{op}_G$ that sends the unique object to $G/1$ induces a functor $G\rtimes_a I\rightarrow\mathcal{O}^{op}_G\rtimes_{\overline{a}}I$, which itself induces a restriction functor
\[L\colon \cat{C}^{\mathcal{O}^{op}_G\rtimes_{\overline{a}}I}\longrightarrow \cat{C}^{G\rtimes_{a}I}\]
Recall from \ref{exfixedptsmodel} that if the fixed point functors of $\cat{C}$ are cellular, the fixed point model structures on $\cat{C}^H$, for $H\leq G$, assemble into a $G$-model category.

\begin{theorem}\label{elmendorf}
Let $\cat{C}$ be a category such that the fixed points functors for the subgroups of $G$ are cellular. The functor $L\colon \cat{C}^{\mathcal{O}^{op}_G\rtimes_{\overline{a}}I}\rightarrow \cat{C}^{G\rtimes_{a}I}$ is the left adjoint of a Quillen equivalence
\[L\colon \cat{C}^{\mathcal{O}^{op}_G\rtimes_{\overline{a}}I}\rightleftarrows \cat{C}^{G\rtimes_{a}I}\colon R\]
where $\cat{C}^{G\rtimes_{a}I}$ has the model structure of \ref{modelstruct} and $\cat{C}^{\mathcal{O}^{op}_G\rtimes_{\overline{a}}I}$ has the projective model structure.
\end{theorem}

\begin{proof} The right adjoint sends a $G$-diagram $X$ in $\cat{C}^{G\rtimes_{a}I}\cong \cat{C}_{a}^I$ to the diagram $R(X)\colon \mathcal{O}^{op}_G\rtimes_{\overline{a}}I\rightarrow\cat{C}$ that sends an object $(G/H, i\in I^H)$ to
\[R(X)(G/H, i\in I^H)=X^{H}_i\]
In order to define $R(X)$ on morphisms, recall that the set of equivariant maps $G/K\rightarrow G/H$ is in natural bijection with $(G/H)^K$. A morphism in $\mathcal{O}^{op}_G$ from $(G/H, i)$ to $(G/K,j)$ is a pair $(z\in (G/H)^K,(\alpha\colon zi\rightarrow j)\in I^K)$, which is sent to the composite
\[X^{H}_i\stackrel{z}{\longrightarrow}X^{K}_{zi} \stackrel{\alpha^{K}_\ast}{\longrightarrow} X^{K}_j\]
A morphism $f\colon X\rightarrow Y$ in $\cat{C}_{a}^I$ is sent to the natural transformation with value $X_{i}^H\stackrel{f^{H}_i}{\longrightarrow}Y_{i}^H$ at the object $(G/H, i \in I^H)$. It is straightforward to see that $R$ is a right adjoint for $L$. The counit $LRX\rightarrow X$ is an isomorphism, and the unit at a diagram $Z$ of $\cat{C}^{\mathcal{O}^{op}_G\rtimes_{\overline{a}}I}$ is the natural transformation
\[\eta_Z\colon Z(G/H,i)\longrightarrow RL(Z)(G/H,i)=Z(G/1,i)^H\]
induced by the morphism $(H\in (G/H)^{1},\id_i)\colon(G/H,i)\rightarrow (G/1,i)$ of $\mathcal{O}^{op}_G\rtimes_{\overline{a}}I$. By definition of the fixed point model structure and of the model structure on $\cat{C}^{G\rtimes_{a}I}$, the right adjoint $R$ preserves and detects equivalences and fibrations. Thus the adjunction $(L,R)$ is a Quillen pair.

Since $R$ preserves and detects equivalences, $(L,R)$ is a Quillen equivalence precisely if the unit $\eta_Z\colon Z\rightarrow RL(Z)$ is an equivalence for all cofibrant objects $Z$ in $\cat{C}^{\mathcal{O}^{op}_G\rtimes_{\overline{a}}I}$. We prove this following the argument of \cite{Marc}. By cellularity of the fixed point functors $RL$ preserves pushouts along generating cofibrations and directed colimits along point-wise cofibrations. Thus it is enough to show that $\eta_Z$ is an isomorphism when $Z$ is a generating cofibrant object, that is, an object of the form
\[Z=\hom_{\mathcal{O}^{op}_G\rtimes_{\overline{a}}I}((G/H,i),-)\otimes c\]
for fixed objects $(G/H,i)$ of $\mathcal{O}^{op}_G\rtimes_{\overline{a}}I$ and $c$ of $\cat{C}$ cofibrant. For such a $Z$, the unit at an object $(G/K,j)$ is the top horizontal map of the commutative diagram
\[\xymatrix{\hom_{\mathcal{O}^{op}_G\rtimes_{\overline{a}}I}((G/H,i),(G/K,j))\otimes c\ar[r]^-{\eta}\ar[d]_{\cong}&
(\hom_{\mathcal{O}^{op}_G\rtimes_{\overline{a}}I}((G/H,i),(G/e,j))\otimes c)^K\ar[d]^{\cong}\\
\{(z\in (G/H)^K,\alpha\colon (zi\rightarrow j)\in I^K)\}\otimes c\ar[r]\ar@{=}[d]&\big(\left\{(z\in G/H,\alpha\colon (zi\rightarrow j)\in I)\right\}\otimes c\big)^K\ar@{=}[d]\\
\Lambda_{ij}^K\otimes c\ar[r]&(\Lambda_{ij}\otimes c)^K
}\]
where $\Lambda_{ij}$ is the set of pairs $(z\in G/H,\alpha\in zi\rightarrow j)$ with $K$ acting by left multiplication on $G/H$ and by the category action on the map to $j$ (notice that $j$ belongs to $I^K$). The bottom horizontal map is an isomorphism by the cellularity conditions on the $K$-fixed points functor.
\end{proof}

For the $G$-model category of spaces, the Elmendorf theorem gives a description of the fixed points of the homotopy limit of a $G$-diagram as a space of natural transformations of diagrams.

\begin{corollary}
For every $G$-diagram of spaces $X$ in $Top_{a}^I$, there is a natural homeomorphism of spaces
\[(\holim_IX)^G\cong Map_{Top^{\mathcal{O}^{op}_G\rtimes_{\overline{a}}I}}\big(
R(BI/(-)),R(X)\big)\]
where $R(X)\colon \mathcal{O}^{op}_G\rtimes_{\overline{a}}I\rightarrow Top$ has vertices $R(X)_{(G/H,i)}=X_{i}^H$.
\end{corollary}

\begin{proof}
The space $(\holim_IX)^G$ is by definition the mapping space from $BI/(-)$ to $X$ in $Top_{a}^I$. As the counit of the adjunction of the Elmendorf theorem is an isomorphism, there is a sequence of natural homeomorphisms
\[Map_{Top^{I}_a}\big(BI/(-),X\big)\cong Map_{Top^I_a}
\big(LR(BI/(-)),X\big)\cong Map_{Top^{\mathcal{O}^{op}_G\rtimes_{\overline{a}}I}}
\big(R(BI/(-)),R(X)\big)\]
\end{proof}


\section{Equivariant excision}\label{sec:exc}

We use the homotopy theory of $G$-diagrams developed earlier in the paper to set up a theory of $G$-excisive homotopy functors.

Classical excision is formulated using cartesian and cocartesian squares, and captures the behavior of  homology theories. Blumberg points out in \cite{Blumberg} that in the equivariant setting, squares of $G$-objects are not enough to capture the behavior of equivariant homology theories. In the rest of the paper we explain how to replace squares by cubical $G$-diagrams to fund a good theory of equivariant excision. We point out that this has already been achieved in \cite{Blumberg} in the category of based $G$-spaces. We prove in \ref{Glintop} that our approach and Blumberg's are equivalent in this category.

\subsection{Equivariant cubes and \texorpdfstring{$G$}{G}-excision}

If $J$ is a finite $G$-set, the poset category of subsets of $J$ ordered by inclusion $\mathcal{P}(J)$ has a canonical $G$-action, where a group element $g\in G$ sends a  subset $U\subset J$ to the set 
\[g\cdot U=\{g\cdot u \ | \  u\in U\}\]
Let $\cat{C}$ be a $G$-model category (cf. \ref{defGmodelstr}).

\begin{definition}
The category of $J$-cubes in $\cat{C}$ is the category of $G$-diagrams $\cat{C}^{\mathcal{P}(J)}_a$ for the action $a$ on $\mathcal{P}(J)$ described above.
\end{definition}
In order to define a homotopy invariant notion of (co)cartesian cubes, we need to make our homotopy (co)limits homotopy invariant. Given a cube $X\in \cat{C}^{\mathcal{P}(J)}_a$ let $FX$ denote a fibrant $J$-cube together with an equivalence $X\stackrel{\simeq}{\rightarrow} FX$. Similarly let $QX\stackrel{\simeq}{\rightarrow} X$ denote an equivalence with $QX$ point-wise cofibrant, that is, with $QX_U$ cofibrant in $\cat{C}^{G_U}$ for every $U\in \mathcal{P}(J)$.
\begin{remark}
To find a replacement $FX$ one can simply use the fibrant replacement in the model category $\cat{C}^{\mathcal{P}(J)}_a$. Similarly, a cofibrant replacement $QX$ in $\cat{C}^{\mathcal{P}(J)}_a$ is in particular point-wise cofibrant by \ref{ptwisecof}. However, for a given cube one can often find a more explicit point-wise cofibrant replacement that is not necessarily cofibrant in $\cat{C}^{\mathcal{P}(J)}_a$ (see e.g. \ref{suspcube} and \ref{wedgecube} below). For example, if a functorial cofibrant replacement $Q$ in $\cat{C}$ lifts to a cofibrant replacement in $\cat{C}^H$ for every $H\leq G$, the diagram $QX$ is point-wise cofibrant.
\end{remark}

For an object $i$ of $I$ fixed by the $G$-action, let $I\backslash i$ be the full subcategory of $I$ with objects different from $i$. The action on $I$ restricts to $I\backslash i$, and the inclusion functor $\iota_i\colon I\backslash i\to I$ is equivariant.

\begin{definition}
Let $\cat{C}$ be a $G$-model category and $J$ a finite $G$-set. A $J$-cube $X\in \cat{C}^{\mathcal{P}(J)}_a$ is homotopy cocartesian if the canonical map 
\[\hocolim_{\mathcal{P}(J)\backslash J} \iota_{J}^\ast QX\longrightarrow QX_J\stackrel{\simeq}{\rightarrow}X_J\]
is an equivalence in $\cat{C}^G$.
Dually, $X\in \cat{C}^{\mathcal{P}(J)}_a$ is homotopy cartesian if the canonical map 
\[X_{\emptyset}\stackrel{\simeq}{\rightarrow}FX_\emptyset\longrightarrow\holim_{\mathcal{P}(J)\backslash \emptyset}\iota_{\emptyset}^{\ast} FX\]
is an equivalence in $\cat{C}^G$.
\end{definition}

\begin{example}\label{suspcube} Let $J$ be a finite $G$-set, and $J_+$ be the $G$-set $J$ with a disjoint fixed base point. For a cofibrant object $c\in\cat{C}^G$ define a $J_+$-cube $S^Jc$ with verticies
\[(S^Jc)_U=\left\{\begin{array}{lll}
c & ,U=\emptyset\\
C^Uc & ,U\lneq J_+\\
\Sigma^Jc & ,U= J_+
\end{array}\right.\]
Here $\Sigma^Jc=\Sigma^{\tilde{J_+}}c$ is the suspension by the permutation representation of $J$ defined in \ref{defloopsusp}, and $C^Uc$ denotes the $U$-iterated cone
\[C^Uc=\hocolim_{\mathcal{P}(U)}\left(S\longmapsto\left\{\begin{array}{ll}c&\mbox{if }S=\emptyset\\
\ast&\mbox{otherwise}
\end{array}\right.\right)\simeq\ast\]
Since $c$ is cofibrant, $S^Jc$ is point-wise cofibrant.
Let us prove that it is homotopy cocartesian. Its restriction to $\mathcal{P}(J_+)\backslash J_+$ is the cofibrant replacement $q$ of Theorem \ref{thm:qxcofrepl} for the diagram $\sigma^J c\colon \mathcal{P}(J_+)\backslash J_+\rightarrow\cat{C}$ with $(\sigma^J c)_\emptyset=c$ and the terminal object at the other vertices. Since homotopy colimits and colimits agree on cofibrant objects (by the homotopy invariance of $\otimes^a_I$), the canonical map from  the homotopy colimit factors as the equivalence
\[\hocolim_{\mathcal{P}(J_+)\backslash J_+}S^Jc=\hocolim_{\mathcal{P}(J_+)\backslash J_+}q(\sigma^J c)\stackrel{\simeq}{\rightarrow}\colim_{\mathcal{P}(J_+)\backslash J_+}q(\sigma^J c)\cong \hocolim_{\mathcal{P}(J_+)\backslash J_+}\sigma^Jc=\Sigma^J c
\]
\end{example}

\begin{example}\label{wedgecube} Suppose that $\cat{C}$ has a zero object $\ast$ and denote the coproduct by $\bigvee$. Let $c$ be a cofibrant object of $\cat{C}^G$ and $J$ a finite $G$-set. Define a $J$-cube $W^Jc$ with vertices
\[(W^Jc)_U=\left\{\begin{array}{lll}
\bigvee_Jc & ,U=\emptyset\\
c & ,|U|=1\\
\ast& ,|U|\geq 2
\end{array}\right.\]
with initial map $(W^Jc)_\emptyset=\bigvee_Jc\rightarrow c=
(W^Jc)_{\{j\}}$ the pinch map that collapses every wedge component different from $j$. This has a $G$-structure defined by the action on $\bigvee_Jc$ on the initial vertex, and by the action maps $g\colon (W^Jc)_{\{j\}}=c\rightarrow c=(W^Jc)_{\{gj\}}$. The cube $W^Jc$ is homotopy cocartesian, that is, its homotopy colimit over $\mathcal{P}(J)\backslash J$ is equivalent in $\cat{C}^G$ to the zero object. To see this, we replace $W^Jc$ by the equivalent cube
\[(\overline{W}c)_U=\left\{\begin{array}{lll}
\bigvee\limits_Jc& ,U=\emptyset\\
 c\bigvee\limits_{J\backslash j}Cc & , U=\{j\}\\
\bigvee\limits_{J}Cc& , |U|\geq 2
\end{array}\right.\]
where $Cc$ is the one-fold cone $Cc=\hocolim(c\rightarrow \ast)$ and the non-identity maps of the diagram are all induced by cone inclusions $c\rightarrow Cc$. The $G$-structure is defined similarly as before, by permuting the wedge components. The cube $\overline{W}c$ is cofibrant, since the latching maps are all cofibrations (see \ref{reedy}). As homotopy colimits preserve equivalences of point-wise cofibrant diagrams we get
\[\hocolim_{\mathcal{P}(J)\backslash J}W^Jc\stackrel{\simeq}{\leftarrow}\hocolim_{\mathcal{P}(J)\backslash J}\overline{W}c\stackrel{\simeq}{\rightarrow}\colim_{\mathcal{P}(J)\backslash J}\overline{W}c\cong \bigvee_JCc\]
This is contractible since $\bigvee_J$ is a left Quillen functor and therefore preserves equivalences of cofibrant objects.
\end{example}

We use homotopy cartesian and cocartesian $G_+$-cubes to express equivariant excision for functors between $G$-model categories $\cat{C}$ and $\cat{D}$. We shall consider functors for which we can express compatibility conditions with the model structures on $\cat{C}^H$ and $\cat{D}^H$ for every subgroup $H\leq G$. These are functors $\Phi\colon \cat{C}\rightarrow \cat{D}^G$. Such a functor $\Phi$ induces a functor $\Phi_\ast\colon \cat{C}^{I}_a\rightarrow \cat{D}^{I}_a$ for any category with $G$-action $I$. The $G$-structure on $\Phi_\ast(X)=\Phi\circ X$ is defined by the maps
\[\Phi(X_i)\stackrel{g}{\longrightarrow}\Phi(X_i)\stackrel{\Phi(g)}{\longrightarrow} \Phi(X_{gi})\]
Since each map $\Phi(g)$ is $G$-equivariant $\Phi(g)g=g\Phi(g)$. For $I=\ast$ the trivial category this functor is the classical extension $\Phi_\ast\colon \cat{C}^G\rightarrow \cat{D}^G$. Similarly, the functor $\Phi\colon \cat{C}\rightarrow \cat{D}^H$ obtained by restricting the $G$-action to $H\leq G$, extends to a functor $\Phi_\ast\colon \cat{C}^H\rightarrow \cat{D}^H$.

\begin{definition}
We call $\Phi\colon \cat{C}\rightarrow \cat{D}^G$ a homotopy functor if for every subgroup $H\leq G$ the extended functor $\Phi_\ast\colon \cat{C}^H\rightarrow \cat{D}^H$ preserves equivalences of cofibrant objects. In particular the induced functor $\Phi_\ast\colon \cat{C}^{I}_a\rightarrow \cat{D}^{I}_a$ preserves equivalences of point-wise cofibrant $G$-diagrams.
\end{definition}

\begin{remark} The following are all examples of functors $\cat{C}^G\rightarrow \cat{D}^G$ that are extensions of homotopy functors $\cat{C}\rightarrow \cat{D}^G$.
\begin{itemize}
\item The identity functor $\cat{C}^G\rightarrow \cat{C}^G$,
\item For a fixed pointed $G$-space $K$, the functors $K\wedge (-),Map_\ast(K,-)\colon Top_{\ast}^G\rightarrow Top_{\ast}^G$,
\item For a fixed orthogonal $G$-spectrum $E$ the functor $E\wedge (-)\colon Top_{\ast}^G\rightarrow (\Sp^O)^G$.
\end{itemize}
An example of a functor $\cat{C}^G\rightarrow \cat{D}^G$ that is not the extension of a functor $\cat{C}\rightarrow \cat{D}^G$ is the functor $(-)/G\colon Top^G\rightarrow Top^G$ that sends a $G$-space to its orbit space with trivial $G$-action.
\end{remark}

\begin{definition}\label{defGexc} Let $\cat{C}$ and $\cat{D}$ be $G$-model categories.
A homotopy functor $\Phi\colon \cat{C}\rightarrow \cat{D}^G$ is called $G$-excisive if the induced functor $\Phi_\ast\colon \cat{C}^{\mathcal{P}(G_+)}_a\rightarrow \cat{D}^{\mathcal{P}(G_+)}_a$ sends homotopy cocartesian $G_+$-cubes to homotopy cartesian $G_+$-cubes. If $\cat{C}$ and $\cat{D}$ are pointed, $\Phi$ is called $G$-linear if it is $G$-excisive and $\Phi(\ast)$ is equivalent to the zero object in $\cat{D}^G$.
\end{definition}

The choice of indexing the cubes on the $G$-set $G_+$ seems arbitrary at first sight. We justify and explain this choice, including the extra basepoint added to $G$, in \ref{fromGtoany} and \ref{nobasepoint} below.

\begin{example}\label{exlin}The following are examples of $G$-linear homotopy functors, as we will see later in the paper.
\begin{itemize}
\item
Let $M$ be an abelian group with additive $G$-action. Consider the homotopy functor $M(-)\colon sSet_\ast\rightarrow sSet^{G}_\ast$ that sends a simplicial set $Z$ to
\[M(Z)_n=\bigoplus_{z\in Z_n}M z/M\ast\]
where $G$ acts diagonally on the direct summands.
We show in \ref{Gconf} that this functor is $G$-linear, and explain how this is related to the equivariant Eilenberg-MacLane spectrum $HM$ being a fibrant orthogonal $G$-spectrum. The homotopy groups of the extension of $M(-)$ to $sSet^{G}_\ast$ are Bredon cohomology of the Mackey functor $H\mapsto M^H$.
\item For a fixed orthogonal $G$-spectrum $E$ in $(\Sp^O)^G$, the homotopy functor $E\wedge(-)\colon Top_{\ast}\rightarrow (\Sp^{O})^G$ is $G$-linear (see \ref{classlintospec}). The stable homotopy groups of the extension of $E\wedge(-)$ to pointed $G$-spaces is the equivariant cohomology theory associated to $E$.
\item The inclusion of spectra with trivial $G$-action $\Sp^O\rightarrow (\Sp^O)^G$ (which extends to the identity on $G$-spectra) is $G$-linear (see \ref{Gcartcocart}).
\end{itemize}
\end{example}

The next result shows that our choice of indexing the cubes on the $G$-set $G_+$ in the definition of $G$-excision plays a minor role, and we could equivalently have indexed the cubes on transitive $G$-sets with disjoint basepoints.

\begin{proposition}\label{fromGtoany} A homotopy functor $\Phi\colon \cat{C}\rightarrow \cat{D}^G$  is $G$-excisive if and only if the induced functor $\Phi_\ast\colon \cat{C}^{\mathcal{P}(G/H_+)}_a\rightarrow \cat{D}^{\mathcal{P}(G/H_+)}_a$ sends homotopy cocartesian $G/H_+$-cubes to homotopy cartesian $G/H_+$-cubes, for every subgroup $H\leq G$.
\end{proposition}

\begin{remark}\label{squaresinG} Setting $H=G$ in \ref{fromGtoany} we see that $\Phi_\ast\colon \cat{C}^{\mathcal{P}(1_+)}_a\rightarrow \cat{D}^{\mathcal{P}(1_+)}_a$ sends cocartesian squares in $\cat{C}^G$ to cartesian squares in $\cat{D}^G$. That is, if $\Phi$ is $G$-excisive then the induced functor $\Phi_\ast\colon \cat{C}^G\rightarrow \cat{D}^G$ is excisive in the classical sense.
\end{remark}

\begin{proof}[Proof of \ref{fromGtoany}] The ``if''-part of the statement is trivial. For the ``only if''-part, let $H$ be a subgroup of $G$ and consider the projection map $p\colon G_+\rightarrow G/H_+$. As part of a broader discussion on how to calculate homotopy limits and colimits of punctured cubes, we show in \ref{surjmapcubes} and \ref{surjmapcubesco} that the induced restriction functor $p^\ast\colon \cat{C}^{\mathcal{P}(G/H_+)}_a\rightarrow \cat{C}^{\mathcal{P}(G_+)}_a$ preserves homotopy cocartesian cubes and detects homotopy cartesian cubes. Therefore, given a homotopy cocartesian cube $X$ in $\cat{C}^{\mathcal{P}(G/H_+)}_a$, the cube $p^{\ast}X$ in $\cat{C}^{\mathcal{P}(G_+)}_a$ is homotopy cocartesian, and by $G$-excision of $\Phi$ the cube $\Phi_\ast(p^{\ast}X)=p^{\ast}\Phi_\ast(X)$ is homotopy cartesian in $\cat{D}^{\mathcal{P}(G_+)}_a$. As $p^\ast$ detects homotopy cocartesian cubes, $\Phi_\ast(X)$ is homotopy cartesian in $\cat{D}^{\mathcal{P}(G/H_+)}_a$.
\end{proof}

\begin{remark}\label{nobasepoint} The basepoint added to $G$ in the definition of $G$-excision \ref{defGexc} has the role of combining in a single condition the behavior of $\Phi\colon\cat{C}\rightarrow\cat{D}^G$ on squares and on $G$-cubes. We already saw (\ref{squaresinG}) that if $\Phi$ is $G$-excisive it sends homotopy cocartesian squares to homotopy cartesian squares. It turns out that $\Phi_\ast\colon\cat{C}^{\mathcal{P}(G/H)}_a\rightarrow \cat{D}^{\mathcal{P}(G/H)}_a$ also turns homotopy cocartesian $G/H$-cubes into homotopy cartesian ones. This can be proved by extending a $G/H$-cube to a $G/H_+$-cube by means of the functor $p\colon \mathcal{P}(G/H_+)\rightarrow \mathcal{P}(G/H)$ that intersects a subset with $G/H$, with a proof analogous to \ref{fromGtoany}. Conversely, similar techniques show that if $\Phi\colon\cat{C}\rightarrow\cat{D}^G$ turns homotopy cocartesian squares and $G$-cubes into homotopy cartesian ones, it is $G$-excisive.
\end{remark}

\begin{remark} $G$-linearity is  hereditary with respect to taking subgroups, under a mild assumption on the $G$-model category $\cat{D}$. That is to say, if $\Phi$ is $G$-linear it is also $H$-linear for every subgroup $H$ of $G$. The proof we suggest requires a surprizing amount of machinery and it is given in \ref{GlinHlin} as a corollary of a higher Wirthm\"{u}ller isomorphism theorem. It is still unknow to the authors if in the unpointed case $G$-excision satisfies a similar property.
\end{remark}

\begin{proposition}\label{wedgesintoprod} Let $\cat{C}$ and $\cat{D}$ be pointed $G$-model categories, and $\Phi\colon \cat{C}\rightarrow \cat{D}^G$  be a $G$-linear homotopy functor. For any finite $G$-set $J$ and any cofibrant $G$-object $c\in\cat{C}^G$ the canonical map
\[\Phi(\bigvee_Jc)\longrightarrow \prod_JF\Phi(c)\]
is an equivalence in $\cat{D}^G$.
\end{proposition}

\begin{proof}
First assume that $J=1_+$ with trivial $G$-action. The square $Vc$
\[\xymatrix{c\vee c\ar[r]^{p_+}\ar[d]_{p_1}&c\ar[d]\\
c\ar[r]&\ast
}\]
in $\cat{C}^G$ is homotopy cocartesian (cf. \ref{wedgecube}). By \ref{squaresinG} its image $\Phi(Vc)$ is homotopy cartesian, that is, the map 
\[\Phi(c\vee c)\stackrel{\simeq}{\rightarrow}F\Phi(c\vee c)\rightarrow \holim_{\mathcal{P}(1_+)\backslash\emptyset}F\Phi(Vc)\cong F\Phi(c)\times F\Phi(c)\]
is a weak equivalence in $\cat{D}^G$, with diagonal action on the target. By induction, the map of the statement is an equivalence for every $J$ with trivial $G$-action. Given a finite $G$-set $J$, decompose it as disjoint union of transitive $G$-sets $J=\coprod_{z\in G\backslash J}z$. The map of the statement decomposes as
\[\Phi(\bigvee_Jc)=\Phi(\bigvee_{z\in G\backslash J}\bigvee_{z}c)\stackrel{\simeq}{\longrightarrow} \prod_{z\in G\backslash J}F\Phi(\bigvee_z c)\longrightarrow\prod_{z\in G\backslash J}\prod_z F\Phi(c)=\prod_JF\Phi(c)\]
with the first map an equivalence as the action on the quotient $G\backslash J$ is trivial. Therefore it is enough to show that the map is an equivalence for $J=G/H$ a transitive $G$-set.
 
Consider the $G/H_+$-cube $Wc$ with vertices
\[(Wc)_U=\left\{\begin{array}{lll}
\bigvee_{G/H}c& ,U=\emptyset\\
c & ,U=\{j\neq +\}\\
\ast&, \mbox{otherwise}
\end{array}\right.\]
It is homotopy cocartesian by an argument completely similar to \ref{wedgecube}.
By \ref{fromGtoany} the cube $\Phi(Wc)$ is homotopy cartesian, that is, the canonical map
\[\Phi(\bigvee_{G/H}c)\rightarrow \holim_{\mathcal{P}(G/H_+)\backslash\emptyset}F\Phi(Wc)\cong \prod_{G/H} F\Phi(c)\]
is an equivalence in $\cat{D}^G$.
\end{proof}

\begin{remark}
In this equivariant setting $G_+$-cubes (or equivalently $J_+$-cubes for $J$ transitive) play the role that squares play in the classical theory. The equivariant analogue of $n$-cubes should be cubes indexed on $G$-sets with $n$ distinct $G$-orbits and a disjoint basepoint. Following \cite{calcII}, the behavior of $\Phi$ on these cubes should be related to higher order $G$-excision. This will be the subject of a later article.
\end{remark}


\subsection{The generalized Wirthm\"{u}ller isomorphism theorem}

Let $\cat{C}$ be a bicomplete category, and $G$ a finite group. We recall from \S\ref{secGmod} that a finite set $K$ with commuting left $H'$-action and right $H$-action induces an adjunction
\[K\otimes_H(-)\colon \cat{C}^H\rightleftarrows  \cat{C}^{H'}\colon \hom_{H'}(K,-)\]
Let $K^\ast$ be the set $K$ with left $H$-action and right $H'$-action defined by $h\cdot k\cdot h'=(h')^{-1}\cdot k\cdot h^{-1}$.
If $\cat{C}$ has a zero-object $\ast$ and if the actions on $K$ are free, a functor $\Phi\colon\cat{C}\rightarrow \cat{D}^G$ induces a natural transformation
\[\eta\colon \Phi(K\otimes_H(-))\longrightarrow\hom_{H}(K^\ast,\Phi(-))\]
of functors $\cat{C}^H\rightarrow\cat{D}^{H'}$. The map $\eta_c$ is the image by the composition
\[\begin{array}{ll}\cat{C}^{H}((K^{\ast}\times_{H'}K)\otimes_Hc,c)\stackrel{\Phi}{\rightarrow}\cat{D}^{H}(\Phi((K^{\ast}\times_{H'}K)\otimes_Hc),\Phi(c))\rightarrow
\\
\cat{D}^{H}(K^{\ast}\otimes_{H'}\Phi(K\otimes_Hc),\Phi(c))
\iso \cat{D}^{H'}(\Phi(K\otimes_Hc),\hom_{H}(K^\ast,\Phi(c)))
\end{array}\]
of the map $\bigvee_{K^{\ast}\times_{H'}K}c\rightarrow c$ defined by $h\colon c\rightarrow c$ on a $(k, k')$-component with $k'h=k$, and by the trivial map $c\rightarrow\ast\rightarrow c$ otherwise. Notice that since the $H$-action is free there is at most one $h$ for which $k'h=k$.

\begin{example}
Suppose that $K=G=H'$ with left $G$-multiplication and right $H$-multiplication. Sending an element to its inverse defines a $H$-$G$-equivariant isomorphism between $G^\ast$ and $G$ with left $H$-multiplication and right $G$-multiplication. We saw in \ref{adjforget} that the forgetful functor $\cat{C}^G\rightarrow \cat{C}^H$ is right adjoint to $G\otimes_H(-)$ and left adjoint to $\hom_H(G^{\ast},-)$. The map $\eta$ for the identity functor is the standard map
\[G\otimes_H(-)\longrightarrow\hom_H(G^{\ast},-)\]
which in the case of spectra is the classical Wirthm\"{u}ller isomorphism map. In \ref{Gcartcocart} we apply \ref{linandwedges} below to recover the Wirthm\"{u}ller isomorphism theorem for $G$-spectra.
\end{example}

\begin{theorem}\label{linandwedges} 
Let $\cat{C}$ and $\cat{D}$ be pointed $G$-model categories, and suppose that $K$ admits an $H'$-$H$-equivariant map to $G$, this happens e.g. if $K=G$. For every $G$-linear homotopy functor $\Phi\colon \cat{C}\rightarrow \cat{D}^G$ and every object  $c$ in $\cat{C}^H$ the composite
\[\Phi(K\otimes_Hc)\stackrel{\eta}{\longrightarrow}\hom_{H}(K^\ast,\Phi(c))\longrightarrow\hom_{H}(K^\ast,F\Phi(c))\]
is an equivalence in $\cat{D}^{H'}$, where $\Phi(c)\stackrel{\simeq}{\rightarrow} F\Phi(c)$ is a fibrant replacement of $\Phi(c)$ in $\cat{D}^{H}$.

In particular, if the right Quillen functor $\hom_{H}(K^\ast,-)$ preserves all weak equivalences, the map $\eta\colon \Phi(K\otimes_Hc)\rightarrow\hom_{H}(K^\ast,\Phi(c))$ is a weak equivalence for any $c\in \cat{C}^H$.
\end{theorem}

\begin{proof} We express the map of the statement as a canonical map into the homotopy limit of a punctured cube, and we use the $G$-linearity of $\Phi$ to conclude that the map is an equivalence. For this we will compare the source and target of $\eta$ with an indexed coproduct and product, respectively.

Choose a section $s_G\colon G/H\rightarrow G$ and an $H'$-$H$-equivariant map $\phi\colon K\rightarrow G$. These choices give a commutative diagram (of sets)
\[\xymatrix{K\ar[r]^{\phi}\ar[d]^{\pi_K}&G\ar[d]_{\pi_G}\\
K/H\ar@/^1pc/[u]^{s_K}\ar[r]_{\overline{\phi}}&G/H\ar@/_1pc/[u]_{s_G}
}\]
where $s_K(kH):=k\cdot (\phi(k)^{-1}\cdot s_G\pi_G \phi (k))$ is a section for $\pi_K$, satisfying the relation $\phi s_K=s_G\overline{\phi}$. This gives a map $\gamma\colon H'\times K/H\rightarrow H$ defined by
\[\gamma(h',z)=s_G(h'\overline{\phi}(z))^{-1}\cdot h'\cdot s_G\overline{\phi}(z)\]
which we use to define two functors $\bigvee_{K/H}(-)\colon \cat{C}^H\rightarrow \cat{C}^{H'}$ and $\prod_{K/H}(-)\colon \cat{D}^H\rightarrow \cat{D}^{H'}$.
These send objects $c$ and $d$ to the coproduct $\bigvee_{K/H}c$ and product $\prod_{K/H}d$, respectively, with $H'$-actions\footnote{For convenience we only spell these actions out in the case that the objects of $\cat{C}$ have ``elements''.} \[h'\cdot(z,x)=(h'z,\gamma(h',z)\cdot x) \ \ \ \ \ \mbox{and}\ \ \ \ \ (h'\cdot \underline{y})_z=\gamma(h',z)\cdot y_{(h')^{-1}z}, \ \mbox{respectively.}\ \]
There is a commutative diagram of natural transformations
\[\xymatrix{\Phi(\bigvee_{K/H}c)\ar[d]_{\Phi(s_K\otimes\id_c)}^{\cong }\ar[r]&\bigvee_{K/H}\Phi(c)\ar[d]_{s_K\otimes\id_{\Phi(c)}}^{\cong }\ar[r]&\prod_{K/H}\Phi(c)\ar[d]_{\cong}^{(-)\circ s_K}\\
\eta\colon\Phi(K\otimes_Hc)\ar[r]&K\otimes_H\Phi(c)\ar[r]& \hom_H(K^\ast,\Phi(c))
}\]
The top right horizontal map is the canonical map from the coproduct to the product. The first two vertical maps are induced by the composite
\[s_K\otimes\id\colon \bigvee_{K/H}c=K/H\otimes c\rightarrow K\otimes c\twoheadrightarrow K\otimes_H c.
\] 
It is an isomorphism with inverse $(k,x)\mapsto(\pi_Kk,(s_G\pi_G\phi(k))^{-1}\phi(k)\cdot x)$. The right vertical map $(-)\circ s_K$ is defined dually and it is also an isomorphism. We can therefore equivalently study the top composition $\Phi(\bigvee_{K/H}c)\rightarrow\prod_{K/H}\Phi(c)$.

Consider the $K/H_+$-cube $Wc \colon \mathcal{P}(K/H_+)\rightarrow \cat{C}$ defined by
\[(Wc)_{S}=\left\{\begin{array}{lll}\bigvee_{K/H}c &, S=\emptyset\\
c&, |S|=1, S\neq \{+\}\\
\ast& , |S|\geq 2\mbox{ or } S= \{+\}
\end{array}\right.\]
with initial map $\bigvee_{K/H}c\rightarrow c=(Wc)_{\{z\}}$ the pinch map that collapses all the wedge components not indexed by $\{z\}$. The structure maps $c=(Wc)_z\rightarrow (Wc)_{h'z}=c$ are defined by action by $\gamma(h',z)\in H$.

The cube $Wc$ is homotopy cocartesian. Indeed, if $Q_{H}c\stackrel{\simeq}{\rightarrow} c$ is a cofibrant replacement of $c$ in $\cat{C}^H$, the cube $WQ_Hc$ is point-wise cofibrant with homotopy colimit over $\mathcal{P}(K/H_+)\backslash K/H_+$ contractible (see \ref{wedgecube}).
Let $\Phi(Wc)\stackrel{\simeq}{\rightarrow} F\Phi(Wc)$ be a fibrant replacement of $\Phi(Wc)$. By linearity of $\Phi$, the canonical map
\[\Phi(\bigvee_{K/H}c)\stackrel{\simeq}{\longrightarrow}\holim_{\mathcal{P}(K/H_+)\backslash\emptyset}F\Phi(Wc)\cong\prod_{K/H}F\Phi(c)\]
is an equivalence in $\cat{D}^{H'}$. This proves the first part of the theorem.

Moreover, the map above fits into a commutative diagram
\[\xymatrix{
\Phi(\bigvee_{K/H}c)\ar[r]\ar[dr]_{\simeq}&\prod_{K/H}\Phi(c)\ar[d]\\
&\prod_{K/H}F\Phi(c)
}\]
where the right vertical map is an equivalence if $\hom_{H}(K^\ast,-)$ (and therefore $\prod_{K/H}(-)$) preserves weak equivalences. 
\end{proof}

\begin{corollary}
If the trivial action inclusion functor $\cat{C}\rightarrow \cat{C}^G$ is $G$-linear, the left and right adjoints to the evaluation functor $\ev_i\colon \cat{C}^{I}_a\rightarrow \cat{C}^{G_i}$ are naturally equivalent on fibrant objects for every $i\in I$.
\end{corollary}

\begin{proof}
We saw in \ref{adjevi} that the left adjoint $F_i\colon \cat{C}^{G_i}\rightarrow \cat{C}^{I}_a$ has $j$-vertex
\[(F_ic)_j=K_{ji}\otimes_{G_i}c\]
where $K_{ji}=\hom_{G\rtimes_a I}(i,j)$ projects $G_j$-$G_i$-equivariantly to $G$. Similarly the right adjoint has $j$-vertex
\[(R_ic)_j=\hom_{G_i}(K^{\ast}_{ji},c)\]
and \ref{linandwedges} provides a natural equivalence from $F_i$ to $R_i$.
\end{proof}

We give a ``higher version'' of the Wirthm\"{u}ller isomorphism theorem, that compares the left and the right adjoints of the functor on $J$-cubes that restricts the action to a subgroup $H$ of $G$.
Given a $G$-set $J$, let $J|_H$ be the $H$-set obtained by restricting the $G$-action to $H$. The poset category with $H$-action $\mathcal{P}(J|_H)$ is the category $\mathcal{P}(J)$ with the restricted action $a|_H$. There is a forgetful functor $\cat{C}^{\mathcal{P}(J)}_a\rightarrow \cat{C}^{\mathcal{P}(J|_H)}_{a|_H}$ that restricts the $G$-structure to a $H$-structure. It has both a left and a right adjoint, that we denote respectively $L^J$ and $R^J$. This can easily be seen with the description of $G$-diagrams as diagrams on a Grothendieck construction of \ref{lemma:eqdef}, as the restriction functor above corresponds to restriction along the inclusion $\iota\colon H\rtimes_{a|_H} \mathcal{P}(J|_H)\rightarrow G\rtimes_a\mathcal{P}(J)$. The following result specializes to theorem \ref{linandwedges}  for $K=G$ when $J$ is the empty $G$-set.

\begin{theorem}\label{higherwirth}
For every $G$-linear homotopy functor $\Phi\colon \cat{C}\rightarrow \cat{D}^G$ and every $J|_H$-cube $X\in \cat{C}^{\mathcal{P}(J|_H)}_{a|_H}$, there is an equivalence of $J$-cubes
\[\Phi L^J(X)\stackrel{\eta}{\longrightarrow}R^J\Phi(X)\longrightarrow R^JF\Phi(X)\]
where $\Phi(X)\stackrel{\simeq}{\rightarrow} F\Phi(X)$ is a fibrant replacement of $\Phi(X)$ .
\end{theorem}

\begin{proof}
Let us describe the left adjoint $L^J$ explicitly, by calculating the left Kan extension of $X$ along $\iota\colon H\rtimes_{a|_H} \mathcal{P}(J|_H)\rightarrow G\rtimes_a\mathcal{P}(J)$. By definition this has values
\[L^J(X)_U=\colim\big(\iota/_U\rightarrow H\rtimes_{a|_H} \mathcal{P}(J|_H)\stackrel{X}{\longrightarrow}\cat{C}\big).\]
The over category $\iota/_U$ is the poset with objects $(g\in G,A\in\mathcal{P}(g^{-1}U))$, and a unique morphism $(g,A)\rightarrow (g',A')$ whenever $g(g')^{-1}$ belongs to $H$ and $g(g')^{-1}A\subset A'$. This can be written as the disjoint union of categories
\[\iota/_U=\coprod_{z\in G/H}(Ez\wr\Psi_z)\]
where $Ez$ is the translation category of the right $H$-set $z$ (see \ref{adjevi}) and $Ez\wr\Psi_z$ is the Grothendieck construction of the functor $\Psi_z\colon Ez\rightarrow Cat$ that sends $g\in G/H$ to the category $\mathcal{P}(g^{-1}U)$. Hence the left Kan extension $L^J(X)$ is naturally isomorphic to
\[L^J(X)_U \cong \bigvee_{z\in G/H}\colim_{(g,A)\in Ez\wr\Psi_z}X_A\cong \bigvee_{z\in G/H}\colim_{g\in Ez}\colim_{A\in\mathcal{P}(g^{-1}U)}X_A \iso \bigvee_{z\in G/H}\colim_{g\in Ez}X_{g^{-1}U}
\]
Here the first isomorphism is the Fubini theorem for colimits (see e.g. \cite[40.2]{CS}, as it is an isomorphism it is enough to see that it is equivariant). The last map is an isomorphism is because $g^{-1}U$ is a terminal object in $\mathcal{P}(g^{-1}U)$. A choice of section $s\colon G/H\rightarrow G$ gives a further identification
\[L^J(X)_U\cong \bigvee_{z\in G/H}X_{s(z)^{-1}U}\]
Chasing through the isomorphisms one can see that the $G$-structure is given by the maps
\[g\colon X_{s(z)^{-1}U}\stackrel{s(gz)^{-1}gs(z)}{\longrightarrow}X_{s(gz)^{-1}gU}\]

The same choice of section gives a similar identification for the right adjoint
\[R^J(X)_U\cong \prod_{z\in G/H}X_{s(z)^{-1}U}\]
A $G/H_+$-cube argument completely analogous to \ref{linandwedges} shows that the inclusion of wedges into products induces a $G$-equivalence $\Phi L^J(X)\rightarrow R^JF\Phi(X)$
\end{proof}

\begin{corollary}\label{GlinHlin}
Let $\Phi\colon \cat{C}\rightarrow\cat{D}^G$ be a homotopy functor, and suppose that the functor $\hom_H(G,-)\colon \cat{D}^H\rightarrow\cat{D}^G$ detects equivalences of fibrant objects. If $\Phi_\ast\colon \cat{C}^{\mathcal{P}(J)}_{a}\rightarrow\cat{D}^{\mathcal{P}(J)}_{a}$ sends homotopy cocartesian cubes to homotopy cartesian cubes, so does $\Phi_\ast\colon \cat{C}^{\mathcal{P}(J|_H)}_{a|_H}\rightarrow\cat{D}^{\mathcal{P}(J|_H)}_{a|_H}$.

It follows that if $\Phi$ is $G$-linear, it is also $H$-linear for every subgroup $H\leq G$.
\end{corollary}

\begin{proof}
From the explicit descriptions of $L^J$ and $R^J$ of \ref{higherwirth} one can see that $L^J$ commutes with homotopy colimits and that $R^J$ commutes with homotopy limits. In particular, if $X$ is a homotopy cocartesian $J|_H$-cube, the $J$-cube $L^J(X)$ is also homotopy cocartesian. Hence by our assumption on $\Phi$, the $J$-cube $\Phi_\ast L^J(X)$ is homotopy cartesian. The top horizontal map in the commutative diagram
\[\xymatrix{
\Phi_\ast L^J(X)_\emptyset\ar[r]^-{\simeq}\ar[d]& \displaystyle\holim_{\mathcal{P}(J)\backslash\emptyset} F\Phi_\ast L^J(X)\ar[d]\\
R^JF\Phi_\ast(X)_\emptyset\ar[r]&
\displaystyle\holim_{\mathcal{P}(J)\backslash\emptyset}R^JF\Phi_\ast(X)
}\] 
is therefore an equivalence. The vertical maps are also equivalences by the higher Wirthm\"{u}ller isomorphism theorem \ref{higherwirth}. Thus the bottom horizontal map is also an equivalence, and it factors as
\[R^JF\Phi_\ast(X)_\emptyset\rightarrow\displaystyle R^\emptyset\holim_{\mathcal{P}(J)\backslash\emptyset}F\Phi_\ast(X)\stackrel{\simeq}{\rightarrow}
\displaystyle\holim_{\mathcal{P}(J)\backslash\emptyset}R^JF\Phi_\ast(X)\]
The first map of the factorization  is therefore also an equivalence, and by the explicit description of $R^J$ in the proof of \ref{higherwirth}, it is just the canonical map
\[\hom_H(G,F\Phi_\ast(X)_\emptyset)\longrightarrow \hom_H(G,\holim_{\mathcal{P}(J)\backslash\emptyset}F\Phi_\ast(X)).\]
Since $\hom_H(G,-)$ detects equivalences of fibrant objects, $\Phi_\ast(X)$ is homotopy cartesian.

For the second part of the statement, assume that $\Phi$ is $G$-linear and let $X$ be a be a homotopy cocartesian $H_+$-cube. Consider the $H$-equivariant surjection $p\colon G_+|_H\rightarrow H_+$ which is the identity on $H$ and that collapses the complement of $H$ to the basepoint. It induces a functor $p^\ast\colon\cat{C}^{\mathcal{P}(H_+)}_a\rightarrow\cat{C}^{\mathcal{P}(G_+|_H)}_{a}$ which by \ref{surjmapcubesco} preserves homotopy cocartesian cubes. Hence $p^\ast X$ is a homotopy cocartesian $G_+|_H$-cube. By the first part of the corollary and $G$-linearity, $\Phi_\ast(p^\ast X)=p^\ast\Phi_\ast(X)$ is homotopy cartesian. By \ref{surjmapcubes}, $p^\ast$ detects homotopy cartesian cubes, hence $\Phi_\ast(X)$ is homotopy cartesian.  

\end{proof}

\subsection{\texorpdfstring{$G$}{G}-linearity and adjoint assembly maps}

Let $\cat{C}$ and $\cat{D}$ be pointed $G$-model categories, and $\Phi\colon\cat{C}\rightarrow\cat{D}^G$ a $sSet$-enriched reduced homotopy functor. Its extension $\Phi\colon\cat{C}^G\rightarrow\cat{D}^G$ is then enriched over $G$-$sSet$, and for any simplicial $G$-set $K$ there is an assembly map
\[K\otimes\Phi(c)\longrightarrow \Phi(K\otimes c)\]
in $\cat{D}^G$. It is adjoint to the map of simplicial $G$-sets
\[K\longrightarrow Map_{\cat{C}}(c,K\otimes c)\stackrel{\Phi}{\longrightarrow} Map_{\cat{D}}(\Phi(c),\Phi(K\otimes c))\]
where the first map is adjoint to the identity on $K\otimes c$. When $K=N(\mathcal{P}(J_+)\backslash\emptyset)$ this induces a map
\[\alpha\colon \Phi(c)\longrightarrow \Omega^J\Phi(\Sigma^J c)\]
called the adjoint assembly map (see \ref{defloopsusp} for the definitions of $\Omega^J$ and $\Sigma^J$ in a general simplicial category). The aim of this section is to explore the relationship between $G$-linearity of $\Phi$ and the adjoint assembly map.

\begin{remark}\label{assembly}
Given a cofibrant $G$-object $c$ in $\cat{C}^G$ and a finite $G$-set $J$, recall the cofibrant $J_+$-cube
\[(S^Jc)_U=\left\{\begin{array}{lll}
c & ,U=\emptyset\\
C^Uc & ,U\lneq J_+\\
\Sigma^Jc & ,U= J_+
\end{array}\right.\]
from \ref{suspcube}. This induces a zig-zag \[\Phi(c)\stackrel{\simeq}{\rightarrow}F\Phi(c)\rightarrow \holim_{\mathcal{P}(J_+)\backslash \emptyset}F\Phi(S^Jc)\stackrel{\simeq}{\leftarrow}\Omega^JF\Phi(\Sigma^Jc)\]
where the last equivalence is induced by the equivalence of fibrant $\mathcal{P}(J_+)\backslash \emptyset$-diagrams
\[\omega^J\big(F\Phi(\Sigma^Jc)\big)\stackrel{\simeq}{\longrightarrow} F\Phi(S^Jc)|_{\mathcal{P}(J_+)\backslash \emptyset}\]
for the $G$-diagram $\omega^Jd$ from \ref{defloopsusp} associated to an object $d$ of $\cat{D}^G$, with vertices  $(\omega^Jd)_{J_+}=d$ and $(\omega^Jd)_{U}=\ast$ for $U\neq J_+$. The adjoint assembly map above fits into the commutative diagram
\[\xymatrix{\Phi(c)\ar[dr]_-{\alpha}\ar[r]^{\simeq}&F\Phi(c)\ar[r]&\displaystyle\holim_{\mathcal{P}(J_+)\backslash \emptyset}F\Phi(S^Jc)\\
&\Omega^J\Phi(\Sigma^Jc)\ar[r]&\Omega^JF\Phi(\Sigma^Jc)\ar[u]_{\simeq}.
}\]
Hence the map $\Phi(c)\to\displaystyle\holim_{\mathcal{P}(J_+)\backslash \emptyset}F\Phi(S^Jc)$ can be thought of as a model for the adjoint assembly map which can be defined without using that $\Phi$ is an enriched functor. 
\end{remark}

\begin{proposition}\label{adjasslin} Let $\cat{C}$ and $\cat{D}$ be pointed $G$-model categories, and $\Phi\colon \cat{C}\rightarrow \cat{D}^G$  a $sSet$-enriched $G$-linear homotopy functor. For any finite $G$-set $J$ and any cofibrant $G$-object $c\in\cat{C}^G$ the composite
\[\Phi(c)\stackrel{\alpha}{\longrightarrow} \Omega^J\Phi(\Sigma^Jc)\longrightarrow\Omega^JF\Phi(\Sigma^Jc)\]
is a weak equivalence in $\cat{D}^G$.
\end{proposition}

\begin{proof}
The decomposition of $J$ as disjoint union of transitive $G$-sets $J_+ \cong (\coprod_{z\in G\backslash J}z)_+$ gives a factorization the map of the statement as an iterated construction
\[\Phi(c)\rightarrow\Omega^{z_1}F\Phi(\Sigma^{z_1}c)\rightarrow\dots\rightarrow \Omega^{z_1}\dots\Omega^{z_n}F\Phi(\Sigma^{z_1}\dots \Sigma^{z_m}c)\]
The functor $\Sigma^{z}(-)$ preserves cocartesian cubes and $\Omega^{z}$ preserves fibrant objects, so using the natural weak equivalences $\Sigma^z \Sigma^w c \stackrel{\simeq}{\to} \Sigma^{z \amalg w}d$ for $d$ cofibrant and $\Omega^{z \amalg w}d \stackrel{\simeq}{\to} \Omega^z \Omega^w d$ for $d$ fibrant, it suffices to show that the map $\Phi(c)\rightarrow \Omega^{G/H}F\Phi(\Sigma^{G/H}c)$ is an equivalence for every transitive $G$-set $G/H$.
 
By \ref{fromGtoany}, $\Phi$ sends the homotopy cocartesian $G/H_+$-cube $S^{G/H}c$ of \ref{assembly} to a homotopy cartesian $G/H_+$-cube. That is, the second map in the zig-zag
\[\Phi(c)\stackrel{\simeq}{\rightarrow}F\Phi(c)\rightarrow\holim_{\mathcal{P}(G/H_+)\backslash\emptyset}F\Phi(S^{G/H}c)\stackrel{\simeq}{\leftarrow}\Omega^{G/H}F\Phi(\Sigma^{G/H}c)\]
is an equivalence in $\cat{D}^G$. The statement now follows from the commutativity of the diagram in \ref{assembly} above.
\end{proof}

We aim at proving a converse to \ref{adjasslin}. We remind the reader that a simplicial category $\cat{C}$ is locally finitely presentable if there is a set $\Theta$ of objects in $\cat{C}$ such that every object of $\cat{C}$ is isomorphic to a filtered colimit of objects in $\Theta$, and for every $\theta\in \Theta$ the functor $Map_\cat{C}(\theta,-)\colon \cat{C}\rightarrow sSet$ preserves filtered colimits (see \cite{Adamekros}, \cite{Kelly}). For example the categories of simplicial sets and of spectra (of simplicial sets) satisfy this condition. 
We will write $\Omega^\rho$ for $\Omega^{G}$ and $\Sigma^\rho$ for $\Sigma^{G}$.

\begin{theorem}\label{linandloops} Let $\cat{C}$ and $\cat{D}$ be pointed $G$-model categories and suppose that the simplicial categories $\cat{D}^H$ are locally finitely presentable for every $H\leq G$. Let $\Phi\colon \cat{C}\rightarrow \cat{D}^G$ be a $sSet$-enriched reduced homotopy functor and let $J$ be a finite $G$-set. If the canonical map
\[\Phi(c)\longrightarrow \Omega^{J|_H}F\Phi(\Sigma^{J|_H}c)\]
is a weak equivalence in $\cat{D}^H$ for every cofibrant object $c\in \cat{C}^H$ and every subgroup $H\leq G$, then the induced functor  $\Phi_\ast\colon \cat{C}^{\mathcal{P}(J_+)}_a\rightarrow \cat{D}^{\mathcal{P}(J_+)}_a$ sends homotopy cocartesian $J_+$-cubes to homotopy cartesian $J_+$-cubes.

In particular, if $\Phi(c)\stackrel{\simeq}{\rightarrow} \Omega^{\rho|_H}F\Phi(\Sigma^{\rho|_H}c)$ is an equivalence for every subgroup $H\leq G$ and every cofibrant $H$-object $c$, the functor $\Phi$ is $G$-linear.
\end{theorem}

The proof of this theorem is technical and it is given at the end of the section.

\begin{remark}\label{presforspectra}
The theorem above holds also in the $G$-model categories of pointed spaces or orthogonal spectra, even though these are not locally finitely presentable. The presentability condition is used to commute a sequential homotopy colimit and a finite equivariant homotopy limit, as explained in \ref{limscolims}. These commute also in $Top_\ast$ and $\Sp^O$, for the following reason.  They commute in $sSet_\ast$ as $sSet^{H} _\ast$ is locally finitely presentable. This property can be transported through the $G$-Quillen equivalence $|-|\colon sSet_\ast\rightleftarrows Top_\ast \colon Sing$, using that realization commutes with finite limits and $Sing$ with sequential colimits along cofibrations. It can be further deduced for $\Sp^O$ as limits and colimits are levelwise.
\end{remark}

\begin{corollary}\label{corlinandloops}
Under the hypotheses of \ref{linandloops}, suppose additionally that the functor $\hom_H(G,-)\colon\cat{D}^H\rightarrow\cat{D}^G $ detects equivalences of fibrant objects for every subgroup $H$ of $G$. Then
the following are equivalent:
\begin{enumerate}
\item $\Phi$ is $G$-linear,
\item For every cofibrant object $c\in \cat{C}^H$ and every $H\leq G$, the canonical map  $\Phi(c)\rightarrow \Omega^{\rho|_H}F\Phi(\Sigma^{\rho|_H}c)$ is an equivalence in $\cat{D}^H$,
\item For every finite $G$-set $J$ the functor $\Phi_\ast\colon \cat{C}^{\mathcal{P}(J_+)}_a\rightarrow \cat{D}^{\mathcal{P}(J_+)}_a$ sends homotopy cocartesian $J_+$-cubes to homotopy cartesian $J_+$-cubes.
\end{enumerate}
\end{corollary}

\begin{proof}
\noindent $(1)\Rightarrow(2)$ By \ref{GlinHlin} the functor $\Phi$ is $H$-linear for every subgroup $H\leq G$. The implication then follows from \ref{adjasslin} for the $H$-set $G_+|_H$.
\\
\noindent $(2)\Rightarrow(3)$ By \ref{linandloops} it is enough to show that $\Phi(c)\rightarrow \Omega^{J|_H}F\Phi(\Sigma^{J|_H}c)$ is an equivalence for every finite $G$-set $J$. But $\Phi$ is $G$-linear by  \ref{linandloops}, and hence $H$-linear by  \ref{GlinHlin}. The adjoint assembly is then an equivalence by $\ref{adjasslin}$.
\\
\noindent $(3)\Rightarrow(1)$ For $J=G$ the conclusion in (3) is the definition of $G$-linearity.
\end{proof}

\begin{remark}\label{Gderivative}
Define the $G$-derivative (at the zero object) of a reduced enriched homotopy functor $\Phi\colon\cat{C}\rightarrow \cat{D}^G$ to be the functor $D_\ast\Phi\colon\cat{C}\rightarrow \cat{D}^G$ defined by
\[D_\ast\Phi(c)=\hocolim\big(Q\Phi(c)\rightarrow Q\Omega^\rho F\Phi(\Sigma^\rho c)\rightarrow Q\Omega^{2\rho}F\Phi(\Sigma^{2\rho} c)\rightarrow\dots\big)\]
where $\Sigma^{n\rho}=\Sigma^{nG}$ is the suspension by the permutation representation of $n$-disjoint copies of $G$. As a direct consequence of point 2 of \ref{corlinandloops} the functor $D_\ast \Phi$ is $G$-linear, and it is equipped with a universal natural transformation $\Phi\rightarrow D_\ast\Phi$. The argument of \cite[1.8]{calcIII} applies verbatim to our equivariant situation, showing that $\Phi\rightarrow D_\ast\Phi$ is essentially initial among maps from $\Phi$ to a $G$-excisive functor. 
\end{remark}

\begin{proof}[Proof of \ref{linandloops}]
We follow the strategy of the proofs of \cite[1.8,1.9]{calcIII} and \cite{Rezk} of showing that the adjoint assembly map evaluated at a cocartesian cube factors through a cartesian cube. It is convenient to introduce a new model for the loop space. For a cofibrant object $c\in\cat{C}^G$ we define
\[\overline{\Omega}^J F\Phi(\Sigma^Jc):=\holim_{\mathcal{P}(J_+)\backslash\emptyset}F\Phi(S^Jc).\]
This object comes with a natural weak equivalence $\overline{\Omega}^J F\Phi(\Sigma^Jc)
\stackrel{\simeq}{\longleftarrow}\Omega^J F\Phi(\Sigma^Jc)$ (see \ref{assembly}).
Let $X\colon \mathcal{P}(J_+)\rightarrow \cat{C}$ be a cofibrant $J_+$-cube. Define a $G$-diagram $K\colon \mathcal{P}(J_+)\times \mathcal{P}(J_+)\rightarrow \cat{C}$ by
\[K(U,T)=\hocolim_{S\in \mathcal{P}(J_+)\backslash J_+}X_{(S\cap U)\cup T}\]
and define a $J_+$-cube $Y\colon \mathcal{P}(J_+)\rightarrow \cat{D}$ by
\[Y_T=\holim_{\mathcal{P}(J_+)\backslash\emptyset} F\Phi(K(-,T))\]
The key of this proof is to define, for every $T\subset J_+$, a factorization, natural in $T$
\[\xymatrix{\Phi(X_T)\ar[rr]\ar[dr]_\phi&& \overline{\Omega}^JF \Phi(\Sigma^{J}X_T)\\
&Y_T\ar[ur]_\psi}\]
and show that $Y$ is homotopy cartesian when $X$ is homotopy cocartesian.
Writing $\Delta^{\tilde{U}}$ for $N\mathcal{P}(U)\backslash\emptyset$, the first map of the factorization has $U$-component 
\[\phi_U\colon \Phi(X_T)\longrightarrow map(\Delta^{\tilde{U}},F\Phi(K(U,T)))\]
adjoint to the composite
\[\Delta^{\tilde{U}}\otimes\Phi(X_T)\rightarrow \Phi(\Delta^{\tilde{U}}\otimes X_T)\rightarrow \Phi(K(U,T))\rightarrow F\Phi(K(U,T))\]
where the second map is induced by $\Delta^{\tilde{U}}\otimes X_T\rightarrow \Delta^{\tilde{U}}\otimes X_{T\cup U}\rightarrow K(U,T)$. The map $\psi$ is the homotopy limit over $U$ of the map of diagrams $F\Phi(K(U,T))\rightarrow F\Phi((S^JX_T)_U)$ induced by the map $K(U,T)\rightarrow (S^JX_T)_U$ defined as follows. For $U\neq J_+$, it is the composite
\[K(U,T)=\displaystyle\hocolim_{S\in \mathcal{P}(J_+)\backslash J_+}X_{(S\cap U)\cup T}=\hocolim_{S\in \mathcal{P}(J_+)\backslash J_+}(-\cap U)^\ast X_{S\cup T}\to\hocolim_{S\in \mathcal{P}(U)}X_{S\cup T}\to C^UX_T
\]
where the first arrow is the canonical map induced by the functor $(-\cap U)\colon \mathcal{P}(J_+)\backslash J_+\rightarrow \mathcal{P}(U)$ and the second arrow is induced by collapsing all the non-initial verticies. For $U=J_+$, the map is
\[K(J_+,T)=\displaystyle\hocolim_{S\in \mathcal{P}(J_+)\backslash J_+}X_{S\cup T}\rightarrow \hocolim_{\mathcal{P}(J_+)\backslash J_+}\sigma^JX_T=\Sigma^JX_T\]
induced on homotopy colimits by the map of $J_+$-cubes given by the identity on the empty set vertex, and that collapses the other vertices to the point.

Now suppose that $X$ is homotopy cocartesian, and let us see that $Y$ is homotopy cartesian. There is a natural equivalence $K(U,T)\stackrel{\simeq}{\to} X_{U\cup T}$. Indeed, the maps $X_{(S\cap U)\cup T}\rightarrow X_{((S\cup \{t\})\cap U)\cup T}$ are the identity for all $t\in T$, and therefore $K(U,T)\stackrel{\simeq}{\to} X_{U\cup T}$ as long as $T\neq\emptyset$, by the lemma \ref{cartcocartid} below. For $T=\emptyset$ and $U\neq J_+$ there is a weak equivalence
\[K(U,\emptyset)=\displaystyle\hocolim_{S\in \mathcal{P}(J_+)\backslash J_+}X_{S\cap U}\stackrel{\simeq}{\to} X_{U}\]
again by \ref{cartcocartid}, as the maps $X_{S\cap U}\rightarrow X_{(S\cup\{v\})\cap U}$ are the identity for all $v\in J_+\backslash U$. Finally,
\[K(J_+,\emptyset)=\displaystyle\hocolim_{S\in \mathcal{P}(J_+)\backslash J_+}X_{S}\stackrel{\simeq}{\to} X_{J_+}\]
since $X$ is assumed to be homotopy cocartesian. This shows that \[Y_T\stackrel{\simeq}{\to} \holim_{U\in \mathcal{P}(J_+)\backslash\emptyset}F\Phi(X_{U\cup T})\]
For every fixed $U\neq\emptyset$, the cube $T\longmapsto F\Phi(X_{U\cup T})$ is homotopy cartesian by \ref{cartcocartid}, as the maps $F\Phi(X_{U\cup T})\rightarrow F\Phi(X_{U\cup T\cup\{u\}})$ are the identity for all $u\in U$. The cube $Y$ is then a homotopy limit of cartesian cubes, and therefore also cartesian since homotopy limits commute with each other.

Iterating this construction and using that $\Sigma^J$ and $\overline{\Omega}^J$ preserve homotopy cocartesian and -cartesian $J_+$-cubes, respectively, one gets a factorization of each map in the colimit system
\[\Phi(X)\stackrel{\simeq}{\longrightarrow} \overline{\Omega}^{J}F\Phi(\Sigma^{J}X)\stackrel{\simeq}{\longrightarrow} \overline{\Omega}^{2J}F\Phi( \Sigma^{2J}X)\stackrel{\simeq}{\longrightarrow} \dots\]
through a homotopy cartesian $J_+$-cube $Y^{(n)}$. The maps in this system are weak equivalences, since all maps $\Phi(c)\to \Omega^{J|_H}F\Phi(\Sigma^{J|_H}c)$ are assumed to be weak equivalences. By (classical) cofinality for diagrams in $\cat{D}^G$ the homotopy colimit of the sequence above is equivalent to $\hocolim_n Y^{(n)}$.
By \ref{limscolims} in the appendix we know that under our presentability assumptions sequential homotopy colimits preserve homotopy cartesian $J_+$-cubes. Therefore $\Phi(X)\simeq \hocolim_{n}Q\overline{\Omega}^{nJ}F\Phi(\Sigma^{nJ}X)$ is homotopy cartesian.
\end{proof}

\begin{lemma}\label{cartcocartid}
Let $J$ be a finite $G$-set, $X\colon\mathcal{P}(J)\rightarrow \cat{C}$ a $J$-cube and $I\subset J$ a non-empty $G$-invariant subset such that the maps $X_S\rightarrow X_{S\cup i}$ are isomorphisms for all $S\subset J$ and $i\in I$. If $X$ is a fibrant diagram, it is homotopy cartesian. Similarly, if $X$ is point-wise cofibrant, it is homotopy cocartesian.
\end{lemma}

\begin{proof}
Let $\mathcal{P}_I(J)$ be the subposet of $\mathcal{P}(J)\backslash\emptyset$ consisting of non-empty subsets of $J$ that contain $I$ and write $\iota$ for the inclusion map. The map $U \mapsto U \cup I$ defines a retraction $u_I \colon \mathcal{P}(J)\backslash\emptyset \to \mathcal{P}_I(J)$. The assumption on the maps $X_S\rightarrow X_{S\cup i}$ implies that the natural map $X \to u_I^\ast \iota^\ast X$ is an isomorphism. The composite of the maps
\[\holim_{\mathcal{P}_I(J)} \iota^\ast X \stackrel{\simeq}{\to} \holim_{\mathcal{P}(J)\backslash\emptyset}u_I^\ast \iota^\ast X \to \holim_{\mathcal{P}_I(J)} \iota^\ast u_I^\ast \iota^\ast X\]
is the identity map and the left hand map is a weak equivalence since $u_I$ is right $G$-cofinal. Hence the right-hand map is a weak equivalence. It fits into a commutative diagram
\[\xymatrix{X_{\emptyset}\ar[r]\ar[d]_{\cong}&\displaystyle\holim_{\mathcal{P}(J)\backslash\emptyset}X\ar[d]^{\simeq}\\
X_I\ar[r]_-{\simeq}&\displaystyle\holim_{\mathcal{P}_I(J)}X.
}\]
The left vertical map is a $G$-map, which is an isomorphism by assumption and the bottom horizontal map is a $G$-equivalence since $I$ is initial in $\mathcal{P}_I(J)$. Therefore the top map in the square is a weak equivalence and $X$ is homotopy cartesian.

A completely analogous argument shows that $X$ is homotopy cocartesian.
\end{proof}


\subsection{\texorpdfstring{$G$}{G}-linear functors on pointed \texorpdfstring{$G$}{G}-spaces}\label{Gtopsec}

In \cite{Blumberg} Blumberg defines a notion of $G$-linearity for endofunctors of the category of pointed $G$-spaces, for a compact Lie group $G$. When $G$ is finite, we show that his definition and ours agree up to a suspension factor.

Before starting, let us remark that when working with spaces we can drop all the point-wise fibrant and cofibrant replacements from the last sections, as homotopy limits and homotopy colimits of $G$-diagrams of spaces are always homotopy invariant. For homotopy limits, it is just because every $G$-space is fibrant. For homotopy colimits, there is a natural homeomorphism
\[(\hocolim_I X)^H\cong \hocolim_{I^H}(\iota_H^{\ast}X)^H\]
for every $G$-diagram $X$ in $(Top_\ast)^{I}_a$ and subgroup $H\leq G$. Here $\iota_H\colon I^H\rightarrow I$ is the inclusion of the subcategory of $I$ of objects and morphisms fixed by the $H$-action. Therefore homotopy invariance of homotopy colimits of $G$-diagrams follows from homotopy invariance of classical homotopy colimits of spaces, proved in \cite{DugIsak}.

\begin{proposition}\label{Glintop}
An enriched reduced homotopy functor $\Phi\colon Top_\ast\rightarrow Top^{G}_\ast$ is $G$-linear if and only if  the following two conditions hold:
\begin{enumerate}[a)]
\item The induced functor $\Phi_{\ast}\colon (Top^{G}_\ast)^{\mathcal{P}(1_+)}\rightarrow (Top^{G}_\ast)^{\mathcal{P}(1_+)}$ sends homotopy cocartesian squares of pointed $G$-spaces to homotopy cartesian ones.
\item For all finite $G$-sets $J$ the natural map
\[\Phi(\bigvee_{J}Z)\rightarrow\prod_J\Phi(Z)\]
is an equivalence of pointed $G$-spaces.
\end{enumerate}
\end{proposition}

\begin{remark}
The two conditions of \ref{Glintop} are essentially the definition of $G$-linearity in the case of a finite group $G$ in \cite{Blumberg}.
\end{remark}

\begin{proof}
If $\Phi$ is $G$-linear, it sends homotopy cocartesian squares to homotopy cartesian squares by \ref{squaresinG}, and the map $\Phi(\bigvee_{J}Z)\rightarrow\prod_J\Phi(Z)$ is an equivalence by \ref{wedgesintoprod}. 

Conversely, Blumberg proves in \cite{Blumberg} that the two conditions above imply that the adjoint assembly map $\Phi(Z)\rightarrow \Omega^V\Phi(Z\wedge S^V)$ is a $G$-equivalence for every $G$-representation $V$.
By \ref{linandloops} this implies $G$-linearity of $\Phi$.
\end{proof}

\begin{example}\label{Gconf}
Let $M$ be a commutative well-pointed topological monoid with additive $G$-action, and suppose that the fixed point monoids $M^H$ are group-like for every subgroup $H$ of $G$. The equivariant Dold-Thom construction $M(-)\colon Top_\ast\rightarrow Top^{G}_\ast$ sends a pointed space $Z$ to the space $M(Z)$ of reduced configurations of points in $Z$ with labels in $M$, with $G$ acting on the labels. After extending $M(-)$ to $Top^{G}_\ast$ the group acts both on the labels and on the space. If $M$ is discrete the homotopy groups of $M(-)$ are Bredon cohomology of the Mackey functor $H\mapsto M^H$.
For a pointed $G$-simplicial set $K$ the simplicial Dold-Thom construction of \ref{exlin} compares to the topological one by a natural $G$-homeomorphism $|M(K)|\cong M(|K|)$.

We prove that $M(-)\colon Top_\ast\rightarrow Top^{G}_\ast$ is $G$-linear by checking the two conditions from \ref{Glintop}.
Given a pointed $G$-space $Z$, the fixed points of the map $M(Z)\rightarrow \Omega M(Z\wedge S^1)$ compares by natural homeomorphisms to the adjoint assembly map
\[M(Z)^H\longrightarrow  \Omega M(Z\wedge S^1)^H\cong \Omega(M(Z))(S^1)^H\cong \Omega M(Z)^H(S^1)\]
for the topological group-like monoid $M(Z)^H$. This is an equivalence by standard arguments, see \cite[7.6]{May}. This implies, by \ref{linandloops} for the trivial $G$-set $J=\{1\}$, that the functor $M(-)$ sends homotopy cocartesian squares of $G$-spaces to homotopy cartesian ones, proving the first property of \ref{Glintop}. The second property easily follows, as the map $M(\bigvee_{J}Z)\rightarrow\prod_JM(Z)$ is an equivariant homeomorphism.

Notice that by $G$-linearity the map $M(Z)\rightarrow \Omega^JM(Z\wedge S^J)$ is a $G$-equivalence 
for every finite $G$-set $J$. This shows that the associated Eilenberg-MacLane $G$-spectrum $HM_n=M(S^n)$ is fibrant in $(\Sp^O)^G$.
\end{example}


\subsection{\texorpdfstring{$G$}{G}-linear functors to \texorpdfstring{$G$}{G}-spectra}\label{tospectra}
We show that the identity functor on $G$-spectra is $G$-linear, and deduce from this the classical Wirthm\"{u}ller isomorphism theorem. We further classify all $G$-linear functors from finite pointed simplicial sets to $G$-spectra.

Let us start by clarifying that when working with spectra, as for spaces, we can forget all about the point-wise cofibrant and fibrant replacements from the previous sections, thanks to the following result.

\begin{lemma}
Let $G$ be a finite group and let $a$ be an action of $G$ on a small category $I$. 
\begin{itemize}
\item The homotopy colimit functor $\hocolim\colon (\Sp^O)^{I}_a\rightarrow (\Sp^O)^G$ preserves weak equivalences between any two diagrams (not necessarily of cofibrant objects).

\item If $I$ has finite dimensional nerve, the homotopy limit functor $\holim\colon (\Sp^O)^{I}_a\rightarrow (\Sp^O)^G$ preserves weak equivalences between any two diagrams (not necessarily of fibrant objects).
\end{itemize}
\end{lemma}

\begin{proof}
For any $H$-spectrum $E$ there is a functorial cofibrant replacement $QE\rightarrow E$ where the map is a level equivalence. By \ref{htpyinvlimcolim} it is enough to show that homotopy colimits preserve level equivalences of maps of $G$-diagrams. Since homotopy colimits of spectra are defined level-wise, this follows from homotopy invariance of homotopy colimits for spaces (see \S\ref{Gtopsec}).

For the statement about homotopy limits, take a $G$-diagram of spectra $X$. The positive equivariant homotopy groups of $\holim_IX$ are the homotopy groups of the $G$-space \[\hocolim_n\Omega^{n\rho}(\holim_IX)(n\rho).\]
Here we use the notation $E(n\rho)=E_n\wedge_{O(n)}L(\mathbb{R}^{n|G|},n\rho)^+$ for a $G$-spectrum $E$, where $L(\mathbb{R}^{n|G|},n\rho)$ is the space of isomorphisms of vector spaces from $\mathbb{R}^{n|G|}$ to $n\rho$. There are natural weak equivalences
\[\begin{array}{ll}\hocolim_n\Omega^{n\rho}(\holim_IX)(n\rho)\cong \hocolim_n\Omega^{n\rho}\holim_I(X(n\rho))\cong\\
\hocolim_n\holim_I\Omega^{n\rho}(X(n\rho))\stackrel{\simeq}{\too}
\holim_I\hocolim_n\Omega^{n\rho}(X(n\rho))
\end{array}\]
where the last map is a weak equivalence by \ref{limscolims} as sequencial homotopy colimits and finite homotopy limits of $G$-diagrams of spaces commute.
Therefore, a weak equivalence of $G$-diagrams of spectra $f\colon X\rightarrow Y$ induces an isomorphism in positive homotopy groups of the homotopy limit precisely when the map $\holim_I\hocolim_n\Omega^{n\rho}f^{(n\rho)}$ is an equivalence of $G$-spaces. Since $f$ is an equivalence of $G$-diagrams of spectra, the map $\hocolim_n\Omega^{n\rho}f^{(n\rho)}_i$ is an equivalence of $G_i$-spaces for all objects $i$ of $I$. It follows  by homotopy invariance \ref{htpyinvhom} that the map of $G$-spaces $\holim_I\hocolim_n\Omega^{n\rho}f^{(n\rho)}$ is a weak equivalence since it is a homotopy limit of a weak equivalence of $G$-diagrams of spaces. A similar argument shows that $\holim_I f$ is an equivalence in negative homotopy groups. 

\end{proof}

\begin{theorem}\label{Gcartcocart}
Let $J$ be a finite $G$-set and let $a$ be the induced action of $G$ on $\mathcal{P}(J_+)$. Any homotopy cocartesian $J_+$-cube $X$ in $(\Sp^O)^{\mathcal{P}(J_+)}_a$ is homotopy cartesian. That is, the inclusion functor $\Sp^O\rightarrow (\Sp^O)^G$ is $G$-linear.

In particular, this implies the Wirthm\"{u}ller isomorphism theorem, stating that for any subgroup $H\leq G$ and $H$-spectrum $E\in (\Sp^O)^H$ the canonical map
\[\eta\colon G\otimes_HE=G_+\wedge_HE\longrightarrow F_H(G_+,E)=\hom_H(G,E)\]
is a weak equivalence of $G$-spectra.
\end{theorem}

\begin{proof}
By the equivariant suspension theorem, the map $E\rightarrow \Omega^{\rho|_H}(E\wedge S^{\rho|_H})$ is a weak equivalence for any $H$-spectrum $E$. By \ref{linandloops} (see also \ref{presforspectra}) this is equivalent to $G$-linearity of the functor $\Sp^O\rightarrow (\Sp^O)^G$. The map $\eta\colon G\otimes_HE\longrightarrow \hom_H(G,E)$ is a weak equivalence by \ref{linandwedges}, as $\hom_H(G,-)\colon (\Sp^O)^H\rightarrow (\Sp^O)^G$ preserves weak equivalences.
\end{proof}

We end the section with a complete characterization of $G$-linear functors from the category $sSet^{f}_\ast$ of finite pointed simplicial sets to $G$-spectra.

\begin{proposition}\label{classlintospec}
Let $\Phi\colon sSet^{f}_\ast\rightarrow (\Sp^O)^G$ be a $sSet$-enriched reduced homotopy functor such that the spectrum $\Phi(S^0)$ is level-wise well-pointed. Then the following conditions are equivalent:
\begin{enumerate}
\item The functor $\Phi$ is $G$-linear.
\item The functor $\Phi_\ast\colon ((sSet^{f}_\ast)^{G})^{\mathcal{P}(1_+)}\rightarrow ((\Sp^O)^G)^{\mathcal{P}(1_+)}$ sends homotopy cocartesian squares in $(sSet^{f}_\ast)^G$ to homotopy cartesian squares of $G$-spectra, and $\Phi(\bigvee_{J}K)\rightarrow \prod_J \Phi(K)$ is an equivalence for every finite pointed simplicial $G$-set $K$ and finite $G$-set $J$.
\item
For every $K\in (sSet^{f}_\ast)^G$ the assembly map
\[\Phi(S^0)\wedge |K|\longrightarrow \Phi(K)\]
is an equivalence of $G$-spectra.
\end{enumerate}
\end{proposition}

\begin{proof}
$(1)\Rightarrow (2)$ This is true in general, by \ref{squaresinG} and \ref{wedgesintoprod}.\\
$(2)\Rightarrow (3)$ This can be proven by induction on the skeleton of $K$. The wedges into products condition gives the equivalence for the $0$-skeleton, and the induction step follows from the condition on squares. We refer to \cite{Gcalc} for the details.
\\
$(3)\Rightarrow (1)$ Since $G$-linearity is invariant under equivalences, we show that $E\wedge |-|$ is $G$-linear for any level-wise well-pointed $G$-spectrum $E$. If $X\colon \mathcal{P}(G_+)\rightarrow sSet^{f}_\ast$ is homotopy cocartesian, the cube of spectra $E\wedge |X|$ is also homotopy cocartesian. Indeed, after applying geometric fixed points $F^H$ the map from the homotopy colimit to the value at $G_+$ factors as
\[\begin{array}{ll}\displaystyle F^H(\hocolim_{\mathcal{P}(G_+)\backslash G_+} E\wedge |X|)\cong F^H(E\wedge \hocolim_{\mathcal{P}(G_+)\backslash G_+} |X|)\cong F^H(E)\wedge (\hocolim_{\mathcal{P}(G_+)\backslash G_+} |X|)^{H}\stackrel{\simeq}{\rightarrow}\\
\stackrel{\simeq}{\rightarrow} F^H(E)\wedge |X_{G_+}|^H\cong  F^H(E\wedge |X_{G_+}|),
\end{array}\]
where the third map is a weak equivalence since $X$ is homotopy cocartesian, and since smashing with a level-wise well-pointed spectrum preserves weak equivalences. By \ref{Gcartcocart} the diagram $E\wedge X$ is also homotopy cartesian.
\end{proof}


\appendix

\section{Appendix}

\subsection{Computing homotopy (co)limits of punctured cubes}

We compare homotopy limits and colimits of punctured cubes of different sizes, specifically how functors between categories of cubes in $\cat{C}$ induced by maps $p\colon K\rightarrow J$ of finite $G$-sets behave on homotopy cartesian and cocartesian cubes.

\begin{proposition}\label{surjmapcubes}
Let $p\colon K\rightarrow J$ be a surjective equivariant map of finite $G$-sets. Taking the image by $p$ induces an equivariant functor $p_\emptyset\colon \mathcal{P}(K)\backslash\emptyset\rightarrow \mathcal{P}(J)\backslash\emptyset$, which is left $G$-cofinal. In particular, the induced functor $p^{\ast}\colon \cat{C}^{\mathcal{P}(J)}_a\rightarrow \cat{C}^{\mathcal{P}(K)}_a$ preserves and detects homotopy cartesian cubes.
\end{proposition}

\begin{proof}
We show that for any subgroup $H\leq G$ and any non-empty object $U\in \mathcal{P}(J)^H$ the set $p^{-1}(U)\subset K$ is the final object of $(p_\emptyset/U)^H$. It is non-empty since $p$ is assumed to be surjective, and clearly satisfies $pp^{-1}(U)=U\subset U$. It is final since objects $V\in (p_\emptyset/U)^H$ satisfy $p(V)\subset U$, and therefore
\[V\subset p^{-1}p(V)\subset p^{-1}(U)\]
This shows that $p_\emptyset$ is left $G$-cofinal.
Now let $X\colon \mathcal{P}(J)\rightarrow \cat{C}$ be a $J$-cube, and $X\stackrel{\simeq}{\rightarrow} FX$ a fibrant replacement. 
There is a commutative diagram
\[\xymatrix{\displaystyle\holim_{\mathcal{P}(J)\backslash\emptyset}\iota_{\emptyset}^\ast F X\ar[d]_{p_{\emptyset}^\ast}^{\simeq} & FX_{\emptyset}=(p^\ast F X)_{\emptyset}\ar[l]\ar[d]\\
\displaystyle\holim_{\mathcal{P}(K)\backslash \emptyset}p_{\emptyset}^\ast (\iota_{\emptyset}^\ast F X)\ar@{=}[r]&
\displaystyle\holim_{\mathcal{P}(K)\backslash \emptyset} \iota_{\emptyset}^\ast p^\ast FX
}\]
where the left vertical map is an equivalence by $G$-cofinality \ref{Gcof} and where $\iota_\emptyset\colon \mathcal{P}(J)\backslash\emptyset\rightarrow \mathcal{P}(J)$ is the canonical inclusion. Notice moreover that $p^{\ast}X\stackrel{\simeq}{\rightarrow} p^{\ast}FX$ is a fibrant replacement for $p^{\ast}X$, as for every subset $S\subset K$ there is an inclusion of the stabilizer groups $G_{S}\leq G_{p(S)}$, and the forgetful functor $\cat{C}^{G_{p(S)}}\rightarrow \cat{C}^{G_{S}}$ preserves fibrant objects and equivalences by assumption. From the diagram above we see that $X$ is homotopy cartesian if and only if $p^{\ast}X$ is.
\end{proof}

Looking for a similar statement for the behavior of $p^\ast$ on cocartesian cubes we run into the problem that $p$ does not restrict to a functor $\mathcal{P}(K)\backslash K\rightarrow \mathcal{P}(J)\backslash J$. There is a formally dual version of the proof of \ref{surjmapcubes} that uses the complement dualities on $\mathcal{P}(K)$ and $\mathcal{P}(J)$, but it involves a functor $\cat{C}^{\mathcal{P}(J)}_a\rightarrow \cat{C}^{\mathcal{P}(K)}_a$ different from $p^\ast$. This is discussed in \ref{pastdual} below. In order to understand the interaction between $p^\ast$ and cocartesian cubes we need to introduce a new functor. Let $p^{-1}(j)\subset K$ denote the fiber of an element $j\in J$, and consider the equivariant functor
\[\lambda\colon \Big(\prod_{j\in J}\mathcal{P}\big(p^{-1}(j)\big)\backslash p^{-1}(j)\Big)\times\mathcal{P}(J)\backslash J\rightarrow \mathcal{P}(K)\backslash K\]
that sends a pair $(\{U_j\}_{j\in J}, V)$ to $(\amalg_{j\in J}U_j)\cup p^{-1}(V)$. 
The product $\prod_{j\in J}\mathcal{P}\big(p^{-1}(j)\big)\backslash p^{-1}(j)$ is the limit of the $G$-diagram of categories $j \mapsto \mathcal{P}(p^{-1}(j)) \setminus p^{-1}(j)$ with the $G$-structure induced by the $G$-action on $J$. 

The functor $\lambda$ is a categorical analogue of a homeomorphism 
\[\big(\prod_{j\in J}\Delta^{|p^{-1}(j)|-1}\big)\times\Delta^{|J|-1}\cong \Delta^{|K|-1}\]
\begin{example}
\begin{itemize}
\item
If $p\colon K_+\rightarrow 1_+$ is the pointed map that sends all the elements of $K$ to $1$, the product of the fibers is simply $\mathcal{P}(K)\backslash K$ and the functor
\[\lambda\colon \mathcal{P}(K)\backslash K\times\mathcal{P}(1_+)\backslash 1_+\rightarrow \mathcal{P}(K_+)\backslash K_+\]
is analogous to a homeomorphism $\Delta^{\overline{K}}\times\Delta^1\cong \Delta^K$ that splits off a copy of the trivial representation from the permutation representation of $K$. This is written on a more familiar form as $\overline{\mathbb{R}}[K]\times \mathbb{R}\cong \mathbb{R}[K]$. One could think of the product of the categories $\mathcal{P}\big(p^{-1}(j)\big)\backslash p^{-1}(j)$ as an orthogonal complement for the image of the embedding $p^{-1}(-)\colon \mathcal{P}(J)\backslash J\rightarrow\mathcal{P}(K)\backslash K$.
\item Let $I$ and $J$ be finite $G$-sets, and consider the pointed projection $p\colon (I\amalg J)_+\rightarrow J_+$ that sends $J$ to $J$ by the identity, and $I$ to the basepoint $+$. The preimages over the elements of $J$ consist of a single point, and the preimage over the basepoint is $p^{-1}(+)=I_+$. The functor $\lambda$ above is the functor
\[\lambda\colon \mathcal{P}(I_+)\backslash I_+\times\mathcal{P}(J_+)\backslash J_+\longrightarrow\mathcal{P}((I\amalg J)_+)\backslash (I\amalg J)_+\]
that sends $(U,V)$ to $U\cup V$. It is analogous to the standard homeomorphism of permutation representations $\mathbb{R}[I]\times \mathbb{R}[J]\cong \mathbb{R}[I\amalg J]$.
\end{itemize}
\end{example}

\begin{proposition}\label{surjmapcubesco}
For a surjective equivariant map $p\colon K\rightarrow J$, the functor $\lambda$ above is right $G$-cofinal. Moreover, the functor $p^{\ast}\colon \cat{C}^{\mathcal{P}(J)}_a\rightarrow \cat{C}^{\mathcal{P}(K)}_a$ preserves homotopy cocartesian cubes.
\end{proposition}

\begin{proof}
Let us first prove that $\lambda$ is well defined, that is, it does not take the value $K$. Write for simplicity $\underline{U}=\{U_j\}_{j\in J}$ and $\amalg \underline{U}=\amalg_{j\in J}U_j$. Suppose that $\lambda(\underline{U},V)=(\amalg \underline{U})\cup p^{-1}(V)=K$. Take $j$ in the complement of $V$ in $J$. The fiber $p^{-1}(j)\subset K$ is disjoint from $p^{-1}(V)$, but it is covered by the collection $\underline{U}$. As each $U_i$ is contained in $p^{-1}(i)$ we must have $U_j=p^{-1}(j)$, but this is absurd since $U_j$ is a proper subset of $p^{-1}(j)$.

Now let $W$ be an $H$-invariant proper subset of $K$. We show that the right fiber category $W/\lambda$ is $H$-contractible by defining a zig-zag of natural transformations between the identity functor and the projection onto the $H$-invariant object $(\underline{\emptyset}=\{\emptyset\}_{j\in J},p(W))$ of  $W/\lambda$. This is a well defined object as $\lambda(\{\emptyset\}_{j\in J},p(W))=p^{-1}p(W)$ which contains $W$. The intermediate functor of the zig-zag is the equivariant functor $\tau\colon W/\lambda\rightarrow W/\lambda$ defined by 
\[\tau(\underline{U},V)=(\underline{U},p(\amalg \underline{U})\cup V)\]
The values of $\tau$ are indeed objects of $W/\lambda$ since $\lambda(\tau(\underline{U},V))$ clearly contains $(\amalg \underline{U})\cup p^{-1}(V)$ which in turn contains $W$ as $(\underline{U},V)$ belongs to $W/\lambda$.
There is a zig-zag of natural transformations
\[\id\longrightarrow\tau\longleftarrow (\underline{\emptyset},p(W))\]
Both maps are obvious on the first component. The second component of the rightward pointing map is the inclusion $V \subset p(\amalg \underline{U})\cup V$. The second component of the left pointing map is induced by the inclusion $W\subset \lambda(\underline{U},V)$, that when projected down to $J$ gives $p(W)\subset p(\amalg \underline{U})\cup pp^{-1}(V)=p(\amalg \underline{U})\cup V$. The zig-zag above realizes to a contracting $H$-invariant homotopy of the category $W/\lambda$ showing that $\lambda$ is right $G$-cofinal.

Now let $X\in \cat{C}^{\mathcal{P}(J)}_a$ be a cocartesian $J$-cube and $QX\stackrel{\simeq}{\rightarrow} X$ a point-wise cofibrant replacement. As in the proof of \ref{surjmapcubes}, notice that $p^\ast QX\stackrel{\simeq}{\rightarrow} p^\ast X$ is a point-wise cofibrant replacement of $p^\ast X$. Let us compute the homotopy colimit of $p^{\ast}QX$ over $\mathcal{P}(K)\backslash K$. By $G$-cofinality and \ref{Fubini} there are $G$-equivalences
\[\displaystyle\hocolim_{\mathcal{P}(K)\backslash K}p^{\ast}QX\stackrel{\simeq}{\longleftarrow}\hocolim_{\Big(\prod\limits_{j\in J}\mathcal{P}\big(p^{-1}(j)\big)\backslash p^{-1}(j)\Big)\times\mathcal{P}(J)\backslash J}\lambda^\ast p^{\ast}QX\stackrel{\simeq}{\longleftarrow}\hocolim_{\Big(\prod\limits_{j\in J}\mathcal{P}\big(p^{-1}(j)\big)\backslash p^{-1}(j)\Big)}\hocolim_{\mathcal{P}(J)\backslash J}\lambda^\ast p^{\ast}QX.\]
We claim that for every fixed collection $\underline{U}$ of subsets of the fibers, the canonical map \[\phi_{\underline{U}}\colon\displaystyle\hocolim_{\mathcal{P}(J)\backslash J}(\lambda^\ast p^{\ast}QX)_{(\underline{U},-)}\rightarrow X_J\]
is a $G_{\underline{U}}$-equivalence. From this claim it follows by homotopy invariance of the homotopy colimit that
$\displaystyle\hocolim_{\mathcal{P}(K)\backslash K}p^{\ast}QX$ is equivalent to the homotopy colimit over $\prod\limits_{j\in J}\mathcal{P}\big(p^{-1}(j)\big)\backslash p^{-1}(j)$ of the constant $G$-diagram with value $X_J$. Since the indexing category is $G$-contractible (it has a $G$-invariant initial object) this is $G$-equivalent to $X_J=(p^{\ast}X)_K$, proving that $p^{\ast}QX$ is homotopy cocartesian.

Let us show that $\phi_{\underline{U}}$ is a weak equivalence. Write $Z^{\underline{U}}=(\lambda^\ast p^{\ast}QX)_{(\underline{U},-)}=QX_{p(\amalg \underline{U})\cup (-)}$. This is a $J$-cube with the $G$-action on $J$ restricted to the stabilizer group $G_{\underline{U}}$. Then $\phi_{\underline{U}}$ is an equivalence precisely when $Z^{\underline{U}}$ is homotopy cocartesian. 
If any of the sets $U_j$ is non-empty, the maps $(Z^{\underline{U}})_V\rightarrow (Z^{\underline{U}})_{V\cup j}$ are identities for every subset $V\subset J$. We proved in \ref{cartcocartid} that in this case $Z^{\underline{U}}$ is homotopy cocartesian. For the family of empty sets $\underline{U}=\underline{\emptyset}$, the $J$-cube $Z^{\underline{\emptyset}}$ is the cube $X$, which is assumed to be homotopy cocartesian.
\end{proof}

\begin{remark}
In general $p^{\ast}\colon \cat{C}^{\mathcal{P}(J)}_a\rightarrow \cat{C}^{\mathcal{P}(K)}_a$ does not detect homotopy cocartesian cubes. In the proof of \ref{surjmapcubesco} we constructed an equivalence over $X_J$
\[\hocolim_{\mathcal{P}(K)\backslash K}p^{\ast}QX\simeq \hocolim_{\prod\limits_{j\in J}\mathcal{P}\big(p^{-1}(j)\big)\backslash p^{-1}(j)} Y\]
where $Y$ is the diagram that sends $\underline{\emptyset}=(\emptyset,\dots,\emptyset)$ to $\displaystyle\hocolim_{\mathcal{P}(J)\backslash J}QX$ and all the other verticies to $X_J$. If $p^{\ast}X$ is homotopy cocartesian the left hand side is also equivalent to $X_J$, but this is in general not enough to conclude that $Y_{\underline{\emptyset}}$ is equivalent to $X_J$. However, this is the case if $\cat{C}$ is the category of spectra, as homotopy cocartesian $J$-cubes are the same as homotopy cartesian $J$-cubes (cf. \ref{Gcartcocart}). Hence the functor $p^{\ast}\colon (\Sp^O)^{\mathcal{P}(J)}_a\rightarrow (\Sp^O)^{\mathcal{P}(K)}_a$ preserves and detects homotopy cocartesian cubes.
\end{remark}

We end this section by discussing the duals of \ref{surjmapcubes} and \ref{surjmapcubesco}.
For an equivariant surjective map of finite $G$-sets $p\colon K\rightarrow J$, let $\overline{p}\colon \mathcal{P}(K)\rightarrow \mathcal{P}(J)$ be the composite functor
\[\overline{p}\colon \mathcal{P}(K)\longrightarrow \mathcal{P}(K)^{op}\stackrel{p^{op}}{\longrightarrow} \mathcal{P}(J)^{op}\longrightarrow\mathcal{P}(J)\]
that sends a subset $U$ of $K$ to $J\backslash p(K\backslash U)$. The dual of the functor $\lambda$ is defined by a similar composition, and an easy calculation shows that it is the functor
\[\overline{\lambda}\colon \Big(\prod_{j\in J}\mathcal{P}\big(p^{-1}(j)\big)\backslash \emptyset\Big)\times\mathcal{P}(J)\backslash \emptyset\rightarrow \mathcal{P}(K)\backslash \emptyset\]
that sends $(\underline{U},V)$ to $(\amalg \underline{U})\cap p^{-1}(V)$. The dual proofs of \ref{surjmapcubes} and \ref{surjmapcubesco} give the following.

\begin{proposition}\label{pastdual}
The restriction $\overline{p}\colon \mathcal{P}(K)\backslash K\rightarrow \mathcal{P}(J)\backslash J$ is right $G$-cofinal, and the functor $\overline{\lambda}$ is left $G$-cofinal. It follows that $\overline{p}^{\ast}\colon \cat{C}^{\mathcal{P}(J)}_a\rightarrow \cat{C}^{\mathcal{P}(K)}_a$ preserves and detects homotopy cocartesian cubes, and preserves homotopy cartesian cubes.
\end{proposition}

We end by noticing that this picture does not have an analogue for injective equivariant maps $\iota\colon J\rightarrow K$. It is easy to see that restricting along $\iota$ does not preserve any cartesian nor cocartesian properties of cubes. The right thing to study seems to be the preimage functor $\iota^{-1}\colon \mathcal{P}(K)\longrightarrow \mathcal{P}(J)$, but this does not restrict to either $\mathcal{P}(K)\backslash\emptyset\longrightarrow \mathcal{P}(J)\backslash\emptyset$ nor $\mathcal{P}(K)\backslash K\longrightarrow \mathcal{P}(J)\backslash J$. However, if $J$ and $K$ are pointed and $\iota$ preserves the basepoint, there is a retraction $p\colon K\rightarrow J$ that collapses the complement of the image of $\iota$ onto the basepoint. In this case we can simply contemplate $p^\ast$.


\subsection{Finite categories and cofibrant \texorpdfstring{$G$}{G}-diagrams}

We give a criterion for determining if a $G$-diagram is cofibrant in the model structure of \ref{modelstruct}, when the over-categories of the indexing category $I$ have finite dimensional nerve. Such categories are sometimes called directed Reedy categories.
The criterion is in terms of latching maps, and it is completely analogous to the classical theory (see e.g. \cite[\S 15]{hirsch}). 

Let $\cat{C}$ be a cocomplete category. We denote by $(I/i)'$ the over-category $I/i$ with the object $i=i$ removed. The latching diagram of a diagram $X\colon I\rightarrow \cat{C}$ is the diagram $L(X)\colon I\rightarrow \cat{C}$ given on objects by
 \[L(X)_i=\colim({(I/i)'}\longrightarrow I\stackrel{X}{\longrightarrow}\cat{C})\]
and on morphisms $f \colon i \to j$ by the map induced on colimits by $f_\ast \colon (I/i)'\rightarrow (I/j)'$. The inclusions $(I/i)' \inj I/i$ induce a maps $L(X)_i \to \colim_{I/i}u_i^\ast X \cong X_i$ which give a natural transformation $L(X)\rightarrow X$.

For a $G$-diagram $X\in\cat{C}^{I}_a$, the latching diagram $L(X)$ inherits a $G$-structure. The structure maps are the composite maps
\[L(X)_i \stackrel{L(g_X)}{\too} \colim\left(\left(I/i\right)' \stackrel{g}{\too} (I/gi)' \too I \stackrel{X}{\too} \cat{C}\right) \too L(X)_{gi}\]
induced by taking colimits of the compositions in the diagram
\[\xymatrix{(I/i)' \ar[r] \ar[d]^g & I \ar[d]^g  \ar@/^1pc/[dr]_{}="x"^X &  \\
(I/gi)' \ar[r] & I \ar@{-} = "i" \ar[r]_X & \cat{C} \ar@{=>}^{g_X} "x";"i"}\]

and the map canonical map $L(X)\rightarrow X$ is a map of $G$-diagrams.

\begin{proposition}\label{reedy}
Let $\cat{C}$ be a $G$-model category (see \ref{defGmodelstr}), and $I$ a category with $G$-action such that the simplicial set $NI/i$ is finite dimensional for every object $i$ in $I$.  Let $X$ be an object of $\cat{C}^{I}_a$ such that for every object $i$ in $I$ the map $L(X)_i\rightarrow X_i$ is a cofibration in $\cat{C}^{G_i}$. Then $X$ is cofibrant in the model structure on $\cat{C}^{I}_a$ of \ref{modelstruct}.
\end{proposition}

\begin{proof}
In order to show that $X$ is cofibrant we need to define a lift for every diagram in $\cat{C}^{I}_a$
\[\xymatrix{ &Y\ar@{->>}[d]^{\sim}\\
X\ar[r]\ar@{-->}[ur]^l&Z}
\]
where the vertical map is an acyclic fibration. We build this lift by induction on a filtration of $I$ defined by the degree function $\deg\colon Ob I\rightarrow \mathbb{N}$
\[\deg(i)=\dim NI/i\]
It is easy to see that the degree function is equivariant (where $\mathbb{N}$ has trivial action), and that if $\alpha\colon i\rightarrow j$ is a non-identity morphism then $\deg(i)<\deg(j)$.
Let $I_{\leq n}$ be the full subcategory of $I$ with objects of degree less than or equal to $n$. Since the degree function is equivariant, the $G$-action of $I$ restricts to $I_{\leq n}$, and the $G$-structure on $X$ restricts to a $G$-structure on the restricted diagram $X_{\leq n}\colon I_{\leq n}\to I\stackrel{X}{\longrightarrow} \cat{C}$. 
 We build the lift inductively on the diagrams $X_{\leq n}$.

For the base step, choose a section $s\colon ObI_{\leq 0}/G\rightarrow ObI_{\leq 0}$. For every orbit $\gamma\in ObI_{\leq 0}/G$ one can choose a $G_{s(\gamma)}$-equivariant lift
\[\xymatrix{ &Y_{s(\gamma)}\ar@{->>}[d]^{\sim}\\
X_{s(\gamma)}\ar[r]\ar@{-->}[ur]^{l_{s(\gamma)}}&Z_{s(\gamma)}}
\]
since the map $\emptyset=L(X)_{s(\gamma)}\rightarrow X_{s(\gamma)}$ is a cofibration in $\cat{C}^{G_{s(\gamma)}}$ by assumption (the map $Y_{s(\gamma)}\rightarrow Z_{s(\gamma)}$ is an acyclic fibration of $\cat{C}^{G_{s(\gamma)}}$ as equivalences and fibrations in $\cat{C}^{I}_a$ are point-wise). Given any object $i\in I_{\leq 0}$ outside the image of $s$, define $l_{i}\colon X_i\rightarrow Y_i$ as the composite
\[X_i\stackrel{g^{-1}}{\longrightarrow}X_{s([i])}\stackrel{l_{s[i]}}{\longrightarrow} Y_{s([i])}\stackrel{g}{\longrightarrow}Y_i\]
for a choice of $g\in G$ with $gs[i]=i$. 
Since the category $I_{\leq 0}$ is discrete (a $G$-set) by the properties of the degree function, these lifts define a map of diagrams $l^0\colon X_{\leq 0}\rightarrow Y_{\leq 0}$ lifting $X_{\leq 0}\rightarrow Z_{\leq 0}$. Moreover $l$ respects the $G$-structure since the lifts $l_{s(\gamma)}$ are $G_{s(\gamma)}$-equivariant.

Now suppose we defined a lift $l^{n-1}\colon X_{\leq n-1}\rightarrow Y_{\leq n-1}$. Let $I_n$ be the full subcategory of $I$ with objects of degree $n$. Choose a section $s^n\colon Ob I_{n}/G\rightarrow Ob I_n$, and for every $\gamma\in ObI_{n}/G$ a lift in $\cat{C}^{G_{s^n(\gamma)}}$
\[\xymatrix{ L(X)_{s^n(\gamma)}\ar[r]\ar@{>->}[d]&Y_{s^n(\gamma)}\ar@{->>}[d]^{\sim}\\
X_{s^n(\gamma)}\ar[r]\ar@{-->}[ur]_{l_{s^n(\gamma)}}&Z_{s^n(\gamma)}}
\]
The top horizontal map is the canonical map given by the universal property of the colimits defining $L(X)$. Again, the lifts exist because $L(X)_{s^n(\gamma)}\rightarrow X_{s^n(\gamma)}$ is a cofibration. For a general object $i$ of $I_n$ define 
\[l_i\colon X_i\stackrel{g^{-1}}{\longrightarrow}X_{s([i])}\stackrel{l_{s[i]}}{\longrightarrow} Y_{s([i])}\stackrel{g}{\longrightarrow}Y_i\]
Commutativity of the diagram above insures that the resulting map $l^n\colon X_{\leq n}\rightarrow Y_{\leq n}$commutes with the structure maps of $X_{\leq n}$ and $Y_{\leq n}$. Moreover $l^n$ respects the $G$-structure by $G_{s^n(\gamma)}$-equivariance of $l_{s(\gamma)}$.
\end{proof}


\subsection{Sequential homotopy colimits and finite \texorpdfstring{$G$}{G}-homotopy limits}

\begin{definition}[\cite{Kelly}]
A simplicial category $\cat{C}$ is locally finitely presentable if there is a set of objects $\Theta$ satisfying
\begin{enumerate}
\item For every $c\in \Theta$ the mapping space functor 
\[Map_\cat{C}(c,-)\colon \cat{C}\longrightarrow sSet\] preserves filtered colimits,
\item every object of $\cat{C}$ is isomorphic to a filtered colimit of objects in $\Theta$.
\end{enumerate}
\end{definition}

When $\cat{C}$ is locally finitely presented the functor $\map_{\cat{C}}(K,-)$ commutes with filtered colimits if $K$ is a finite simplicial set. This follows from the conditions above and an adjunction argument. 

We consider the poset category $\mathbb{N}$ of natural numbers as a category with trivial $G$-action.

\begin{proposition}\label{limscolims}
Let $\cat{C}$ be a $G$-model category, and suppose that the underlying simplicial categories $\cat{C}^{H}$ are locally finitely presentable for all $H\leq G$. Let $J$ be a finite $G$-set and $X\colon \mathbb{N}\times\mathcal{P}(J_+)\rightarrow \cat{C}$ a $G$-diagram with the property that for every $n\in \mathbb{N}$ the $J_+$-cube $X_n$ is homotopy cartesian. Then the $J_+$-cube $\hocolim_{\mathbb{N}}QX_n$ is also homotopy cartesian. 
\end{proposition}

\begin{proof}
We must show that the top horizontal map in the commutative diagram
\[\xymatrix{\displaystyle\hocolim_{\mathbb{N}}QX_{n,\emptyset}\ar[r]\ar[d]_{\simeq}&\displaystyle \holim_{S\in\mathcal{P}(J_+)\backslash\emptyset}F\hocolim_{\mathbb{N}}QX_{n,S}\ar[d]^{\simeq}\\
\displaystyle\colim_{\mathbb{N}}X_{n,\emptyset}\ar[r]&\displaystyle \holim_{S\in \mathcal{P}(J_+)\backslash\emptyset}F\colim_{\mathbb{N}}X_{n,S}
}
\]
is a weak equivalence in $\cat{C}^G$. The left hand vertical map is an equivalence since in the locally finitely presentable category $\cat{C}^G$ filtered colimits are homotopy invariant (see e.g. \cite[7.3]{DugComb}, or \cite{BK} for simplicial sets). Similarly, the right hand vertical map is the homotopy limit of an equivalence of pointwisefibrant $G$-diagrams, as each $\cat{C}^{G_S}$ is locally finitely presentable. The bottom map can be factored as
\[\xymatrix{\displaystyle\colim_{\mathbb{N}}X_{n,\emptyset}\ar[r]\ar[dr]_{\simeq}&\displaystyle \holim_{ S\in \mathcal{P}(J_+)\backslash\emptyset}F\colim_{\mathbb{N}}X_{n,S}\\
&\displaystyle\colim_{\mathbb{N}}\holim_{S\in \mathcal{P}(J_+)\backslash\emptyset}FX_{n,S}\ar[u]
}\]
with the diagonal map an equivalence in $\cat{C}^G$ since $X_n$ is homotopy cartesian and filtered colimits in $\cat{C}^G$ preserve equivalences.
To show that the vertical map is an equivalence, we compute from the definition of homotopy limits. Denoting $K_S=N\mathcal{P}(S)\backslash\emptyset$ we have isomorphisms in $\cat{C}^G$
\[\begin{array}{ll}\displaystyle
\colim_{\mathbb{N}}\holim_{S\in \mathcal{P}(J_+)\backslash\emptyset}FX_{n,S}=\colim_{\mathbb{N}}\lim\left(\prod_{S} \map_\cat{C}(K_S,FX_{n,S})\rightrightarrows \prod_{S\rightarrow T}\map_\cat{C}(K_s,FX_{n,T}) \right)\cong\\
\cong\displaystyle\lim\left(\prod_{S} \map_\cat{C}(K_S , \colim_{\mathbb{N}} FX_{n,S}) \rightrightarrows \prod_{S\rightarrow T}\map_\cat{C}(K_s, \colim_{\mathbb{N}} FX_{n,T})\right) = \holim_{S\in \mathcal{P}(J_+)\backslash\emptyset}\colim_{\mathbb{N}} FX_{n,S} 
\end{array}
\]
where the middle map is an isomorphism because sequential colimits commute with finite limits and with the functors $\map_\cat{C}(K_s,-)$, since each $K_S$ is finite. Now let $\overline{FX} \stackrel{\sim}{\surj} FX$ be a replacement of $FX$ by a sequence of diagrams such that for each $S \subset J_+$ the sequence $\overline{FX}_S$ is a sequence of $G_S$-cofibrations. There is a commutative diagram,
 \[\xymatrix{\displaystyle \colim_{\mathbb{N}}\holim_{S\in \mathcal{P}(J_+)\backslash\emptyset}FX_{n,S} \ar[d] \ar[r]^\cong& \displaystyle \holim_{S\in \mathcal{P}(J_+)\backslash\emptyset}\colim_{\mathbb{N}} FX_{n,S} \ar[d] & \displaystyle\holim_{S\in \mathcal{P}(J_+)\backslash\emptyset}\colim_{\mathbb{N}} \overline{FX}_{n,S} \ar[l]_\sim \ar[d]^\sim\\
\displaystyle \holim_{ S\in \mathcal{P}(J_+)\backslash\emptyset}F\colim_{\mathbb{N}}X_{n,S} \ar[r]^\sim & \displaystyle\holim_{ S\in \mathcal{P}(J_+)\backslash\emptyset}F\colim_{\mathbb{N}}FX_{n,S} & \displaystyle\holim_{ S\in \mathcal{P}(J_+)\backslash\emptyset}F\colim_{\mathbb{N}}\overline{FX}_{n,S} \ar[l]_\sim }\]
where the right hand vertical is a weak equivalence because $\colim_{\mathbb{N}}\overline{FX}_{n,S}$ is fibrant by an application of the small object argument in the cofibrantly generated model category $\cat{C}^G$ (see e.g. \cite[1.3.2]{specmod}). It follows that the left hand vertical map is a weak equivalence as desired.


\end{proof}

\bibliographystyle{amsalpha}
\bibliography{gdiag}

\end{document}

%% file: def.tex
\newtheoremstyle{theoremstyle}
  {10pt}      
  {5pt}       
  {\itshape}  
  {}          
  {\bfseries} 
  {:}         
  {.5em}      
  {}          

\newtheoremstyle{examplestyle}
  {10pt}      
  {5pt}       
  {}          
  {}          
  {\bfseries} 
  {:}         
  {.5em}      
  {}          

\theoremstyle{theoremstyle}
\newtheorem{theorem}{Theorem}[section]
\newtheorem*{theorem*}{Theorem}
\newtheorem{lemma}[theorem]{Lemma}
\newtheorem{proposition}[theorem]{Proposition}
\newtheorem*{proposition*}{Proposition}
\newtheorem{corollary}[theorem]{Corollary}
\newtheorem*{corollary*}{Corollary}

\theoremstyle{definition}
\newtheorem{example}[theorem]{Example}
\newtheorem*{example*}{Example}
\newtheorem{definition}[theorem]{Definition}
\newtheorem*{definition*}{Definition}
\newtheorem{remark}[theorem]{Remark}
\newtheorem{remark*}{Remark}

\newcommand{\too}{{\longrightarrow}}

\newcommand{\R}{{\mathbb{R}}}

\newcommand{\cat}{\mathscr}
\newcommand{\inj}{\hookrightarrow}
\newcommand{\surj}{\twoheadrightarrow}
\newcommand{\iso}{\stackrel{\cong}{\longrightarrow}}

\newcommand{\res}{{\mathrm{res}}}
\newcommand{\Sp}{{\mathrm{Sp}}}
\newcommand{\ev}{{\mathrm{ev}}}
\newcommand{\id}{{\mathrm{id}}}
\newcommand{\map}{{\mathrm{map}}}

\DeclareMathOperator*{\colim}{colim}
\DeclareMathOperator*{\hocolim}{hocolim}
\DeclareMathOperator*{\holim}{holim}
\DeclareMathOperator*{\ho}{ho}